\documentclass[11pt]{article}
\usepackage{amssymb}
\usepackage{amsmath,amsthm}
\usepackage{verbatim}
\usepackage{graphicx}
\usepackage{epstopdf}
\usepackage{epsf}
\usepackage{longtable}
\usepackage{booktabs}
\usepackage[usenames,dvipsnames,table]{xcolor}
\usepackage{geometry}
\usepackage[colorlinks=true, linkcolor=blue, filecolor=magenta, urlcolor=cyan, citecolor=gray, unicode]{hyperref}
\usepackage[flushmargin]{footmisc}
\usepackage{subcaption}

\usepackage[numbers]{natbib}
\usepackage{diagbox}

\usepackage{enumitem}
\setlist{itemsep=0pt, topsep=0pt}

\newtheorem{theorem}{Theorem}[section]
\newtheorem{lemma}[theorem]{Lemma}
\newtheorem{claim}[theorem]{Claim}

\newtheorem{fact}[theorem]{Fact}
\newtheorem{conjecture}[theorem]{Conjecture}
\newtheorem{corollary}[theorem]{Corollary}

\newtheorem{proposition}[theorem]{Proposition}
\newtheorem{observation}[theorem]{Observation}
\newtheorem{problem}[theorem]{Problem}
\newtheorem{example}[theorem]{Example}

\numberwithin{subcase}{case}

\numberwithin{property}{theorem}

\newcommand{\floor}[1]{\left\lfloor#1\right\rfloor}
\newcommand{\ceiling}[1]{\left\lceil#1\right\rceil}

\newcommand{\tbf}[1]{\textbf{#1}}

\newcommand{\diam}{\textrm{diam}}
\newcommand{\rad}{\textrm{rad}}

\newcommand{\cT}{{\mathcal T}}

\newcommand{\tp}{\mathrm{tp}}
\newcommand{\tc}{\mathrm{tc}}

\newcommand{\mc}{\mathrm{mc}}

\newcommand{\dc}{\mathrm{dc}}
\newcommand{\tdc}{\mathrm{tdc}}

\title{Generalizations and strengthenings of Ryser's conjecture}
\author{
Louis DeBiasio\thanks{Department of Mathematics, Miami University, \texttt{debiasld@miamioh.edu}. Research supported in part by Simons Foundation Collaboration Grant \#283194 and NSF grant DMS-1954170.}, ~
Yigal Kamel\thanks{Department of Mathematics, University of Illinois at Urbana-Champaign,  \texttt{ykamel2@illinois.edu}}, ~
Grace McCourt\thanks{Department of Mathematics, University of Illinois at Urbana-Champaign, \texttt{mccourt4@illinois.edu}}, ~
Hannah Sheats\thanks{Department of Mathematics, The Ohio State University, \texttt{sheats.6@osu.edu}}
}

\begin{document}

\maketitle

\begin{abstract}
Ryser's conjecture says that for every $r$-partite hypergraph $H$ with matching number $\nu(H)$, the vertex cover number is at most $(r-1)\nu(H)$.  This far-reaching generalization of K\"onig's theorem is only known to be true for $r\leq 3$, or when $\nu(H)=1$ and $r\leq 5$.  An equivalent formulation of Ryser's conjecture is that in every $r$-edge coloring of a graph $G$ with independence number $\alpha(G)$, there exists at most $(r-1)\alpha(G)$ monochromatic connected subgraphs which cover the vertex set of $G$.  

We make the case that this latter formulation of Ryser's conjecture naturally leads to a variety of stronger conjectures and generalizations to hypergraphs and multipartite graphs.  In regards to these generalizations and strengthenings, we survey the known results, improving upon some, and we introduce a collection of new problems and results.
\end{abstract}

\section{Introduction}

Let $H$ be a hypergraph.  We say that $H$ is \emph{$k$-uniform} if every edge of $H$ contains exactly $k$ vertices.  We say that $H$ is \emph{$r$-partite} if there exists a partition of $V(H)$ into sets $\{V_1, \dots, V_r\}$ such that for every edge $e$ of $H$, $|e\cap V_i|\leq 1$ for all $i\in [r]$; we use \emph{bipartite} to mean 2-partite.  Note that throughout the paper we won't assume that an $r$-partite hypergraph is necessarily $r$-uniform.  A \emph{matching} in H is a set of pairwise disjoint edges.  A \emph{vertex cover} of $H$ is a set of vertices $T$ such that each edge of $H$ contains a vertex from $T$.  We denote the size of a largest matching in $H$ by $\nu(H)$ and we denote the size of a minimum vertex cover of $H$ by $\tau(H)$.  Note that for every hypergraph $H$ we have $\nu(H)\leq \tau(H)$, since a minimum vertex cover must contain at least one vertex from each edge in a maximum matching.  

The following theorem of K\"onig from 1931, is one of the cornerstone results in graph theory. 

\begin{theorem}[K\"onig \cite{K}]\label{konig}
For every bipartite graph $H$, $\tau(H)\leq \nu(H)$.
\end{theorem}

In the 1970's, Ryser made the following conjecture which would generalize K\"onig's theorem to $r$-partite hypergraphs. 

\begin{conjecture}[Ryser (see \cite{H})]\label{con:ryser}
For every $r$-partite hypergraph $H$, $\tau(H)\leq (r-1)\nu(H)$.
\end{conjecture}

A hypergraph $H$ is \emph{intersecting} if every pair of edges has non-empty intersection; equivalently $H$ is intersecting if $\nu(H)=1$.  The most well-studied special case of Ryser's conjecture is that for every $r$-partite intersecting hypergraph $H$, $\tau(H)\leq r-1$.

Aside from the case $r=2$ which is K\"onig's theorem, Ryser's conjecture has only been verified in the following cases: $r=3$ by Aharoni \cite{A}, $r=4$ and $\nu=1$ by Tuza \cite{T1}, $r=5$ and $\nu=1$ by Tuza \cite{T1}.

Finally, we note that in \cite{H}, Ryser's conjecture was not orginally formulated in the way we have stated above.  The original, equivalent, formulation is as follows:
Let $r\geq 2$ and let $A$ be a $r$-dimensional $0,1$-matrix.  The \emph{term rank} of $A$, denoted $\nu(A)$, is the maximum number of 1's, such that no pair is in the same $(r-1)$-dimensional hyperplane.  The \emph{covering number} of $A$, denoted $\tau(A)$, is the minimum number of $(r-1)$-dimensional hyperplanes which contain all of the 1's of $A$.  In this language, Ryser's conjecture says that if $A$ is an $r$-dimensional $0,1$-matrix, then $\tau(A)\leq (r-1)\nu(A)$.

\subsection{Fractional versions of Ryser's conjecture}

Given a hypergraph $H=(V,E)$, a \emph{fractional matching} is a function $m:E\to [0,1]$ such that for all $v\in V$, $\sum_{e\ni v}m(e)\leq 1$, and a \emph{fractional vertex cover} is a function $t:V\to [0,1]$ such that for all $e\in E$, $\sum_{v\in e}t(v)\geq 1$. 
We let 
\[
\nu^*(H)=\max\left\{\sum_{e\in E}m(e): m \text{ is a fractional matching on } H\right\} 
\]
and 
\[
\tau^*(H)=\min\left\{\sum_{v\in V}t(v): t \text{ is a fractional vertex cover on } H\right\}.  
\]

All three fractional versions of Ryser's conjecture are known to be true (that is, replacing at least one of $\tau$ or $\nu$ with $\tau^*$ or $\nu^*$ respectively):

First, it is well-known consequence of the duality theorem in linear programming that $\tau^*(H)=\nu^*(H)$ for all hypergraphs $H$.  For all $r$-partite hypergraphs $H$, Lov\'asz \cite{Lov75} proved $\tau(H)\leq \frac{r}{2}\nu^*(H)$ and F\"uredi \cite{F81} proved $\tau^*(H)\leq (r-1)\nu(H)$.


\subsection{Duality}

We say that a hypergraph $H$ is \emph{connected} if for all $u,v\in V(H)$ there exists $e_1, \dots, e_k\in E(H)$ such that $u\in e_1$, $v\in e_k$ and $e_{i}\cap e_{i+1}\neq \emptyset$ for all $i\in [k-1]$.  
A \textit{component} of a hypergraph $H$ is a maximal connected subgraph of $H$.  Given a set of vertices $A\subseteq V(H)$, let $H[A]$ denote the subhypergraph of $H$ induced by $A$, and if $H[A]$ has no edges, we say that $A$ is an \emph{independent set}.  The size of a maximum independent set of vertices in $H$ is denoted by $\alpha(H)$.  An \emph{r-coloring} of the edges of a hypergraph $H$ is a function $c:E(H)\to [r]$; equivalently, a partition of $E(H)$ into (possibly empty) sets $\{E_1, \dots, E_r\}$. We say an edge $e \in E(G)$ is \emph{color i} if $c(e)=i$; equivalently, $e \in E_i$.  In an $r$-colored hypergraph $H$, a \emph{monochromatic cover} of $H$ is a set $\cT$ of monochromatic connected subgraphs of $H$ such that $V(H)=\cup_{T\in \mathcal{T}}V(T)$.  For a positive integer $t$,  a \emph{monochromatic $t$-cover} of $H$ is a monochromatic cover of order at most $t$.  Let $\tc_r(H)$ be the minimum integer $t$ such that in every $r$-coloring of the edges of $H$, there exists a monochromatic $t$-cover of $H$.  Note that since every connected subgraph contains a spanning tree, we can think of the connected subgraphs in a monochromatic cover as trees; this explains the notation ``$\tc$'' which stands for ``tree cover.''

In this language, the well known remark of Erd\H{o}s and Rado that a graph or its complement is connected (see \cite{BDV}), can be formulated as $\tc_2(K_n)=1$.  

Gy\'arf\'as \cite{Gy} noted that Ryser's conjecture is equivalent to the following statement about edge colored graphs.

\begin{conjecture}[Ryser]\label{ryserdual}
For every graph $G$ and every integer $r\geq 2$, $\tc_r(G)\leq (r-1)\alpha(G)$.
\end{conjecture}

So in particular, K\"onig's theorem can be reformulated as follows.

\begin{theorem}[K\"onig]\label{dualkonig}
For any graph $G$, $\tc_2(G)\leq \alpha(G)$.
\end{theorem}

To see why this equivalence holds, given an $r$-colored graph $G$, we let $H$ be a hypergraph where the vertex set is the set of monochromatic components in $G$ which is naturally partitioned into $r$ parts depending on the color of the component, and a set of vertices in $H$ forms an edge if the corresponding set of components has non-empty intersection in $G$ and is maximal with respect to this property.  One can see that an independent set of order $m$ in $G$ will correspond to a matching of order $m$ in $H$, and a monochromatic cover of $G$ will correspond to a vertex cover of $H$.  

On the other hand, given an $r$-partite hypergraph $H$ with vertex set partitioned as $\{V_1, \dots, V_r\}$, we let $G$ be a graph with $V(G)=E(H)$ and we put an edge of color $i$ between $e,f\in V(G)$ if and only if $e\cap f\cap V_i\neq \emptyset$.  Since edges from $H$ can intersect in more than one set, $G$ will be an $r$-colored multigraph (or a $r$-multicolored graph) in which every monochromatic component is a clique.  Note that a matching of order $m$ in $H$ will correspond to an independent set of order $m$ in $G$, and a vertex cover of $H$ will correspond to a monochromatic cover of $G$ (in which all of the monochromatic components are cliques).  

We have now seen that there are at least three equivalent ways of stating Ryser's conjecture.  For the remainder of the paper we focus on these two.

\begin{enumerate}
\item[(R1)] For every $r$-partite hypergraph  $H$, $\tau(H)\leq (r-1)\nu(H)$.

\item[(R2)] For every graph $G$, $\tc_r(G)\leq (r-1)\alpha(G)$.
\end{enumerate}

Now suppose we have two $r$-colored graphs $G$ and $G'$ on the same vertex set $V$ such that for all $i\in [r]$, the components of color $i$ in $G$ and the components of color $i$ in $G'$ give the same partition of $V$.  In the above discussion, we see that from $G$ and $G'$, we will derive the exact same $r$-partite hypergraph $H$. On the other hand, given an $r$-partite hypergraph $H$, we will only derive a single $r$-colored graph $G$.

This brings us to one of the main themes of this paper.  It is certainly true that (R1) feels most natural in that it directly generalizes the well known K\"onig's theorem.  
However, by stating Ryser's conjecture in terms of (R2), we can access a whole host of interesting strengthenings and generalizations which have no analogue in the $r$-partite hypergraph setting.  For instance we can ask if there is a monochromatic cover $\cT$ in which every subgraph in $\cT$ has small diameter, or whether $\cT$ can be chosen so that subgraphs in $\cT$ are pairwise disjoint (i.e. $\cT$ forms a partition rather than just a cover).  Furthermore, we can generalize the problem to settings like complete multipartite graphs and hypergraphs.  In Section \ref{sec:further} we give many more such examples.

Finally, we make note of the following trivial upper bound on $\tc_r(G)$ in the (R2) language.

\begin{fact}\label{fact:trivial}
Let $r\geq 2$.  For all graphs $G$, $\tc_r(G)\leq r\alpha(G)$.
\end{fact}

\subsection{Lower bounds}

A \emph{projective plane of order $q$} is a $(q+1)$-uniform hypergraph on $q^2+q+1$ vertices and $q^2+q+1$ edges such that each pair of vertices is contained in exactly one edge.  A \emph{truncated projective plane of order $q$} is a $(q+1)$-uniform hypergraph on $q^2+q$ vertices and $q^2$ edges obtained by deleting one vertex $v$ from a projective plane of order $q$ and removing the $q+1$ edges which contained $v$.
An \emph{affine plane of order $q$} is a $q$-uniform hypergraph on $q^2$ vertices and $q^2+q$ edges obtained by deleting one edge $e$ from a projective plane of order $q$ and removing the $q+1$ vertices which are contained in $e$.
Note that truncated projective planes and affine planes are duals of each other in the geometric sense where the roles of lines and points are switched.  

It is well known that a projective plane of order $q$ exists whenever $q$ is a prime power (and it is unknown whether there exists a projective plane of non-prime power order).  Also it is clear that a truncated projective plane of order $q$ and an affine plane of order $q$ exist if and only if a projective plane of order $q$ exists.  

A truncated projective plane $H$ of order $r-1$ is an intersecting $r$-uniform hypergraph with vertex cover number $r-1$ and if we take $\nu$ vertex disjoint copies of $H$, we have an $r$-uniform hypergraph with matching number $\nu$ and vertex cover number $(r-1)\nu$.  Thus Ryser's conjecture is tight for a given value of $r$ whenever a truncated projective plane $H$ of order $r-1$ exists.  

An affine plane $H$ of order $r-1$ is an $(r-1)$-uniform hypergraph with edge chromatic number $r$ and edge cover number $r-1$.  From $\alpha$ vertex disjoint affine planes of order $r-1$, we can create an $r$-colored graph $G$ with independence number $\alpha$ such that $\tc_r(G)=(r-1)\alpha$.

So we have the following fact.

\begin{fact}
Let $r\geq 2$ and $\alpha\geq 1$ be integers. If there exists an affine plane of order $r-1$, then for all $n\geq (r-1)^2\alpha$ there exists a graph $G$ on $n$ vertices with $\alpha(G)=\alpha$ such that $\tc_r(G)\geq (r-1)\alpha$.
\end{fact}

Finding matching lower bounds when affine plane of order $r-1$ does not exist is an active area of research with some interesting recent results  (\cite{AP}, \cite{ABPS}, \cite{ABW}, \cite{HaxS}); however, it is still unknown whether for all $r\geq 2$ and $\alpha\geq 1$, there exists a graph $G$ with $\alpha(G)=\alpha$ such that $tc_r(G)\geq (r-1)\alpha$.  The best general result is due to Haxell and Scott \cite{HaxS} who show that for all $r\geq 5$, there exists a graph $G$ such that $\tc_r(G)\geq (r-4)\alpha(G)$.

Finally, we note that most efforts to improve the lower bound have focused on the case $\alpha(G)=1$ because if one can prove that $\tc_r(K)\geq r-1$ for a complete graph $K$, then by taking $\alpha$ disjoint copies of $K$, we obtain a graph $G$ such that $\tc_r(G)\geq (r-1)\alpha(G)$.  It was shown in \cite{HNS} that for $r=3$, this is essentially the only such example.  However, it was shown in \cite{Abu} that for $r=4$, there is an example which is different than two disjoint 4-colored complete graphs and a more general example was given in \cite{BP}.

\subsection{Large monochromatic components}

We now briefly discuss the related problem of finding large monochromatic components in $r$-colored graphs.

\begin{theorem}[F\"uredi \cite{F81} (see Gy\'arf\'as \cite{GySurv1})]\label{thm:fur}
In every $r$-coloring of the edges of a graph $G$ with $n$ vertices, there exists a monochromatic component of order at least $\frac{n}{(r-1)\alpha(G)}$.

In the dual language, in every $r$-partite hypergraph $H$ with $n$ edges, there exists a vertex of degree at least $\frac{n}{(r-1)\nu(H)}$.
\end{theorem}

Note that if true, Ryser's conjecture implies Theorem \ref{thm:fur}.

\subsection{Notation}

Given a positive integer $k$, we let $[k]=\{1,2,\dots, k\}$.

Let $G$ be a graph. For sets $A,B\subseteq V(G)$, an $A,B$-edge is an edge with one endpoint in $A$ and the other in $B$.  If $A=\{v\}$, we write $v,B$-edge instead of $\{v\}, B$-edge.  We write $\delta(A,B)$ for $\min\{|N(v)\cap B|:v\in A\}$.  If $A$ and $B$ are disjoint, we let $[A,B]$ be the bipartite graph induced by the sets $A$ and $B$.  Given $k\geq 2$, we let $\mathcal{K}_k$ be the family of all complete $k$-partite graphs.

Given a set $S$ we say that $G$ is $S$-colored if the edges of $G$ are colored from the set $S$.  Given an integer $r$, we say that $G$ is $r$-colored if the edges of $G$ are colored with $r$ colors (unless otherwise stated, the set of colors will be $[r]$).  Given an $r$-coloring of $G$, say $c:E(G)\to [r]$, and a set $S\subseteq [r]$, we let $G_S$ be the $S$-colored graph on $V(G)$ with $E(G)=\{e\in E(G): c(e)\in S\}$.  If $S=\{i\}$ or $S=\{i,j\}$, we simply write $G_i$ or $G_{i,j}$ respectively.

\section{Overview of the paper}\label{sec:overview}

In this section we give a detailed overview of the results in the paper. In addition, we discuss a variety of other generalizations and strengthenings of Ryser's conjecture in Section \ref{sec:further} and we collect some observations about a hypothetical minimal counterexample to Ryser's conjecture in Appendix A.

\subsection{Monochromatic covers with restrictions on the colors}

We begin with a few conjectures which can be stated both in terms of intersecting $r$-partite hypergraphs (R1) and in terms of $r$-colored complete graphs (R2).

The results mentioned here are proved in Section \ref{sec:tuza} and Section \ref{sec:multi}.

\begin{conjecture}\label{con:SnotS}
For all integers $r\geq 2$, all $S\subseteq [r]$, and all $r$-colorings of a complete graph $K$, there exists a monochromatic $(r-1)$-cover consisting entirely of subgraphs having a color in $S$ or consisting entirely of subgraphs having a color in $[r]\setminus S$.
\end{conjecture}

In the dual (R1) language, Conjecture \ref{con:SnotS} says that for every $r$-partite intersecting hypergraph with vertex partition $\{V_1, \dots, V_r\}$ and every $S\subseteq [r]$, there is a vertex cover of order $r-1$ which is contained in $\{V_i:i\in S\}$ or contained in $\{V_i: i\in [r]\setminus S\}$. 

We prove Conjecture \ref{con:SnotS} for $r\leq 4$.  In the process of doing so, we formulate three other conjectures (all of which imply the $\alpha=1$ case of Ryser's conjecture). 

\begin{conjecture}\label{con:rpart}
For all integers $r\geq 2$ and all $K\in \mathcal{K}_r$, $\tc_{r-1}(K)\leq r-1$. In particular, this implies that for every $r$-coloring of a complete graph $K$ and every color $i\in [r]$, either there is a monochromatic $(r-1)$-cover consisting entirely of subgraphs of color $i$, or entirely of subgraphs which don't have color $i$. 
\end{conjecture}

In the dual (R1) language, Conjecture \ref{con:rpart} says that for every $r$-partite intersecting hypergraph, if some part $V_i$ has at least $r$ vertices, then there is a vertex cover of order at most $r-1$ which uses no vertices from $V_i$. 

Note that Conjecture \ref{con:rpart} implies the $\alpha=1$ case of Ryser's conjecture, but we will actually prove the following stronger conjecture for $r\leq 4$.

\begin{conjecture}\label{con:rpart2}
For all integers $r\geq 3$ and all $K\in \mathcal{K}_{r-1}$, $\tc_{r-1}(K)\leq r-1$.  In particular, this implies that for every $r$-coloring of a complete graph $K$ and every color $i\in [r]$, either there is a monochromatic $(r-2)$-cover consisting entirely of subgraphs of color $i$, or a monochromatic $(r-1)$-cover consisting entirely of subgraphs which don't have color $i$.
\end{conjecture}

A special case of Conjecture \ref{con:SnotS} obtained by setting $|S|=\ceiling{r/2}$ is the following.

\begin{conjecture}\label{con:r/2}
For all integers $r\geq 2$, in every $r$-coloring of a complete graph $K$ there exists a monochromatic $(r-1)$-cover such that the monochromatic subgraphs have at most $\ceiling{r/2}$ different colors.
\end{conjecture}

In the dual (R1) language, in every $r$-partite intersecting hypergraph, there is a vertex cover of order at most $r-1$ which is made up of vertices from at most $\ceiling{r/2}$ parts.

We give an example to show that $\ceiling{r/2}$ cannot be reduced in Conjecture \ref{con:r/2}.

\begin{example}\label{ex:r/2}
For all $r\geq 3$ and $n\geq r\binom{r}{\floor{r/2}+1}$, there exists an $r$-coloring of $K_n$, such that every monochromatic cover of $K_n$ with at most $r-1$ components consists of components of at least $\ceiling{r/2}$ different colors.
\end{example}

\begin{proof}
Let $\mathcal{A} = \binom{[r]}{\floor{r/2}+1}$ (that is the family of subsets of $[r]$ with just over half the elements).  
Now let $V$ be a set of at least $r\binom{r}{\floor{r/2}+1}$ vertices and let $\{V_X: X\in \mathcal{A}\}$ be a partition of $V$ into sets of order at least $r$ which are indexed by the elements in $\mathcal{A}$.  For all $u\in V_X$, $v\in V_Y$, let $uv$ be an edge of some arbitrarily chosen color $i\in X\cap Y$ (which is possible since $X\cap Y\neq \emptyset$ for all $X,Y\in \mathcal{A}$).  We now have an $r$-colored complete graph $K$ on vertex set $V$.

\begin{figure}[ht]
\centering
\includegraphics{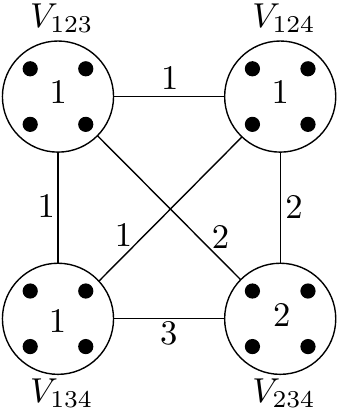}
\caption{Example \ref{ex:r/2} in the case $r=4$.}\label{fig:r/2}
\end{figure}

Suppose for contradiction that there exists $S \subseteq [r]$ with $|S|=\ceiling{r/2}-1$ and that $K$ has a monochromatic $(r-1)$-cover $\mathcal{T}$ such that all of the subgraphs in $\cT$ have a color in $S$.  Since $r-(\ceiling{r/2}-1)=\floor{r/2}+1$, there exists $X\in \mathcal{A}$ such that $X = [r] \setminus S$.  This means that there are no edges having a color from $S$ which are incident with a vertex in $V_X$.  Since there are at most $r-1$ components in $\mathcal{T}$ all having colors from $S$ and there are at least $r$ vertices in $V_X$, this contradicts the fact that $\cT$ was the desired monochromatic cover.
\end{proof}

\subsection{Monochromatic covers with subgraphs of bounded diameter}

Now we move on to some results which can only be stated in terms of $r$-colored graphs (R2).

The results mentioned here are proved in Section \ref{sec:diameter}.

Let $G$ be a graph. For vertices $u,v\in V(G)$, let $d(u,v)$ denote the length of the shortest $u,v$-path in $G$. If there is no $u,v$-path, we write $d(u,v) = \infty$.  The \emph{diameter} of $G$, denoted $\diam(G)$, is the smallest integer $d$ such that $d(u,v)\leq d$ for all $u,v\in V(G)$.  If $G$ is disconnected, we say $\diam(G)=\infty$.  The \emph{radius} of $G$, denoted $\rad(G)$, is the smallest integer $r$ such that there exists $u\in V(G)$ such that $d(u,v)\leq r$ for all $v\in V(G)$.

It is well known that a graph or its complement has diameter at most 3; in other words, in every 2-coloring of a complete graph $K$, there is a spanning monochromatic subgraph of diameter at most 3.

Mili\'cevi\'c conjectured an extension of this to $r$-colors which strengthens the $\alpha=1$ case of Ryser's conjecture. 

\begin{conjecture}[Mili\'cevi\'c \cite{M2}]\label{con:ryserdiametercomplete}
For all $r\geq 2$, there exists $d=d(r)$ such that in every $r$-coloring of a complete graph $K$, there exists a monochromatic $(r-1)$-cover consisting of subgraphs of diameter at most $d$.
\end{conjecture}

Mili\'cevi\'c proved that in every $3$-coloring of a complete graph $K$, there is a monochromatic $2$-cover consisting of subgraphs of diameter at most 8 \cite{M1}, and in every $4$-coloring of a complete graph $K$, there is a monochromatic $3$-cover consisting of subgraphs of diameter at most 80 \cite{M2}.

For the case $r=3$, we improve the upper bound on the diameter from 8 to 4.  In the case $r=4$, we improve the upper bound on the diameter from 80 to 6 while at the same time giving a significantly simpler proof.  

\begin{theorem}\label{thm:diam}~
\begin{enumerate}
\item In every $3$-coloring of a complete graph $K$, there is a monochromatic $2$-cover consisting of trees of diameter at most 4.

\item In every $4$-coloring of a complete graph $K$, there is a monochromatic $3$-cover consisting of subgraphs of diameter at most 6.
\end{enumerate}
\end{theorem}

We also conjecture a generalization of Ryser's conjecture for graphs with arbitrary independence number.  

\begin{conjecture}\label{con:ryserdiameter}
For all $\alpha\geq 1$, there exists $d=d(\alpha)$ such that for all $r\geq 2$, if $G$ is a graph with $\alpha(G)=\alpha$, then in every $r$-coloring of $G$, there exists a monochromatic $(r-1)\alpha$-cover consisting of subgraphs of diameter at most $d$.  
\end{conjecture}

Note that in Conjecture \ref{con:ryserdiametercomplete}, it is conjectured that $d$ depends on $r$.  We speculate that it is even possible to choose a $d$ which is independent of both $r$ and $\alpha$, but we have no concrete evidence to support this.  

We prove Conjecture \ref{con:ryserdiameter} for $\alpha=2=r$.

\begin{theorem}\label{thm:diamalpha}
Let $G$ be a graph with $\alpha(G)=2$.  In every 2-coloring of $G$ there is a monochromatic $2$-cover consisting of subgraphs of diameter at most 6.
\end{theorem}

Gy\'arf\'as raised the following problem which would strengthen Theorem \ref{thm:fur} in the case $\alpha=1$.

\begin{problem}[Gy\'arf\'as \cite{GySurv1}]\label{prob:Gydiam}
In every $r$-coloring of the edges of $K_n$, there exists a monochromatic subgraph of diameter at most 3 on at least $\frac{n}{r-1}$ vertices.  Perhaps the subgraph can even be chosen to be a tree of diameter at most 3 (which is necessarily a double star).
\end{problem}

Improving on earlier results of Mubayi \cite{Mub} and Ruszink\'o \cite{Rus}, Letzter almost solved Problem \ref{prob:Gydiam}.

\begin{theorem}[Letzter \cite{L}]
In every $r$-coloring of the edges of $K_n$, there exists a monochromatic tree of diameter at most 4 (in fact, the tree can be chosen to be a triple star) on at least $\frac{n}{r-1}$ vertices.
\end{theorem}

Note that Theorem \ref{thm:diam}(i) implies Letzter's result in the case $r=3$ (except we can't guarantee that both of the trees are triple stars).

\subsection{Monochromatic covers of complete multipartite graphs}

The results mentioned here are proved in Section \ref{sec:diameter} and  Section \ref{sec:multi}.

Gy\'arf\'as and Lehel made the following conjecture which would be tight if true.

\begin{conjecture}[Gy\'arf\'as, Lehel \cite{Gy}]\label{con:GL}
For all $K\in \mathcal{K}_2$, $\tc_r(K)\leq 2r-2$.  
\end{conjecture}

Chen, Fujita, Gy\'arf\'as, Lehel, and T\'oth \cite{CFGLT} proved this for $r\leq 5$.  Also note that for all $K\in \mathcal{K}_2$, a trivial upper bound is $\tc_r(K)\leq 2r-1$ (by considering a pair of vertices $u,v$ on opposite sides of the bipartition and the union of the monochromatic components containing $u$ and $v$).

We now mention the following generalization of Conjecture \ref{con:GL} for which we don't even have a conjecture.  The first interesting test case (outside the scope of Conjectures \ref{con:rpart} and \ref{con:rpart2}) is $k=3$ and $r=4$.

\begin{problem}\label{prob:kpart}
Let $k$ and $r$ be integers with $k,r\geq 2$.  Determine an upper bound on $\tc_r(K)$ which holds for all $K\in \mathcal{K}_k$.
\end{problem}

We also make the following strengthening of Conjecture \ref{con:GL} and prove it for $r=2$ and $r=3$ (the $r=2$ case is an improvement of a result of Mili\'cevi\'c \cite{M1}). 

\begin{conjecture}\label{con:bipartitediameter}
There exists $d$ such that for all $r\geq 2$, if $K\in \mathcal{K}_2$, then in every $r$-coloring of $K$, there exists a monochromatic $(2r-2)$-cover consisting of subgraphs of diameter at most $d$.
\end{conjecture}

\begin{theorem}Let $K\in \mathcal{K}_2$.
\begin{enumerate}
\item In every $2$-coloring of $K$, there is a monochromatic $2$-cover consisting of trees of diameter at most 4.

\item In every $3$-coloring of $K$, there is a monochromatic $4$-cover consisting of subgraphs of diameter at most 6.
\end{enumerate}
\end{theorem}

\subsection{Partitioning into monochromatic connected subgraphs}

The results mentioned here are proved in Section \ref{sec:partition}.

For positive integers $t$ and $r$, a \emph{monochromatic $t$-partition} of an $r$-colored hypergraph $H$ is a monochromatic $t$-cover $\cT$ of $H$ such that $V(T)\cap V(T')=\emptyset$ for all $T,T'\in \cT$.  Let $\tp_r(H)$ be the minimum integer $t$ such that in every $r$-coloring of the edges of $H$, there exists a monochromatic $t$-partition of $H$.  

Erd\H{o}s, Gy\'arf\'as, and Pyber made the following conjecture and proved it for $r=3$.

\begin{conjecture}[Erd\H{o}s, Gy\'arf\'as, Pyber \cite{EGP}]\label{con:EGP}
For all $r\geq 2$ and all finite complete graphs $K$, $\tp_r(K)\leq r-1$.
\end{conjecture}

Later Fujita, Furuya, Gy\'arf\'as, and T\'oth made the following conjecture and proved it for $r=2$.  Note that this is a significant strenghtening of Ryser's conjecture.

\begin{conjecture}[Fujita, Furuya, Gy\'arf\'as, T\'oth \cite{FFGT1}]\label{con:FFGT}
For all $r\geq 2$ and all finite graphs $G$, $\tp_r(G)\leq (r-1)\alpha(G)$.
\end{conjecture}

Haxell and Kohayakawa \cite{HK} proved $\tp_r(K_n)\leq r$ for sufficiently large $n$ (in fact, they proved that there is a monochromatic $r$-partition consisting of trees of radius at most 2).  The bound on $n$ was improved in \cite{BD}.  In Section \ref{sec:partition}, we discuss why the bound on $n$ essentially cannot be improved any further using this approach, and in the process find an interesting connection to a different problem.  

We also raised the question of determining an upper bound on $\tp_r(K)$ for $K\in \mathcal{K}_k$.  Surprisingly we found that in contrast to the cover version of the problem, no such upper bound (which depends only on $k$) is possible.

\begin{theorem}\label{thm:partmulti}
For all $k\geq 2$ and all functions $f:\mathbb{Z}^+\to \mathbb{R}$ there exists $K\in \mathcal{K}_k$ such that $\tp_2(K)>f(k)$.
\end{theorem}

\subsection{Monochromatic covers of hypergraphs}

The results mentioned here are proved in Section \ref{sec:hypergraph}.

Denote the complete $r$-uniform hypergraph on $n$ vertices by $K_n^r$.  Again, the well known remark of Erd\H{o}s and Rado, which says $\tc_2(K_n^2)=1$, was generalized by Gy\'arf\'as \cite{Gy} who proved that for all $r\geq 2$, $\tc_r(K_n^r)=1$.

Kir\'aly \cite{K} proved that for all $k\geq 3$, $\tc_r(K_n^k)= \ceiling{r/k}$.  In the dual (R1) language, this means that for $k\geq 3$ if we have an $r$-partite hypergraph $H$ in which every set of $k$ edges has a common non-empty intersection, then $\tau(H)\leq \ceiling{r/k}$.

We begin the study of a much more general setting in which we allow for different notions of connectivity in hypergraphs.  Given an $k$-uniform hypergraph $H$, say that $H$ is \emph{tightly connected} if for every pair of vertices $u,v\in V(H)$, there exists edges $e_1, \dots, e_p\in E(H)$ such that $u\in e_1$, $v\in e_p$, and $|e_i\cap e_{i+1}|=k-1$ for all $i\in [p-1]$.  We prove a generalization of Kir\'aly's theorem, but we delay the statement until Section \ref{sec:hypergraph}.

One of our main (and easiest to state) conjectures in this setting is the following strengthening of Gy\'arf\'as' result, which we prove for $r=3$.

\begin{conjecture}\label{con:tight}
For all $r\geq 3$, in every $r$-coloring of $K_n^r$, there exists a monochromatic tightly connected subgraph which covers $V(K)$.
\end{conjecture}

\begin{problem}\label{prob:tight}
Let $r,k\geq 2$ be integers.  Given an arbitary $r$-coloring of $K_n^k$, determine an upper bound on the number of monochromatic tightly connected subgraphs needed to cover $V(K_n^k)$.
\end{problem}

The following tables given in Table \ref{tab:tables} recap what is known about the various generalizations and strengthenings of Ryser's conjecture discussed so far, using green to indicate previously known results and yellow to indicate new or improved results that we will show in the following sections.

\begin{table}[ht!]
\begin{subtable}{.5\linewidth}
\centering
\begin{tabular}{|l|l|l|l|l|l|}
\firsthline
\backslashbox{\textbf{$\alpha$}}{$r$} & \textbf{2} & \textbf{3} & \textbf{4} & \textbf{5} & \textbf{6} \\ \hline
\textbf{1} & \cellcolor{green}1 & \cellcolor{green}2 & \cellcolor{green}3 & \cellcolor{green}4 & 5 \\ \hline
\textbf{2} & \cellcolor{green}2 & \cellcolor{green}4 & 6 & 8 & 10 \\ \hline
\textbf{3} & \cellcolor{green}3 & \cellcolor{green}6 & 9 & 12 & 15 \\ \hline
\textbf{4} & \cellcolor{green}4 & \cellcolor{green}8 & 12 & 16 & 20 \\ \hline
\textbf{5} & \cellcolor{green}$\downarrow$ & \cellcolor{green}$\downarrow$ & 15 & 20 & 25 \\ \hline
\end{tabular}
\captionof{table}{Conjecture \ref{ryserdual}}\label{tab:1} 
\end{subtable}
~
\begin{subtable}{.5\linewidth}
\centering
\begin{tabular}{|l|l|l|l|l|l|}
\firsthline
$r$ & \cellcolor{yellow}\textbf{2} & \cellcolor{yellow}\textbf{3} & \cellcolor{yellow}\textbf{4} & \textbf{5} & \textbf{6} \\ \hline
\end{tabular}
\captionof{table}{Conjecture \ref{con:SnotS}}\label{tab:2} 
~
\begin{tabular}{|l|l|l|l|l|l|}
\hline
$r$ & \cellcolor{yellow}\textbf{2} & \cellcolor{yellow}\textbf{3} & \cellcolor{yellow}\textbf{4} & \textbf{5} & \textbf{6} \\ \hline
\end{tabular}
\captionof{table}{Conjecture \ref{con:rpart2}}\label{tab:3} 
~
\begin{tabular}{|l|l|l|l|l|l|}
\hline
$r$ & \cellcolor{yellow}\textbf{2} & \cellcolor{yellow}\textbf{3} & \cellcolor{yellow}\textbf{4} & \cellcolor{yellow}\textbf{5} & \textbf{6} \\ \hline
\end{tabular}
\captionof{table}{Conjecture \ref{con:r/2}}\label{tab:4} 
\end{subtable}
~
\begin{subtable}[t]{.5\linewidth}
\centering
\begin{tabular}{|l|l|l|l|l|l|}
\firsthline
\backslashbox{\textbf{$\alpha$}}{$r$} & \textbf{2} & \textbf{3} & \textbf{4} & \textbf{5} & \textbf{6} \\ \hline
\textbf{1} & \cellcolor{green}1 & \cellcolor{green}2 & 3 & 4 & 5 \\ \hline
\textbf{2} & \cellcolor{green}2 & 4 & 6 & 8 & 10 \\ \hline
\textbf{3} & \cellcolor{green}3 & 6 & 9 & 12 & 15 \\ \hline
\textbf{4} & \cellcolor{green}4 & 8 & 12 & 16 & 20 \\ \hline
\textbf{5} & \cellcolor{green}$\downarrow$ & 10 & 15 & 20 & 25 \\ \hline
\end{tabular}
\captionof{table}{Conjecture \ref{con:FFGT}}\label{tab:5} 
\end{subtable}
~
\begin{subtable}[t]{.5\linewidth}
\centering
\begin{tabular}{|l|l|l|l|l|l|}
\firsthline
\backslashbox{\textbf{$\alpha$}}{$r$} & \textbf{2} & \textbf{3} & \textbf{4} & \textbf{5} & \textbf{6} \\ \hline
\textbf{1} & \cellcolor{green}1 & \cellcolor{yellow}2 & \cellcolor{yellow}3 & 4 & 5 \\ \hline
\textbf{2} & \cellcolor{yellow}2 & 4 & 6 & 8 & 10 \\ \hline
\textbf{3} & 3 & 6 & 9 & 12 & 15 \\ \hline
\textbf{4} & 4 & 8 & 12 & 16 & 20 \\ \hline
\textbf{5} & $\downarrow$ & 10 & 15 & 20 & 25 \\ \hline
\end{tabular}
\captionof{table}{Conjecture \ref{con:ryserdiameter}}\label{tab:6} 
\end{subtable}
\begin{subtable}[t]{.5\linewidth}
\centering
\begin{tabular}{|l|l|l|l|l|l|}
\firsthline
\backslashbox{\textbf{$k$}}{$r$} & \textbf{2} & \textbf{3} & \textbf{4} & \textbf{5} & \textbf{6} \\ \hline
\textbf{2} & \cellcolor{green}2 & \cellcolor{green}4 & \cellcolor{green}6 & \cellcolor{green}8 & 10 \\ \hline
\textbf{3} & \cellcolor{green}2 & \cellcolor{yellow}3 &  &  &  \\ \hline
\textbf{4} & \cellcolor{green}2 & \cellcolor{yellow}3 &  &  &  \\ \hline
\textbf{5} & \cellcolor{green}$\downarrow$ & \cellcolor{yellow}$\downarrow$ &  &  &  \\ \hline
\end{tabular}
\captionof{table}{Conjecture \ref{con:GL}, Problem \ref{prob:kpart}}\label{tab:7} 
\end{subtable}
~
\begin{subtable}[t]{.5\linewidth}
\centering
\begin{tabular}{|l|l|l|l|l|l|}
\firsthline
\backslashbox{\textbf{$k$}}{$r$} & \textbf{2} & \textbf{3} & \textbf{4} & \textbf{5} & \textbf{6} \\ \hline
\textbf{2} & \cellcolor{green}2 & \cellcolor{yellow}4 & 6 & 8 & 10 \\ \hline
\textbf{3} & 2 & 3 &  &  &  \\ \hline
\textbf{4} & 2 & 3 &  &  &  \\ \hline
\textbf{5} & $\downarrow$ & $\downarrow$ &  &  &  \\ \hline
\end{tabular}
\captionof{table}{Conjecture \ref{con:bipartitediameter}}\label{tab:8} 
\end{subtable}
~
\begin{subtable}[t]{.5\linewidth}
\centering
\begin{tabular}{|l|l|l|l|l|l|}
\firsthline
\backslashbox{\textbf{$k$}}{$r$} & \textbf{2} & \textbf{3} & \textbf{4} & \textbf{5} & \textbf{6} \\ \hline
\textbf{2} & \cellcolor{green}1 & \cellcolor{green}2 & \cellcolor{green}3 & \cellcolor{green}4 & 5 \\ \hline
\textbf{3} & \cellcolor{green}1 & \cellcolor{yellow}1 &  &  &  \\ \hline
\textbf{4} & \cellcolor{green}1 & \cellcolor{yellow}1 & 1 &  &  \\ \hline
\textbf{5} & \cellcolor{green}$\downarrow$ & \cellcolor{yellow}$\downarrow$ &  & 1 &  \\ \hline
\end{tabular}
\captionof{table}{Conjecture \ref{con:tight}, Problem \ref{prob:tight}}\label{tab:9} 
\end{subtable}
\caption{A recap of the results discussed above.  Results known before this paper are highlighted in green.  New or improved results from this paper are highlighted in yellow.}\label{tab:tables}
\end{table}

\section{Tuza's proofs/Monochromatic covers with restrictions on the colors}\label{sec:tuza}

Let $G$ be a graph with $\alpha=\alpha(G)$.  Tuza proved that $\tc_r(G)\leq (r-1)\alpha$ in the cases when $$(r, \alpha)\in \{(3,1), (3,2), (3,3), (3,4), (4,1), (5,1)\}.$$  
The proof for  $(r, \alpha)\in \{(3,1), (3,2), (3,3), (3,4)\}$ can be found in \cite{T1} and \cite{T2}. The proof for $(r, \alpha)=(4,1)$ can be found in \cite{T1} and \cite{T3}. The proof for $(r, \alpha)=(5,1)$ can be found in \cite{T1}. Finally, the case $(r,\alpha)=(4,2)$ is claimed in \cite{T1} and \cite{T2}, but no proof is given.  Note that the cases $(r, \alpha)\in \{(3,1), (3,2), (3,3), (3,4)\}$ are superseded by Aharoni's theorem \cite{A}, but Tuza's proof may still be of some interest because of its elementary nature. 

In all cases, Tuza's proofs are given in the dual (R1) language of vertex covers of $r$-partite hypergraphs.  The objective of this section is to both reprove all of these results in the language of monochromatic covers of edge colored graphs and do so in such a way that we can use these results to prove Conjecture \ref{con:r/2} for $r\leq 5$ which in turn, together with the results in Section \ref{sec:multi}, allow us to prove Conjecture \ref{con:SnotS} for $r\leq 4$.  Also, since the $r=5$ case is unpublished, we feel that this may be of some benefit to others who would like to understand Tuza's proof of this case.  One of the original goals of this project was to explore the possibility of extending Tuza's methods to prove the case $(r, \alpha)=(6,1)$.  While we were unsuccessful in this goal, we were able to classify the (many) special cases which would need to be dealt with in order to prove such a result.  More specifically, when $(r, \alpha)=(5,1)$, Tuza's proof goes by making some general observations which, out of 37 possible cases, leaves two special cases each of which can be dealt with in an ad-hoc manner.  In trying to extend this to the case $(r, \alpha)=(6,1)$, we make analogous observations which, out of 560 possible cases, leaves 173 special cases (most of which do not seem to have an analogously easy ad-hoc proof).

We will prove the following.

\begin{theorem} \label{thm:r3-5}
Let $r\in \{3,4,5\}$ and $S\subseteq [r]$ with $S=|2|$.  In every $r$-coloring of a complete graph $K$, there exists a monochromatic $(r-1)$-cover $\mathcal{H}$ in which every subgraph in $\mathcal{H}$ has a color in $S$ or every subgraph in $\mathcal{H}$ has a color in $[r]\setminus S$.  In particular, $\tc_{r}(K) \leq r-1$. 
\end{theorem}

%
%

We begin with some general observations. The \emph{closure} of a graph $G$ with respect to a given coloring is a multigraph $\hat{G}$ on $V(G)$ with edge set defined as follows: there is an edge of color $i$ between $u$ and $v$ in $\hat{G}$ if and only if there is a path of color $i$ between $u$ and $v$ in $G$.

Let the edges of a graph $G$ be $r$-colored. Take the closure of $G$ with respect to this coloring. Note that $\tc_r(G) = \tc_r(\hat{G})$, since given a monochromatic cover of $\hat{G}$, the corresponding monochromatic components of $G$ form a monochromatic cover.

\begin{observation}\label{obs:closure}
In proving an upper bound on $\tc_r(G)$ we will instead prove an upper bound on $\tc_r(\hat{G})$; that is, we will assume that every monochromatic component in the $r$-edge (multi)coloring of $G$ is a clique.
\end{observation}

Let $G_{i,j}$ be the subgraph of $G$ induced by the edges of colors $i$ and $j$. By Theorem \ref{dualkonig},
\begin{equation}
\tc_r(G) \leq \tc_2(G_{i,j}) \leq \alpha(G_{i,j}). 
\end{equation}

Thus we have the following useful observation.

\begin{observation}\label{obs:tuza}
If there exist distinct colors $i,j \in [r]$ such that $\alpha(G_{i,j}) \leq (r - 1)\alpha(G)$, then $\tc_r(G)\leq (r-1)\alpha(G)$.
\end{observation}

\subsection{\texorpdfstring{$r=3$}{r=3}, \texorpdfstring{$\alpha=1$}{a=1}}

\begin{proof}[Proof of Theorem \ref{thm:r3-5} when $r=3$]
Let $S\subseteq [3]$ with $|S|=2$ and without loss of generality, suppose $S=\{2,3\}$.  If $\alpha(G_{2,3})\leq 2$, then we are done by Observation \ref{obs:tuza}; so suppose $\alpha(G_{2,3})\geq 3$ and let $X=\{x_1, x_2, x_3\}$ be an independent set in $G_{2,3}$.  This means every edge in $X$ has color 1.  Also since $X$ is independent in $G_{2,3}$, then by Observation \ref{obs:closure}, every vertex sends at most one edge of color 2 and at most one edge of color 3 to $X$.  Thus every vertex sends an edge of color 1 to $X$ and thus there is monochromatic cover consisting of a single component of color 1.
\end{proof}

\subsection{\texorpdfstring{$r=4$}{r=4}, \texorpdfstring{$\alpha=1$}{a=1}}

\begin{proof}[Proof of Theorem \ref{thm:r3-5} when $r=4$]
Let $S\subseteq [4]$ with $|S|=2$ and without loss of generality, suppose $S=\{3,4\}$.  If $\alpha(G_{3,4}) \leq 3$, then we are done by Observation \ref{obs:tuza}; so suppose $\alpha(G_{3,4})\geq 4$ and let $X=\{x_1, x_2, x_3, x_4\}$ be an independent set in $G_{3,4}$.  Note that $X$ induces a $[2]$-colored $K_4$.  Therefore (by Theorem \ref{dualkonig} for instance) there exists a monochromatic component $A_1$, say of color 1, which covers $X$.  Let $B_1$ and $B_2$ be components of color 2 which have the largest intersection with $X$ and without loss of generality, suppose $|B_1\cap X|\geq |B_2\cap X|$ and note that $B_2 \cap X$ is empty if $|B_1\cap X|=4$.

We now claim that $\{A_1, B_1, B_2\}$ is the desired monochromatic 3-cover.  If $v\not\in A_1$, then $v$ sends no edges of color 1 to $X$, at most one edge of color 3, at most one edge of color 4, and consequently at least two edges of color 2.  Thus $v$ must be in either $B_1$ or $B_2$.
\end{proof} 

\subsection{\texorpdfstring{$r=5$}{r=5}, \texorpdfstring{$\alpha=1$}{a=1}}

Let $G$ be an $r$-colored graph and let $X\subseteq V(G)$.  For all $i\in [r]$, the \emph{$i$-signature of $X$}, denoted $\sigma^i[X]$, is the integer partition $(n^i_1, \dots, n^i_{t_i})$ of $|X|$ such that the graph $G_i[X]$ induced by edges of color $i$ in the set $X$ has components of order $n^i_1\geq  n^i_2\geq  \dots\geq  n^i_t$ (where $|X|=\sum_{j=1}^tn_j^i$).  For all $S\subseteq [r]$, the \emph{$S$-signature of $X$}, denoted $\sigma_S(X)$, is the set $\{\sigma^i[X]: i\in [S]\}$.

Now let $n$ and $p$ be positive integers and let $\sigma=\{\sigma^1, \dots, \sigma^p\}$ be a set of integer partitions of $n$.  We say that $\sigma$ is a \emph{valid signature}, if there exists a $p$-coloring of a graph $F$ on $n$ vertices such that the $[p]$-signature of $V(F)$ is $\sigma$ (note that a valid signature may be realized by non-isomorphic colored graphs).  For example, $\{(4,1), (3,1,1), (2,2,1)\}$ is not a valid signature since there is no way to 3-color a $K_5$ so that there are components of order 4 and 1 in color 1, components of order 3, 1, and 1 in color 2, and components of order 2, 2, and 1 in color 3.  

While we don't have a characterization of all valid signatures, the following is a useful necessary condition (and the above example shows that it is not sufficient), which follows simply by counting the number of possible edges.

\begin{observation}\label{obs:valid}
Let $n$ and $p$ be positive integers and let $\sigma=\{\sigma^1, \dots, \sigma^p\}$ be a set of integer partitions of $n$.  If $\sum_{i=1}^p\sum_{j\in [t_i]}\binom{n_j^i}{2}<\binom{n}{2}$, then $\sigma$ is not a valid signature.
\end{observation}

\begin{proof}[Proof of Theorem \ref{thm:r3-5} when $r=5$]
Let $S\subseteq [5]$ with $|S|=2$ and without loss of generality, suppose $S=\{4,5\}$.  If $\alpha(G_{4,5}) \leq 4$, then we are done by Observation \ref{obs:tuza}; so suppose $\alpha(G_{4,5})\geq 5$ and let $X=\{x_1, x_2, x_3, x_4, x_5\}$ be an independent set in $G_{4,5}$.  Note that $X$ induces a $[3]$-colored $K_5$. 

We now split into cases depending on the $[3]$-signature of $X$.  There are 84 possible signatures and 37 of them are valid.  The following lemma deals with 35 of the 37 cases.

\begin{lemma}
\label{lem:r5}
Let $\{\sigma^1, \sigma^2, \sigma^3\}=\{(n^1_1,\dots, n^1_{t_1}), (n^2_1,\dots, n^2_{t_2}), (n^3_1,\dots, n^3_{t_3})\}$ be the $[3]$-signature of $X$. If $i,j,k\in [3]$ are distinct and 
\begin{enumerate}
\item\label{prop:51} $t_i+t_j+|\{n^k_\ell: n^k_\ell\geq 3\}|\leq 4, \text{ or}$
\item\label{prop:52} $t_i+|\{n^j_\ell: n^j_\ell\geq 2\}|+|\{n^k_\ell: n^k_\ell\geq 2\}|\leq 4,$
\end{enumerate}
then there exists a monochromatic $4$-cover in which all of the subgraphs have colors from $[3]$.
\end{lemma}

Since the conditions in Lemma \ref{lem:r5} are a bit hard to parse at first sight, note that (i) says that the number of components of color $i$ or $j$ plus the number of components of order at least 3 of color $k$ in the graph induced by $X$ is at most 4, and (ii) says that the number of components of color $i$ plus the number of components of order at least 2 of color $j$ or $k$ in the graph induced by $X$ is at most 4.  For example, $\{(5), (3,2), (3,2)\}$ and $\{(4,1), (3,1,1), (3,1,1)\}$ are valid signatures to which Lemma \ref{lem:r5}(i) and Lemma \ref{lem:r5}(ii), respectively, apply.  

\begin{proof}
Let $\cT$ denote the set of at most four monochromatic components which intersect $X$ as described in one of the two cases.  Suppose for contradiction that $\cT$ is not a monochromatic 4-cover and let $v$ be an uncovered vertex.  In either case, this implies  $v\not\in X$.  First note that since $G_{4,5}[X]$ is an independent set, $v$ sends at most one edge of color 4 and at most one edge of color 5 to $X$.  Thus $v$ sends at least three edges of color 1, 2, or 3 to $X$ ($\star$).

\begin{enumerate}
    \item Without loss of generality we can assume $\cT$ contains all components of colors 2 and 3 which intersect $X$ and all components of of color 1 which intersect $X$ in at least 3 vertices.  Then $v$ sends no edges of color 2 or 3 to $X$, and at most 2 edges of color 1 to $X$; a contradiction to ($\star$). 
   
    \item Without loss of generality we can assume $\cT$ contains all components of color 3 which intersect $X$ and all components of colors 1 or 2 which intersect $X$ in at least 2 vertices. Then $v$ sends no edges of color 3 to $X$, at most one edge of color 1, and at most one edge of color 2; a contradiction to ($\star$).\qedhere
\end{enumerate}
\end{proof}

By direct inspection, one can see that there are only two valid signatures which do not meet the conditions of Lemma \ref{lem:r5}: $\{(3,2), (3,2), (3,2)\}$ and $\{(4,1), (3,2), (3,2)\}$.  In both cases there are two components of each color which intersect $X$.  Let $A_1, A_2$ be the components of color $1$ which intersect $X$, let $B_1, B_2$ be the components of color $2$ which intersect $X$, and let $C_1, C_2$ be the components of color $3$ which intersect $X$.  Suppose that $|A_1\cap X|\geq |A_2\cap X|$, $|B_1\cap X|\geq |B_2\cap X|$, and $|C_1\cap X|\geq |C_2\cap X|$.

We now deal with these two cases separately.  

\noindent
\textbf{Case 1.} $\{(3,2),(3,2),(3,2)\}$.

Without loss of generality, we must have the following situation:
\begin{align*}
V(A_1)\cap X &= \{x_1, x_2, x_3\}, ~V(A_2)\cap X = \{x_4, x_5\},\\
V(B_1)\cap X &= \{x_1, x_2, x_4\}, ~V(B_2)\cap X  = \{x_3, x_5\},\\
V(C_1)\cap X &= \{x_1, x_2, x_5\}, ~V(C_2)\cap X  = \{x_3, x_4\}. 
\end{align*}

\begin{figure}[ht]
\begin{subfigure}[t]{.5\textwidth}
  \centering
 \includegraphics{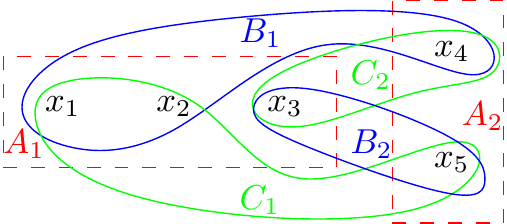} 
\caption{Case 1}\label{fig:51}
\end{subfigure}
\begin{subfigure}[t]{.5\textwidth}
   \centering
 \includegraphics{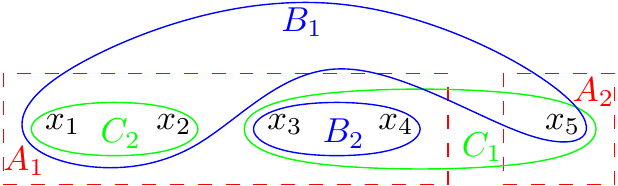} 
\caption{Case 2}\label{fig:52}
\end{subfigure}
\caption{$r=5$}\label{fig:5}
\end{figure}

Suppose that neither  $\{ B_1, B_2, C_1, C_2 \}$  nor $\{ A_1, A_2, B_1, C_1 \}$ are monochromatic 4-covers of $K$.  Note that if $u\notin B_1\cup  B_2\cup  C_1\cup  C_2$, then $u$ must send one edge of color 4 and one edge of color 5 to $\{x_4,x_5\}$ and three edges of color 1 to $\{x_1, x_2, x_3\}$.   If $v\notin A_1\cup  A_2\cup  B_1\cup  C_1$, then $v$ must send one edge of color 4 and one edge of color 5 to $\{x_1,x_2\}$, at least one edge of color 3 to $\{x_3,x_4\}$ and at least one edge of color $2$ to $\{x_3,x_5\}$.  But then no matter what the color of the edge between $u$ and $v$ we have a contradiction.  

\noindent
\textbf{Case 2.} $\{(4,1),(3,2),(3,2)\}$.

Without loss of generality, we must have the following situation:
\begin{align*}
V(A_1)\cap X &= \{x_1, x_2, x_3, x_4\}, ~V(A_2)\cap X = \{x_5\},\\
V(B_1)\cap X &= \{x_1, x_2, x_5\}, ~V(B_2)\cap X = \{x_3, x_4\},\\
V(C_1)\cap X &= \{x_3, x_4, x_5\}, ~V(C_2)\cap X = \{x_1, x_2\}. 
\end{align*}

Suppose that neither $\{A_1, A_2, B_1, B_2\}$ nor $\{A_1, A_2, C_1, C_2\}$ are monochromatic 4-covers of $K$. Note that any vertex $u$ which is not in $A_1\cup A_2\cup B_1\cup B_2$ must send one edge of color 4 and one edge of color 5 to $\{x_1, x_2\}$ and three edges of color 3 to $\{x_3, x_4, x_5\}$ (so $u\in C_1$).  Likewise any vertex $v$ which is not in $A_1\cup A_2\cup C_1\cup C_2$ must send one edge of color 4 and one edge of color 5 to $\{x_3, x_4\}$ and must send three edges of color 2 to $\{x_1, x_2, x_5\}$ (so $v\in B_1$).  

The only possible color for the edge $uv$ is color 1.  Let $A_3$ be the component of color 1 which contains $u$ and $v$.   We now claim that $\{A_1, A_2, A_3, B_1\}$ is a monochromatic cover.  We establish this claim by showing that if $w\not\in A_1\cup A_2\cup B_1$, then $w$ must send an edge of color 1 to either $u$ or $v$.  So let $w$ be such that $w\not\in A_1\cup A_2\cup B_1$ and suppose for contradiction that $w$ does not send an edge of color 1 to $\{u,v\}$.  If $w$ sends an edge of color 3 to $\{x_1, x_2\}$, then $w$ must send an edge of color 2 to $\{x_3,x_4\}$ which further implies that $w$ must send an edge of color 4 or 5, say color 5, to $\{x_5\}$.  Now $w$ can only send edges of color $4$ to $\{u,v\}$, but then this causes $u$ and $v$ to be in the same component of color 4, a contradiction.  So suppose $w$ does not send an edge of color 3 to $\{x_1, x_2\}$, which means $w$ must send an edge of color 4 to $\{x_1, x_2\}$ and an edge of color 5 to $\{x_1, x_2\}$, consequently $w$ must send an edge of color 3 to $\{x_3, x_4, x_5\}$ (so $w\in C_1$).  Now $w$ is forced to send an edge of color 1 to $v$.  This completes the case.  
\end{proof}

\subsection{What we know for \texorpdfstring{$r=6$}{r=6}, \texorpdfstring{$\alpha=1$}{a=1}}

Let $S\subseteq [6]$ with $|S|=2$ and without loss of generality, suppose $S=\{5,6\}$.  If $\alpha(G_{5,6}) \leq 5$, then we are done by Observation \ref{obs:tuza}; so suppose $\alpha(G_{5,6})\geq 6$ and let $X=\{x_1, x_2, x_3, x_4, x_5,x_6\}$ be an independent set in $G_{5,6}$.  Note that $X$ induces a $[4]$-colored $K_6$. 

We now split into cases depending on the $[4]$-signature of $X$.  There are 1001 possible signatures, 560 of which are valid.  The following two lemmas deal with 387 of the 560 cases\footnote{We wrote a computer program which first determined all valid $[4]$-signatures of $X$ (the program goes beyond Observation \ref{obs:valid} by doing a search to see whether each potential signature can be realized by a $4$-coloring of $K_6$), then eliminates all valid signatures for which Lemma \ref{lem:r6} or Lemma \ref{lem:r6ii} applies.}.

\begin{lemma}
\label{lem:r6}
Let $\{\sigma^1, \sigma^2, \sigma^3,\sigma^4\}=\{(n^1_1,\dots, n^1_{t_1}), (n^2_1,\dots, n^2_{t_2}), (n^3_1,\dots, n^3_{t_3}), (n^4_1,\dots, n^4_{t_4})\}$ be the $[4]$-signature of $X$. If there exists distinct $i,j,k,\ell\in [4]$ such that
\begin{enumerate}[label={(\roman*)}]
\item \label{prop:61} $t_i+|\{n^j_m: n^j_m \geq 2|+|\{n^k_m: n^k_m \geq 2|+|\{n^\ell_m: n^\ell_m \geq 2|\leq 5, \text{ or}$

\item\label{prop:62} $t_i+t_j+|\{n^k_m: n^k_m\geq 2\}|+|\{n^\ell_m: n^\ell_m\geq 3\}|\leq 5, \text{ or}$

\item\label{prop:63} $t_i+t_j+t_k+|\{n^\ell_m: n^\ell_m\geq 4\}|\leq 5,$
\end{enumerate}
then there exists a monochromatic $5$-cover in which all of the subgraphs have colors from $[4]$.
\end{lemma}

For example, $\{(5,1), (3,1,1,1), (3,1,1,1), (2,1,1,1,1)\}$, $\{(5,1), (3,3), (3,1,1,1), (2,2,2)\}$, and $\{(6), (4,2), (4,2), (3,3)\}$ are valid signatures to which Lemma \ref{lem:r6}(i), Lemma \ref{lem:r6}(ii), Lemma \ref{lem:r6}(iii), respectively, apply.

\begin{proof}
Let $\cT$ denote the set of at most five monochromatic components which intersect $X$ as described in the three cases.  Suppose for contradiction that $\cT$ is not a monochromatic 5-cover and let $v$ be an uncovered vertex.  First note that since $G_{5,6}[X]$ is an independent set, $v$ sends at most one edge of color 5 and at most one edge of color 6 to $X$, unless $v\in X$ in which case $v$ sends no edges of color 5 and no edges of color 6 to $X$.  Thus in any case $v$ sends at least four edges of color 1, 2, 3, or 4 to $X$ ($\star$).

\begin{enumerate}[label={(\roman*)}]
    \item  Without loss of generality we can assume $\cT$ contains all components of color 4 which intersect $X$ and all components of  color 1, 2, or 3 which intersect $X$ in at least 2 vertices.  Thus $v$ sends no edges of color 4 to $X$, and at most one edge of colors 1,2, or 3; a contradiction to ($\star$).
    
    \item Without loss of generality we can assume $\cT$ contains all components of color 3 and 4 which intersect $X$, all components of color 2 which intersect $X$ in at least 3 vertices, and all components of color $1$ which intersect $X$ in at least 2 vertices.  Thus $v$ sends no edges of color 3 or 4 to $X$, at most two edges of color 2 to $X$, and at most one edge of color 1 to $X$; a contradiction to ($\star$). 
    
    \item Without loss of generality, we can assume $\cT$ contains all components of colors 2, 3, and 4 which intersect $X$, and all components of color 1 which intersect $X$ in at least 4 vertices. Thus $v$ sends no edges of color 2, 3 or 4 to $X$ and at most three edges of color 1 to $X$; a contradiction to ($\star$).\qedhere
\end{enumerate}
\end{proof}

\begin{lemma}
\label{lem:r6ii}
Let $W\subseteq X$ and let $$\{\sigma^1, \sigma^2, \sigma^3,\sigma^4\}=\{(n^1_1,\dots, n^1_{t_1}), (n^2_1,\dots, n^2_{t_2}), (n^3_1,\dots, n^3_{t_3}), (n^4_1,\dots, n^4_{t_4})\}$$ be the $[4]$-signature of $W$. If there exists distinct $i,j,k,\ell\in [4]$ such that
\begin{enumerate}[label={(\roman*)}]
\item\label{prop:64} $|W|=3$ and $t_1+t_2+t_3+t_4\leq 5$, or

\item\label{prop:65} $|W|=4$ and $t_i+t_j+t_k+|\{n^\ell_m: n^\ell_m\geq 2\}|\leq 5, \text{ or}$

\item\label{prop:66} $|W|=5$ and $t_i+t_j+|\{n^k_m: n^k_m\geq 2\}|+|\{n^\ell_m: n^\ell_m\geq 2\}|\leq 5, \text{ or}$

\item\label{prop:67} $|W|=5$ and $t_i+t_j+t_k+|\{n^\ell_m: n^\ell_m\geq 3|\leq 5,$
\end{enumerate}
then there exists a monochromatic $5$-cover in which all of the subgraphs have colors from $[4]$.
\end{lemma}

\begin{proof}
Let $\cT$ denote the set of at most five monochromatic components which intersect $W$ as described in the three cases.  Suppose for contradiction that $\cT$ is not a monochromatic 5-cover and let $v$ be an uncovered vertex.  First note that since $G_{5,6}[X]$ is an independent set, $v$ sends at most one edge of color 5 and at most one edge of color 6 to $X$, unless $v\in X$ in which case $v$ sends no edges of color 5 and no edges of color 6 to $X$ ($\star$).  

\begin{enumerate}[label={(\roman*)}]
    \item  Note that $\cT$ contains all components of colors 1, 2, 3, and 4 which intersect $W$, so $v$ sends no edges of color 1, 2, 3, or 4 to $W$ which together with ($\star$) and the fact that $|W|=3$ is a contradiction.
    
    \item Without loss of generality we can assume $\cT$ contains all components of color 2, 3,  and 4 which intersect $W$, and all components of color 1 which intersect $W$ in at least 2 vertices. Thus $v$ sends at most one edge of color 1 to $W$ which together with ($\star$) and the fact that $|W|=4$ is a contradiction. 
    
    \item Without loss of generality, we can assume $\cT$ contains all components of colors 3 and 4 which intersect $W$, and all components of color 1 or 2 which intersect $W$ in at least 2 vertices. Thus $v$ sends no edges of color 3 or 4 to $W$, at most one edge of color 1, and at most one edge of color 2, which together with ($\star$) and the fact that $|W|=5$ is a contradiction.
    
  \item Without loss of generality, we can assume $\cT$ contains all components of colors 2, 3, and 4 which intersect $W$, and all components of color 1 which intersect $W$ in at least 3 vertices. Thus $v$ sends no edges of color 2, 3, or 4 to $W$, at most two edges of color 1, which together with ($\star$) and the fact that $|W|=5$ is a contradiction.\qedhere
\end{enumerate}
\end{proof}

We are left with are 173 valid signatures for which an ad-hoc proof is needed (see Table \ref{tab:r6}).  


\begin{table}[ht]
\tiny
\begin{minipage}[t]{0.33\textwidth}
\begin{tabular}[t]{l}
$\{(6), (4, 2), (4, 2), (4, 2)\}$ \\
$\{(6), (4, 2), (3, 2, 1), (3, 2, 1)\}$ \\
$\{(6), (4, 2), (2, 2, 2), (2, 2, 2)\}$ \\
$\{(6), (4, 1, 1), (3, 3), (3, 3)\}$ \\
$\{(6), (4, 1, 1), (3, 2, 1), (3, 2, 1)\}$ \\
$\{(6), (4, 1, 1), (2, 2, 1, 1), (2, 2, 1, 1)\}$ \\
$\{(6), (3, 3), (3, 2, 1), (3, 2, 1)\}$ \\
$\{(6), (3, 3), (2, 2, 2), (2, 2, 2)\}$ \\
$\{(6), (3, 2, 1), (3, 2, 1), (3, 2, 1)\}$ \\
$\{(6), (3, 2, 1), (3, 2, 1), (3, 1, 1, 1)\}$ \\
$\{(6), (3, 2, 1), (3, 2, 1), (2, 2, 2)\}$ \\
$\{(6), (3, 2, 1), (3, 2, 1), (2, 2, 1, 1)\}$ \\
$\{(6), (3, 2, 1), (3, 2, 1), (2, 1, 1, 1, 1)\}$ \\
$\{(6), (3, 2, 1), (2, 2, 2), (2, 2, 2)\}$ \\
$\{(6), (3, 2, 1), (2, 2, 2), (2, 2, 1, 1)\}$ \\
$\{(6), (3, 2, 1), (2, 2, 1, 1), (2, 2, 1, 1)\}$ \\
$\{(6), (3, 1, 1, 1), (3, 1, 1, 1), (2, 2, 2)\}$ \\
$\{(6), (3, 1, 1, 1), (2, 2, 1, 1), (2, 2, 1, 1)\}$ \\
$\{(6), (2, 2, 2), (2, 2, 2), (2, 2, 2)\}$ \\
$\{(6), (2, 2, 2), (2, 2, 2), (2, 2, 1, 1)\}$ \\
$\{(6), (2, 2, 2), (2, 2, 1, 1), (2, 2, 1, 1)\}$ \\
$\{(6), (2, 2, 1, 1), (2, 2, 1, 1), (2, 2, 1, 1)\}$ \\
$\{(6), (2, 2, 1, 1), (2, 2, 1, 1), (2, 1, 1, 1, 1)\}$ \\
$\{(5, 1), (5, 1), (4, 2), (4, 2)\}$ \\
$\{(5, 1), (5, 1), (4, 2), (3, 3)\}$ \\
$\{(5, 1), (5, 1), (4, 1, 1), (3, 2, 1)\}$ \\
$\{(5, 1), (5, 1), (3, 2, 1), (3, 2, 1)\}$ \\
$\{(5, 1), (5, 1), (3, 2, 1), (3, 1, 1, 1)\}$ \\
$\{(5, 1), (5, 1), (3, 2, 1), (2, 2, 2)\}$ \\
$\{(5, 1), (5, 1), (3, 2, 1), (2, 2, 1, 1)\}$ \\
$\{(5, 1), (5, 1), (2, 2, 2), (2, 2, 2)\}$ \\
$\{(5, 1), (5, 1), (2, 2, 2), (2, 2, 1, 1)\}$ \\
$\{(5, 1), (4, 2), (4, 2), (4, 2)\}$ \\
$\{(5, 1), (4, 2), (4, 2), (4, 1, 1)\}$ \\
$\{(5, 1), (4, 2), (4, 2), (3, 3)\}$ \\
$\{(5, 1), (4, 2), (4, 2), (3, 2, 1)\}$ \\
$\{(5, 1), (4, 2), (4, 2), (3, 1, 1, 1)\}$ \\
$\{(5, 1), (4, 2), (4, 1, 1), (3, 3)\}$ \\
$\{(5, 1), (4, 2), (4, 1, 1), (3, 2, 1)\}$ \\
$\{(5, 1), (4, 2), (3, 3), (3, 3)\}$ \\
$\{(5, 1), (4, 2), (3, 3), (3, 2, 1)\}$ \\
$\{(5, 1), (4, 2), (3, 2, 1), (3, 2, 1)\}$ \\
$\{(5, 1), (4, 2), (3, 2, 1), (3, 1, 1, 1)\}$ \\
$\{(5, 1), (4, 2), (3, 2, 1), (2, 2, 2)\}$ \\
$\{(5, 1), (4, 2), (3, 2, 1), (2, 2, 1, 1)\}$ \\
$\{(5, 1), (4, 2), (3, 2, 1), (2, 1, 1, 1, 1)\}$ \\
$\{(5, 1), (4, 2), (2, 2, 2), (2, 2, 2)\}$ \\
$\{(5, 1), (4, 2), (2, 2, 2), (2, 2, 1, 1)\}$ \\
$\{(5, 1), (4, 1, 1), (4, 1, 1), (3, 2, 1)\}$ \\
$\{(5, 1), (4, 1, 1), (4, 1, 1), (2, 2, 2)\}$ \\
$\{(5, 1), (4, 1, 1), (4, 1, 1), (2, 2, 1, 1)\}$ \\
$\{(5, 1), (4, 1, 1), (3, 3), (3, 3)\}$ \\
$\{(5, 1), (4, 1, 1), (3, 3), (3, 2, 1)\}$ \\
$\{(5, 1), (4, 1, 1), (3, 2, 1), (3, 2, 1)\}$ \\
$\{(5, 1), (4, 1, 1), (3, 2, 1), (3, 1, 1, 1)\}$ \\
$\{(5, 1), (4, 1, 1), (3, 2, 1), (2, 2, 2)\}$ \\
$\{(5, 1), (4, 1, 1), (3, 2, 1), (2, 2, 1, 1)\}$ \\
$\{(5, 1), (4, 1, 1), (3, 2, 1), (2, 1, 1, 1, 1)\}$ \\

\end{tabular}
\end{minipage}%
\hfill
\begin{minipage}[t]{0.33\textwidth}
\begin{tabular}[t]{l}

$\{(5, 1), (4, 1, 1), (3, 1, 1, 1), (2, 2, 2)\}$ \\
$\{(5, 1), (4, 1, 1), (3, 1, 1, 1), (2, 2, 1, 1)\}$ \\
$\{(5, 1), (4, 1, 1), (2, 2, 2), (2, 2, 2)\}$ \\
$\{(5, 1), (4, 1, 1), (2, 2, 2), (2, 2, 1, 1)\}$ \\
$\{(5, 1), (4, 1, 1), (2, 2, 1, 1), (2, 2, 1, 1)\}$ \\
$\{(5, 1), (3, 3), (3, 3), (3, 3)\}$ \\
$\{(5, 1), (3, 3), (3, 3), (3, 2, 1)\}$ \\
$\{(5, 1), (3, 3), (3, 2, 1), (3, 2, 1)\}$ \\
$\{(5, 1), (3, 3), (3, 2, 1), (3, 1, 1, 1)\}$ \\
$\{(5, 1), (3, 3), (3, 2, 1), (2, 2, 2)\}$ \\
$\{(5, 1), (3, 3), (3, 2, 1), (2, 2, 1, 1)\}$ \\
$\{(5, 1), (3, 2, 1), (3, 2, 1), (3, 2, 1)\}$ \\
$\{(5, 1), (3, 2, 1), (3, 2, 1), (3, 1, 1, 1)\}$ \\
$\{(5, 1), (3, 2, 1), (3, 2, 1), (2, 2, 2)\}$ \\
$\{(5, 1), (3, 2, 1), (3, 2, 1), (2, 2, 1, 1)\}$ \\
$\{(5, 1), (3, 2, 1), (3, 2, 1), (2, 1, 1, 1, 1)\}$ \\
$\{(5, 1), (3, 2, 1), (3, 1, 1, 1), (3, 1, 1, 1)\}$ \\
$\{(5, 1), (3, 2, 1), (3, 1, 1, 1), (2, 2, 2)\}$ \\
$\{(5, 1), (3, 2, 1), (3, 1, 1, 1), (2, 2, 1, 1)\}$ \\
$\{(5, 1), (3, 1, 1, 1), (3, 1, 1, 1), (2, 2, 2)\}$ \\
$\{(5, 1), (3, 1, 1, 1), (3, 1, 1, 1), (2, 2, 1, 1)\}$ \\
$\{(4, 2), (4, 2), (4, 2), (4, 2)\}$ \\
$\{(4, 2), (4, 2), (4, 2), (4, 1, 1)\}$ \\
$\{(4, 2), (4, 2), (4, 2), (3, 3)\}$ \\
$\{(4, 2), (4, 2), (4, 2), (3, 2, 1)\}$ \\
$\{(4, 2), (4, 2), (4, 2), (3, 1, 1, 1)\}$ \\
$\{(4, 2), (4, 2), (4, 2), (2, 2, 2)\}$ \\
$\{(4, 2), (4, 2), (4, 2), (2, 2, 1, 1)\}$ \\
$\{(4, 2), (4, 2), (4, 1, 1), (4, 1, 1)\}$ \\
$\{(4, 2), (4, 2), (4, 1, 1), (3, 3)\}$ \\
$\{(4, 2), (4, 2), (4, 1, 1), (3, 2, 1)\}$ \\
$\{(4, 2), (4, 2), (4, 1, 1), (3, 1, 1, 1)\}$ \\
$\{(4, 2), (4, 2), (3, 3), (3, 3)\}$ \\
$\{(4, 2), (4, 2), (3, 3), (3, 2, 1)\}$ \\
$\{(4, 2), (4, 2), (3, 3), (3, 1, 1, 1)\}$ \\
$\{(4, 2), (4, 2), (3, 3), (2, 2, 2)\}$ \\
$\{(4, 2), (4, 2), (3, 3), (2, 2, 1, 1)\}$ \\
$\{(4, 2), (4, 2), (3, 2, 1), (3, 2, 1)\}$ \\
$\{(4, 2), (4, 2), (3, 2, 1), (3, 1, 1, 1)\}$ \\
$\{(4, 2), (4, 2), (3, 2, 1), (2, 2, 2)\}$ \\
$\{(4, 2), (4, 2), (3, 2, 1), (2, 2, 1, 1)\}$ \\
$\{(4, 2), (4, 2), (3, 2, 1), (2, 1, 1, 1, 1)\}$ \\
$\{(4, 2), (4, 2), (3, 1, 1, 1), (3, 1, 1, 1)\}$ \\
$\{(4, 2), (4, 2), (2, 2, 2), (2, 2, 2)\}$ \\
$\{(4, 2), (4, 2), (2, 2, 2), (2, 2, 1, 1)\}$ \\
$\{(4, 2), (4, 2), (2, 2, 1, 1), (2, 2, 1, 1)\}$ \\
$\{(4, 2), (4, 1, 1), (4, 1, 1), (3, 3)\}$ \\
$\{(4, 2), (4, 1, 1), (4, 1, 1), (3, 2, 1)\}$ \\
$\{(4, 2), (4, 1, 1), (4, 1, 1), (2, 2, 2)\}$ \\
$\{(4, 2), (4, 1, 1), (4, 1, 1), (2, 2, 1, 1)\}$ \\
$\{(4, 2), (4, 1, 1), (3, 3), (3, 3)\}$ \\
$\{(4, 2), (4, 1, 1), (3, 3), (3, 2, 1)\}$ \\
$\{(4, 2), (4, 1, 1), (3, 3), (3, 1, 1, 1)\}$ \\
$\{(4, 2), (4, 1, 1), (3, 2, 1), (3, 2, 1)\}$ \\
$\{(4, 2), (4, 1, 1), (3, 2, 1), (3, 1, 1, 1)\}$ \\
$\{(4, 2), (4, 1, 1), (3, 2, 1), (2, 2, 2)\}$ \\
$\{(4, 2), (4, 1, 1), (3, 2, 1), (2, 2, 1, 1)\}$ \\
$\{(4, 2), (4, 1, 1), (2, 2, 2), (2, 2, 2)\}$ \\

\end{tabular}
\end{minipage}%
\hfill
\begin{minipage}[t]{0.33\textwidth}
\begin{tabular}[t]{l}

$\{(4, 2), (4, 1, 1), (2, 2, 2), (2, 2, 1, 1)\}$ \\
$\{(4, 2), (4, 1, 1), (2, 2, 1, 1), (2, 2, 1, 1)\}$ \\
$\{(4, 2), (3, 3), (3, 3), (3, 3)\}$ \\
$\{(4, 2), (3, 3), (3, 3), (3, 2, 1)\}$ \\
$\{(4, 2), (3, 3), (3, 3), (3, 1, 1, 1)\}$ \\
$\{(4, 2), (3, 3), (3, 3), (2, 2, 2)\}$ \\
$\{(4, 2), (3, 3), (3, 3), (2, 2, 1, 1)\}$ \\
$\{(4, 2), (3, 3), (3, 3), (2, 1, 1, 1, 1)\}$ \\
$\{(4, 2), (3, 3), (3, 2, 1), (3, 2, 1)\}$ \\
$\{(4, 2), (3, 3), (3, 2, 1), (3, 1, 1, 1)\}$ \\
$\{(4, 2), (3, 3), (3, 2, 1), (2, 2, 2)\}$ \\
$\{(4, 2), (3, 3), (3, 2, 1), (2, 2, 1, 1)\}$ \\
$\{(4, 2), (3, 3), (3, 1, 1, 1), (3, 1, 1, 1)\}$ \\
$\{(4, 2), (3, 3), (2, 2, 2), (2, 2, 2)\}$ \\
$\{(4, 2), (3, 3), (2, 2, 2), (2, 2, 1, 1)\}$ \\
$\{(4, 2), (3, 3), (2, 2, 1, 1), (2, 2, 1, 1)\}$ \\
$\{(4, 2), (3, 2, 1), (3, 2, 1), (3, 2, 1)\}$ \\
$\{(4, 2), (3, 2, 1), (3, 2, 1), (3, 1, 1, 1)\}$ \\
$\{(4, 2), (3, 2, 1), (3, 2, 1), (2, 2, 2)\}$ \\
$\{(4, 2), (3, 2, 1), (3, 2, 1), (2, 2, 1, 1)\}$ \\
$\{(4, 1, 1), (4, 1, 1), (4, 1, 1), (4, 1, 1)\}$ \\
$\{(4, 1, 1), (4, 1, 1), (4, 1, 1), (3, 2, 1)\}$ \\
$\{(4, 1, 1), (4, 1, 1), (4, 1, 1), (3, 1, 1, 1)\}$ \\
$\{(4, 1, 1), (4, 1, 1), (4, 1, 1), (2, 2, 2)\}$ \\
$\{(4, 1, 1), (4, 1, 1), (4, 1, 1), (2, 2, 1, 1)\}$ \\
$\{(4, 1, 1), (4, 1, 1), (4, 1, 1), (2, 1, 1, 1, 1)\}$ \\
$\{(4, 1, 1), (4, 1, 1), (3, 3), (3, 3)\}$ \\
$\{(4, 1, 1), (4, 1, 1), (3, 3), (3, 2, 1)\}$ \\
$\{(4, 1, 1), (4, 1, 1), (3, 3), (2, 2, 2)\}$ \\
$\{(4, 1, 1), (4, 1, 1), (3, 3), (2, 2, 1, 1)\}$ \\
$\{(4, 1, 1), (4, 1, 1), (3, 2, 1), (3, 2, 1)\}$ \\
$\{(4, 1, 1), (4, 1, 1), (3, 2, 1), (3, 1, 1, 1)\}$ \\
$\{(4, 1, 1), (4, 1, 1), (3, 2, 1), (2, 2, 2)\}$ \\
$\{(4, 1, 1), (4, 1, 1), (3, 2, 1), (2, 2, 1, 1)\}$ \\
$\{(4, 1, 1), (4, 1, 1), (3, 1, 1, 1), (3, 1, 1, 1)\}$ \\
$\{(4, 1, 1), (4, 1, 1), (2, 2, 2), (2, 2, 2)\}$ \\
$\{(4, 1, 1), (4, 1, 1), (2, 2, 2), (2, 2, 1, 1)\}$ \\
$\{(4, 1, 1), (4, 1, 1), (2, 2, 1, 1), (2, 2, 1, 1)\}$ \\
$\{(4, 1, 1), (3, 3), (3, 3), (3, 3)\}$ \\
$\{(4, 1, 1), (3, 3), (3, 3), (3, 2, 1)\}$ \\
$\{(4, 1, 1), (3, 3), (3, 2, 1), (3, 2, 1)\}$ \\
$\{(4, 1, 1), (3, 3), (3, 2, 1), (3, 1, 1, 1)\}$ \\
$\{(4, 1, 1), (3, 2, 1), (3, 2, 1), (3, 2, 1)\}$ \\
$\{(4, 1, 1), (3, 2, 1), (3, 2, 1), (3, 1, 1, 1)\}$ \\
$\{(3, 3), (3, 3), (3, 3), (3, 3)\}$ \\
$\{(3, 3), (3, 3), (3, 3), (3, 2, 1)\}$ \\
$\{(3, 3), (3, 3), (3, 3), (3, 1, 1, 1)\}$ \\
$\{(3, 3), (3, 3), (3, 3), (2, 2, 2)\}$ \\
$\{(3, 3), (3, 3), (3, 2, 1), (3, 2, 1)\}$ \\
$\{(3, 3), (3, 3), (3, 2, 1), (3, 1, 1, 1)\}$ \\
$\{(3, 3), (3, 3), (3, 2, 1), (2, 2, 2)\}$ \\
$\{(3, 3), (3, 3), (2, 2, 2), (2, 2, 2)\}$ \\
$\{(3, 3), (3, 3), (2, 2, 2), (2, 2, 1, 1)\}$ \\
$\{(3, 3), (3, 2, 1), (3, 2, 1), (3, 2, 1)\}$ \\
$\{(3, 3), (3, 2, 1), (3, 2, 1), (2, 2, 2)\}$ \\
$\{(3, 3), (2, 2, 2), (2, 2, 2), (2, 2, 2)\}$ \\
$\{(3, 2, 1), (3, 2, 1), (3, 2, 1), (2, 2, 2)\}$ \\
\end{tabular}
\end{minipage}
\caption{The 173 valid signatures for the case $r=6$ which are not ruled out by Lemma \ref{lem:r6} or Lemma \ref{lem:r6ii}.}\label{tab:r6}
\end{table}

\section{Covering with monochromatic subgraphs of bounded diameter}\label{sec:diameter}

The following is a well-known fact (see \cite[Theorem 2.1.11]{W}).

\begin{proposition}[Folklore]\label{prop:folk}
Let $G$ be a 2-colored complete graph. If $\diam(G_1)\geq 3$, then $\diam(G_2)\leq 3$.
\end{proposition}

Also note that Proposition \ref{prop:folk} is best possible.  To see this, partition $V$ as $\{V_1, V_2, V_3, V_4\}$ and color all edges from $V_i$ to $V_{i+1}$ with color 1 for all $i\in [3]$ and color all other edges with color 2.  Both $G_1$ and $G_2$ have diameter 3.

Let $\dc_r^\delta(G)$ be the smallest integer $t$ such that in every $r$-coloring of the edges of $G$, there exists 
a monochromatic $t$-cover $\cT$ such that for all $T\in \cT$, $\diam(T)\leq \delta$.  For $r\geq 1$ and a graph $G$, let $D_r(G)$ be the smallest $\delta$ such that $\dc_r^\delta(G)\leq \tc_r(G)$.  For instance, Proposition \ref{prop:folk} implies $\dc_2^3(K)=1$ for all complete graphs $K$ (i.e.\ $D_2(K)=3$). Erd\H{o}s and Fowler \cite{EF} proved that there exists a 2-coloring of $K_n$ such that every monochromatic subgraph of diameter at most 2 has order at most $(3/4+o(1))n$ and thus $\dc_2^2(K_n)\geq 2$.  Also by considering the edges incident with any vertex, we clearly have $\dc_2^2(K)=2$ for all complete graphs $K$.

In this language, Mili\'cevi\'c conjectured the following strengthening of Ryser's conjecture.  

\begin{conjecture}[Mili\'cevi\'c \cite{M2}]\label{con:mil}
For all $r\geq 2$, there exists $\delta=\delta(r)$ such that $\dc_r^\delta(K_n)\leq r-1$.
\end{conjecture}

We make the following more general conjecture which is also stronger in the sense that $\delta$ doesn't depend on $r$ (we note that perhaps $\delta$ doesn't even depend on $\alpha$).

\begin{conjecture}
For all $\alpha\geq 1$, there exists $\delta=\delta(\alpha)$ such that for all $r\geq 2$, if $G$ is a graph with $\alpha(G)=\alpha$, then $\dc_r^{\delta}(G)\leq (r-1)\alpha$.
\end{conjecture}

Sometimes we will make the distinction between whether the subgraphs in our monochromatic cover are trees or not.  
Let $\tdc_r^\delta(G)$ be the smallest integer $t$ such that in every $r$-coloring of the edges of $G$, there exists 
a monochromatic $t$-cover $\cT$ such that for all $T\in \cT$, $T$ is a tree and $\diam(T)\leq \delta$.  For $r\geq 1$ and a graph $G$, let $TD_r(G)$ be the smallest $\delta$ such that $\tdc_r^\delta(G)\leq \tc_r(G)$.  The following fact implies that $\tdc_r^{2\delta}(G)\leq \dc_r^{\delta}(G)$ (i.e.\ $TD_r(G)\leq 2D_r(G)$).

\begin{fact}~
\begin{enumerate}
\item Let $G$ be a connected graph.  If $\diam(G)=d$, then $G$ has a spanning tree $T$ with $d\leq \diam(T)\leq 2d$.  

\item Let $T$ be a tree.  If $\rad(T)=d$, then $d\leq \diam(T)\leq 2d$.  
\end{enumerate}
\end{fact}

Note that by considering a random 2-coloring of $K_n$, there is no monochromatic spanning tree of diameter $3$, so $TD_2(K)\geq 4$.  It is well-known (see \cite[Exercise 2.1.49]{W} and \cite[Theorem 2.1]{BDV}) that $\tdc_2^4(K)=1$ for all complete graphs $K$ and thus $TD_2(K)=4$.

The following theorem summarizes the relevant results from \cite{M1} and \cite{M2}.

\begin{theorem}[Mili\'cevi\'c \cite{M1}, \cite{M2}]\label{thm:mil}~
\begin{enumerate}
\item For all complete graphs $K$, $\dc_3^8(K)\leq 2$ (i.e. $D_3(K)\leq 8$).
\item For all $G\in \mathcal{K}_2$, $\dc_2^9(G)\leq 2$  (i.e. $D_2(G)\leq 9$).
\item For all complete graphs $K$, $\dc_4^{80}(K)\leq 3$  (i.e. $D_4(K)\leq 80$).
\end{enumerate}
\end{theorem}

We improve the bounds in each item of Theorem \ref{thm:mil} and give a simpler proof for (iii).

\begin{theorem}\label{thm:r3complete}
For all complete graphs $K$, $\dc_3^4(K)\leq \tdc_3^4(K)\leq 2$ and $3\leq D_3(K)\leq TD_3(K)\leq 4$.
\end{theorem}

\begin{theorem}\label{thm:r2bipartite}
For all $G\in \mathcal{K}_2$, $\dc_2^4(G)\leq \tdc_2^4(G)\leq 2$ and $3\leq D_2(G)\leq TD_2(G)\leq 4$.
\end{theorem}

\begin{theorem}\label{thm:r4complete}
For all complete graphs $K$, $\dc_4^6(K)\leq 2$ and $2\leq D_4(K)\leq 6$.
\end{theorem}

We also prove analogous results in some new cases.

\begin{theorem}\label{thm:r2a2}
For all graphs $G$ with $\alpha(G)=2$, $\dc_2^6(G)\leq 2$ and $3\leq D_2(G)\leq 6$.
\end{theorem}

\begin{theorem}\label{thm:r3bipartite}
For all $G\in \mathcal{K}_2$, $\dc_3^{6}(G)\leq 4$ (i.e. $D_3(G)\leq 6$).
\end{theorem}

Very interestingly, the reason Mili\'cevi\'c proved Theorem \ref{thm:mil}(i) and then formulated Conjecture \ref{con:mil} had to do with a generalization of Banach's fixed point theorem 

\begin{theorem}[Banach \cite{Ban}]
Every contracting operator on a complete metric space has a fixed point. 
\end{theorem}

Austin \cite{Aus} conjectured that every commuting contracting family $\{f_1,f_2,\dots,f_r\}$ of operators\footnote{A family of operators (i.e.\ continuous functions from $M$ to $M$) $\{f_1,f_2,\dots,f_r\}$ is \emph{contracting} if there exists $0< \lambda\leq 1$ such that for all $x,y\in M$, there exists $i\in [r]$ such that $d(f_i(x), f_i(y))<\lambda  \cdot d(x,y)$.} on a complete metric space $(M,d)$ has a common fixed point and proved it for $r=2$.  Mili\'cevi\'c proved the case $r=3$.

\begin{theorem}[Mili\'cevi\'c \cite{M1}]\label{thm:mil3}
Every commuting contracting family $\{f_1,f_2,f_3\}$ of operators on a complete metric space has a common fixed point.
\end{theorem}

In the course of proving Theorem \ref{thm:mil3}, Mili\'cevi\'c requires a lemma which says that  that there exists some $\delta$ ($\delta=8$ suffices) such that $\dc_3^\delta(K)\leq 2$ for the countably infinite complete graph $K$.  

We note that Mili\'cevi\'c's proofs and our proofs apply equally well to finite or infinite (countable or uncountable) graphs.

\subsection{Examples}

Given a graph $G$, a \emph{blow-up} of $G$ is a graph obtained by replacing each vertex of $G$ with an independent set and replacing each edge of $G$ with a complete bipartite graph between the corresponding independent sets. A \emph{closed blow-up} of $G$ is a graph obtained by replacing each vertex of $G$ with a clique and each edge of $G$ with a complete bipartite graph between the corresponding cliques.  

\begin{example}
For all $n\geq 7$, $D_3(K_n)\geq 3$ (i.e.\ there exists a 3-coloring of $K_n$ such that if $\{H_1, H_2\}$ is a monochromatic 2-cover, then $\diam(H_i)\geq 3$ for some $i\in [2]$).
\end{example}

\begin{proof}
Color $K_7$ with 3 colors so that each color class is a 7-cycle.  Then take a closed blow-up of this example on $n$ vertices coloring the edges between the sets according to the original coloring and color the edges inside the sets arbitrarily.  If $\{H_1,  H_2\}$ is a monochromatic 2-cover, then for some $i\in [2]$, $H_i$ contains vertices from at least four different sets which implies $\diam(H_i)\geq 3$.  
\end{proof}

\begin{example}
For all $m\geq 3$, $n\geq 4$, $D_2(K_{m,n})\geq 3$ (i.e. there exists a 2-coloring of $K_{m,n}$ such that if $\{H_1, H_2\}$ is a monochromatic 2-cover, then $\diam(H_i)\geq 3$ for some $i\in [2]$).
\end{example}

\begin{proof}
Take a red $P_7$ and note that its bipartite complement is a blue $P_7$.  Now take a blow-up of this example on $n$ vertices coloring the edges between the sets according to the original coloring.  If $\{H_1,  H_2\}$ is a monochromatic 2-cover, then for some $i\in [2]$, $H_i$ contains vertices from at least four different sets which implies $\diam(H_i)\geq 3$. 

Another example (provided $m\geq 4$) comes from taking a red $C_8$ and noting that its bipartite complement is a blue $C_8$.
\end{proof}

\begin{example}\label{ex:random3}
For all sufficiently large $n$, $TD_3(K_n)\geq 3$ (i.e.\ there exists a 3-coloring of $K_n$ so that there is no monochromatic 2-cover consisting of trees each of diameter at most 3).
\end{example}

\begin{proof}
Trees of diameter 3 are double-stars.  In a random 3-coloring of $K_n$ (for sufficiently large $n$), no two monochromatic double stars will cover the vertex set.
\end{proof}

\begin{example}\label{ex:biprandom2}
For all sufficiently large $m$ and $n$, $TD_2(K_{m,n})\geq 3$ (i.e.\ there exists a 2-coloring of $K_{m,n}$ so that there is no monochromatic 2-cover consisting of trees each of diameter at most 3).
\end{example}

\begin{proof}
As above, in a random 2-coloring of $K_{m,n}$ (for sufficiently large $m$, $n$), no two monochromatic double stars will cover the vertex set.
\end{proof}

\subsection{Complete graphs, \texorpdfstring{$r=3$}{r=3}}

\begin{proof}[Proof of Theorem \ref{thm:r3complete}]
Let $x \in V(G)$. For $i \in [3]$, let $A_i$ be the neighbors of $x$ of color $i$. If $A_i = \emptyset$ for some $i \in [3]$, then for distinct $j,k\in [3]\setminus \{i\}$ there are trees $H_1\subseteq G_j[\{x\} \cup A_j]$ and $H_2\subseteq G_k[\{x\} \cup A_k]$ with $\diam(H_1),\diam(H_2)\leq 2$ (in this case, stars centered at $x$) which form a monochromatic cover.  So we assume $A_i \neq \emptyset$ for all $i\in[3]$. 

For distinct $i,j \in [3]$, define $B_{ij}$ to be the set of vertices $v \in A_i$ such that $v$ sends no edge of color $j$ to $A_j$.

\begin{figure}[ht]
\centering
\includegraphics{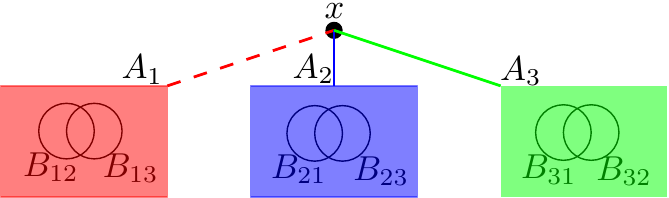}
\caption{Set up for the proof of Theorem \ref{thm:r3complete}}\label{fig:r3diam}
\end{figure}

First suppose $B_{ij} = \emptyset$ for some distinct $i,j \in [3]$.  Without loss of generality say $B_{12}=\emptyset$. This means every vertex in $A_1$ sends an edge of color 2 to $A_2$.  Thus there are trees $H_1\subseteq G_2[\{x\} \cup A_2 \cup A_1]$ and $H_2\subseteq G_3[\{x\} \cup A_3]$ with $\diam(H_1)\leq 4$ and $\diam(H_2)\leq 2$ which form a monochromatic cover.  So suppose $B_{ij} \neq \emptyset$ for all distinct $i,j \in [3]$.  Note that for all distinct $i,j,k\in [3]$,  $[B_{ij},B_{ji}]$ is a complete bipartite graph of color $k$. 

Next, suppose there exist distinct $i,j,k \in [3]$ such that $B_{ij} \setminus B_{ik} \neq \emptyset$.  Without loss of generality say $B_{12}\setminus B_{13}\neq \emptyset$ and let $z\in B_{12}\setminus B_{13}$.  Then there is a vertex $u\in A_3$ such that $zu$ is color $3$. Since every $z,B_{21}$-edge is color $3$, there are trees $H_1\subseteq G_3[\{x\} \cup A_3 \cup \{z\} \cup B_{21}]$ and $H_2\subseteq G_1[\{x\} \cup A_1 \cup (A_2 \setminus B_{21})]$ with $\diam(H_1),\diam(H_2)\leq 4$ which form a monochromatic cover. 

Finally, suppose $B_{ij} \setminus B_{ik} = \emptyset$ for all distinct $i,j,k \in [3]$. 
Then $B_{ij} = B_{ik} =: B_i$ for all distinct $i,j,k \in [3]$. If there exists $i \in [3]$ such that $A_i \neq B_i$, say $A_1\neq B_1$, then since every vertex in $(A_2\setminus B_2)\cup (A_3\setminus B_3)$ sends an edge of color 1 to $A_1\setminus B_1$ and since $G_1[B_2 \cup B_3]$ is a complete bipartite graph in color 1, there are trees $H_1\subseteq G_1[\{x\} \cup A_1 \cup (A_2 \setminus B_2) \cup (A_3 \setminus B_3)]$ and $H_2 \subseteq G_1[B_2 \cup B_3]$ with $\diam(H_1) \leq 4$ and $\diam(H_2) \leq 3$ which form a monochromatic cover. If, on the other hand, $A_i = B_i$ for all $i \in [3]$, then since $G_1[A_2 \cup A_3]$ is a complete bipartite graph in color 1, there is a tree $H_1\subseteq G_1[\{x\} \cup A_1]$ and a tree $H_2\subseteq G_1[A_2 \cup A_3]$ with $\diam(H_1) \leq 2$ and $\diam(H_2) \leq 3$  which form a monochromatic cover. 
\end{proof}

Note that it may be possible to improve the previous result by covering with two monochromatic \emph{subgraphs} of diameter at most 3, but we cannot hope to cover with two monochromatic \emph{trees} of diameter at most 3 (see Example \ref{ex:random3}).  

\begin{conjecture}
Let $G$ be a 3-colored complete graph.  There exists a monochromatic 2-cover consisting of  subgraphs of diameter at most 3.
\end{conjecture}

\subsection{Complete bipartite graphs, \texorpdfstring{$r=2$}{r=2}}

\begin{lemma}\label{lem:bip2}
Let $G$ be a complete bipartite graph with vertex classes $X$ and $Y$. In any 2-coloring of $G$, one of the following properties holds:
\begin{enumerate}[label={(P\arabic*)}]
\item \label{prop3} There exists $x_1, x_2\in X$ such that every edge incident with $x_i$ has color $i$ or there exists $y_1, y_2\in Y$ such that every edge incident with $y_i$ has color $i$.  In this case, $G$ can be covered by a color $i$ tree of diameter at most 3 and color $(3-i)$ tree of diameter at most 2 for all $i\in [2]$.
\item \label{prop2} There are partitions $\{X_1, X_2\}$ of $X$ and $\{Y_1,Y_2\}$ of $Y$ such that $[X_1,Y_1] \cup [X_2,Y_2] = G_1$ and $[X_1,Y_2] \cup [X_2,Y_1] = G_2$.  In this case, $G_i$ can be covered by two color $i$ trees of diameter at most 3, for all $i\in [2]$.
\item \label{prop1} There exists $i\in [2]$ such that $G_i$ has diameter at most 6 and $G$ has a monochromatic 2-cover consisting of trees of diameter at most 4.
\end{enumerate}
\end{lemma} 

\begin{proof}
First suppose there exists $x_1, x_2\in X$ such that every edge incident with $x_i$ has color $i$.  Let $y\in Y$ and note that the tree consisting of all color 1 edges incident with $x_1$ or $y$ has diameter at most 3 and covers $Y$.  The star consisting of all color 2 edges incident with $y$ covers the remaining vertices in $X$ and has diameter at most 2.  So suppose that every vertex is incident with, say, a color 1 edge.  If there exists $x\in X$ such that every edge incident with $x$ has color $1$, then since every vertex is incident with an edge of color 1, we have that $G$ contains a spanning tree of color 1 and diameter at most 4 in which case we are in \ref{prop1}.  Looking ahead to a potential improvement of the upper bound on the diameter in this result, we note that in this particular case (where there exists $x\in X$ such that every edge incident with $x$ has color $1$) we can say more than just that we are in \ref{prop1}.  For any $y\in Y$, the tree consisting of all color 1 edges incident with $x$ or $y$ has diameter at most 3 and covers $Y$ and the star consisting of all color 2 edges incident with $y$ covers the remaining vertices in $X$ and has diameter at most 2.  So for the rest of the proof, suppose every vertex is incident with edges of both colors.  

Suppose both $G_1$ and $G_2$ are disconnected.  Let $\{X_1, X_2\}$ and $\{Y_1, Y_2\}$ be partitions of $X$ and $Y$ respectively such that there are no color 2 edges from $X_1$ to $Y_2$ and no color 2 edges from $X_2$ to $Y_1$.  Note that $X_i\neq \emptyset$ and $Y_i\neq \emptyset$ for all $i\in [2]$ since every vertex is incident with a color 2 edge.  Thus $[X_1, Y_2]$ and $[X_2, Y_1]$ are complete bipartite graphs of color 1. Since we are assuming $G_1$ is disconnected, both $[X_1, Y_1]$ and $[X_2, Y_2]$ are complete bipartite graphs of color 2 and thus we have \ref{prop2}.  

Finally, suppose that at least one of $G_1$ and $G_2$ is connected and recall that every vertex is incident with edges of both colors.  If $\diam(G_i)=2$ for some $i\in [2]$, then we have \ref{prop1}.  
So suppose $\diam(G_i)\geq 3$ for all $i\in [2]$.  Let $x\in X$ (without loss of generality) be a vertex which witnesses $\diam(G_1)$. Let $D_i$ be the set of vertices of distance $i$ from $x$.  Note that $D_0=\{x\}$, $D_i\subseteq X$ if $i$ is even, and $D_i\subseteq Y$ if $i$ is odd.  Also note that for all $i,j\geq 0$, every edge in $[D_{2i}, D_{2j+1}]$ has color 2 if and only if $|2j+1-2i|\geq 3$.
Finally, note that if $\diam(G_1)\leq 4$, then since every vertex sees an edge of color 2, it must be the case that every vertex in $D_2$ sends a color 2 edge to $D_1\cup D_3$ ($\star$).

Now, if $\diam(G_1)=3$, then $G_1[D_0\cup D_1\cup D_2]$ is covered by a tree of radius 2 and $G_2[D_0\cup D_3]$ is covered by a tree of radius 1. 
If $\diam(G_1)=4$, then by ($\star$) $G_2$ is covered by two trees, each of radius at most 2.  If $\diam(G_1)=5$, then since every edge in $[D_0, D_3\cup D_5]$, $[D_2, D_5]$, and $[D_1, D_4]$ has color 2, $G_2$ is covered by two trees both of diameter at most 3 (a double star with central edge in $[D_0, D_5]$ and a double star with central edge in $[D_1, D_4]$).  Finally, if $\diam(G_1)\geq 6$, then $G_2$ is covered by one tree of radius at most 3 (with central vertex in $D_3$), and for similar reasons as in the case when $\diam(G_1)=5$, two trees each of diameter at most 3 (a double star with central edge in $[D_0, D_5]$ and a double star with central edge in $[D_1, D_4]$).
\end{proof}

Note that in the above proof, the only case in which we are not able to get a monochromatic 2-cover consisting of subgraphs of diameter at most 3 is the case where say $G_1$ has diameter 4 and $G_{2}$ can be covered by at most two subgraphs (trees), each of diameter at most 4. So if it were the case that $D_2(G)\geq 4$ for some $G\in \mathcal{K}_2$, then the example would have the property that both $G_1$ and $G_2$ have diameter exactly 4.

Note that Theorem \ref{thm:r2bipartite} is a direct corollary of Lemma \ref{lem:bip2}.

Later we will want to use a simpler version of Lemma \ref{lem:bip2} which doesn't make reference to the diameter of the subgraphs in the specific cases.  

\begin{lemma}\label{2colorbip_nodiam}
Let $G$ be a complete bipartite graph with vertex classes $X$ and $Y$. In any 2-coloring of $G$, one of the following properties holds:
\begin{enumerate}[label={(P\arabic*$'$)}]
\item \label{p1'} There exists $x_1, x_2\in X$ such that every edge incident with $x_i$ has color $i$ (in which case we say $Y$ is the double covered side) or there exists $y_1, y_2\in Y$ such that every edge incident with $y_i$ has color $i$ (in which case we say that $X$ is the double covered side). 
\item \label{p2'} There are partitions $\{X_1, X_2\}$ of $X$ and $\{Y_1,Y_2\}$ of $Y$ such that $[X_1,Y_1] \cup [X_2,Y_2] = G_1$ and $[X_1,Y_2] \cup [X_2,Y_1] = G_2$.
\item \label{p3'} There exists $i\in [2]$ such that $G_i$ is connected.
\end{enumerate}
\end{lemma}

\begin{figure}[ht]
\begin{subfigure}[t]{.5\textwidth}
  \centering
 \includegraphics{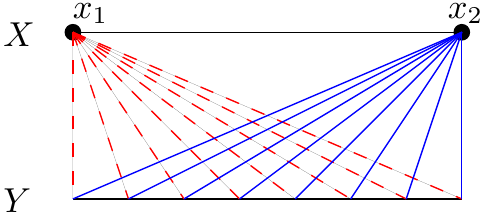} 
\caption{\ref{p1'} where $Y$ is the double covered side}\label{fig:P1}
\end{subfigure}
\begin{subfigure}[t]{.5\textwidth}
   \centering
 \includegraphics{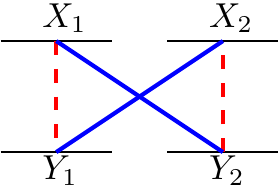} 
\caption{\ref{p2'}}\label{fig:P2}
\end{subfigure}
\caption{Two cases from Lemma \ref{2colorbip_nodiam}}\label{fig:P'}
\end{figure}

\subsection{Graphs with \texorpdfstring{$\alpha(G)=2$}{a(G)=2}, \texorpdfstring{$r=2$}{r=2}}

\begin{proof}[Proof of Theorem \ref{thm:r2a2}]
Let $\{x,y\}$ be an independent set in $G$.  Since $\alpha(G)=2$, every vertex in $V(G) \setminus \{x,y\}$ is adjacent to a vertex in $\{x,y\}$.  So we let $A_x = N(x) \setminus N(y)$, $A_y = N(y) \setminus N(x)$, and $A = N(x) \cap N(y)$. Let $A_{ij} = A \cap N_i(x) \cap N_j(y)$ for $i,j \in [2]$.
Note that $\alpha(G[A_x \cup \{x\}]) = 1$ and $\alpha(G[A_y \cup \{y\}]) = 1$, otherwise we would have an independent set of order 3. By Proposition \ref{prop:folk}, either $\diam(G_j[A_x \cup \{x\}]) = 3$ for all $j \in [2]$, or $\diam(G_j[A_x \cup \{x\}]) \leq 2$ for some $j \in [2]$.  Likewise, either $\diam(G_j[A_y \cup \{y\}]) = 3$ for all $j \in [2]$, or $\diam(G_j[A_y \cup \{y\}]) \leq 2$ for some $j \in [2]$.  Now we consider three cases.

\begin{figure}[ht]
\centering
\includegraphics{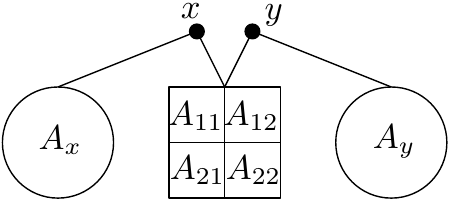}
\caption{Set-up for the proof of Theorem \ref{thm:r2a2}}\label{fig:r2a2diam}
\end{figure}

\noindent
\textbf{Case 1.} ($\diam(G_j[A_x \cup \{x\}]) = 3$ for all $j \in [2]$, or $\diam(G_j[A_y \cup \{y\}]) = 3$ for all $j \in [2]$.)

Without loss of generality, suppose that $\diam(G_j[A_x \cup \{x\}]) = 3$ for all $j \in [2]$.  Also without loss of generality, suppose that $\diam(G_1[A_y \cup \{y\}]) \leq 3$.
If $A_{11} \neq \emptyset $, then $H_1 = G_1[\{x\} \cup (A \setminus A_{22}) \cup \{y\} \cup A_y]$ and $H_2 = G_2[A_x\cup \{x\} \cup A_{22}]$ cover $G$, with $\diam(H_1) \leq 6$ and $\diam(H_2) \leq 4$. If $A_{11} = \emptyset$ and $A_{22} \neq \emptyset$, then $H_1 = G_2[A_x \cup \{x\} \cup A \cup \{y\}]$ and $H_2 = G_1[A_y \cup \{y\}]$ cover $G$, with $\diam(H_1) \leq 6$ and $\diam(H_2) \leq 3$. If $A_{11} = \emptyset$ and $A_{22} = \emptyset$, then $H_1 = G_1[A_x \cup \{x\} \cup A_{12}]$ and $H_2 = G_1[A_y \cup \{y\} \cup A_{21}]$ cover $G$, with $\diam(H_1), \diam(H_2) \leq 4$.

\noindent
\textbf{Case 2.} ($\diam(G_i[A_x \cup \{x\}]) \leq 2$ for some $i \in [2]$ and $\diam(G_j[A_y \cup \{y\}]) \leq 2$ for some $j \in [2]$.)

Without loss of generality, suppose $\diam(G_1[A_x \cup \{x\}]) \leq 2$.

\textbf{Case 2.1.} ($\diam(G_1[A_x \cup \{x\}]) \leq 2$ and $\diam(G_2[A_y \cup \{y\}]) \leq 2$.)

If $A_{ii} \neq \emptyset$ for some $i \in [2]$, say $A_{11}\neq \emptyset$, then $H_1 = G_1[A_x \cup \{x\} \cup \{y\} \cup A \setminus A_{22}]$ and $H_2 = G_2[A_y \cup \{y\} \cup A_{22}]$ cover $G$ with $\diam(H_1) \leq 5$, and  $\diam(H_2) \leq 3$. So assume $A_{ii} = \emptyset$ for all $i \in [2]$.  Notice that if any $A_x,A_{21}$-edge were color 1, then $H_1 = G_1[\{x,y\} \cup A_x \cup A_{21}]$ and $H_2 = G_2[\{y\}\cup A_{12} \cup A_y]$ cover $G$ with $\diam(H_1)\leq 5$ and $\diam(H_2)\leq 3$. So assume that every $A_x,A_{21}$-edge is color 2. Likewise we can assume that every $A_y,A_{21}$-edge is color 1. Let $Z_1 = \{v \in A_x : v \notin N(A_{21})\}$ and $Z_2 = \{v \in A_y : v \notin N(A_{21})\}$.  If $Z_i=\emptyset$ for some $i\in [2]$, say $Z_1=\emptyset$, then $H_1=G_2[\{x\}\cup A_x\cup A_{21}]$ and $H_2=G_2[\{y\}\cup A_y\cup A_{12}]$ cover $G$ with $\diam(H_1)\leq 4$ and $\diam(H_2)\leq 3$.  So suppose $Z_i\neq\emptyset$ for all $i\in [2]$.  Since there are no edges from $Z_1\cup Z_2$ to $A_{21}$, we have that $[Z_1, Z_2]$ is a complete bipartite graph, which implies $G[Z_1\cup Z_2]$ is a complete graph.  By Proposition \ref{prop:folk} we have, say $\diam(G_1[Z_1\cup Z_2])\leq 3$, and thus $H_1=G_1[\{x\}\cup A_x\cup A_{12}\cup Z_2]$ and $H_2=G_1[\{y\}\cup A_{21}\cup (A_y\setminus Z_2)]$ cover $G$ with $\diam(H_1)\leq 6$ and $\diam(H_2)\leq 4$.

\textbf{Case 2.2} ($\diam(G_1[A_x \cup \{x\}]) \leq 2$ and $\diam(G_1[A_y \cup \{y\}]) \leq 2$.)

If $A_{11} \neq \emptyset$, then $H_1 = G_1[\{x,y\} \cup A_x \cup A_y \cup (A \setminus A_{22})]$ and $H_2 = G_2[\{x\} \cup A_{22}]$ cover $G$ with $\diam(H_1)\leq 6$ and $\diam(H_2)\leq 2$. If $A_{11} = \emptyset$ and $A_{22} = \emptyset$, then $H_1 = G_1[\{x\} \cup A_x \cup A_{12}]$ and $H_2 = G_1[\{y\} \cup A_y \cup A_{21}]$ cover $G$ with diameters at most 3. So suppose $A_{11} = \emptyset$ and $A_{22} \neq \emptyset$. If every vertex in $A_{22}$ sends a color 1 edge to $A_x\cup A_y$ then let $U_i = \{v \in A_{22} : \exists u \in A_i, vu \in E(G_1)\}$ for $i \in [2]$. Then $H_1 = G_1[\{x\} \cup A_x \cup A_{12} \cup U_1]$ and $H_2 = G_1[\{y\} \cup A_y \cup A_{21} \cup (U_2\setminus U_1)]$ cover $G$ with diameters at most 4. So suppose there exists a vertex $w \in A_{22}$ such that every $w,A_i$-edge is color 2 for all $i \in [2]$. Let $Z = \{v \in A_x \cup A_y : wv \notin E(G)\}$. Since $\alpha(G) = 2$, $G[Z]$ is a complete graph, so there exists $i \in [2]$ for which $H_1 = G_i[Z]$ has diameter at most 3. Then $H_1$ and $H_2 = G_2[\{x,y\} \cup A \cup N_2(x)]$ cover $G$ with $\diam(H_2) \leq 4$.
\end{proof}

\subsection{Complete graphs, \texorpdfstring{$r=4$}{r=4}}

\begin{proof}[Proof of Theorem \ref{thm:r4complete}]
Let $x \in V(K)$. For $i \in [4]$, let $A_i$ be the neighbors of $x$ of color $i$. If $A_i = \emptyset$ for some $i \in [4]$, say $i = 4$, then letting $H_j = G_j[\{x\} \cup A_j]$ satisfies the theorem with $\diam(H_j) \leq 2$. Therefore, assume $A_i \neq \emptyset$ for all $i \in [4]$. 
For all distinct $i,j \in [4]$, define $B_{ij}$ to be the set of vertices $v \in A_i$ such that $vu$ is not color $j$ for all $u \in A_j$.  If $B_{ij} = \emptyset$ for some $i,j \in [4]$, then $H_1 = G_j[\{x\} \cup A_j \cup A_i]$, $H_2 = G_k[\{x\} \cup A_k]$ and $H_3 = G_l[\{x\} \cup A_l]$ cover $V(K)$, where $\diam(H_1)\leq 4$ and $\diam(H_2), \diam(H_3) \leq 2$. So suppose $B_{ij} \neq \emptyset$ for all $i,j \in [4]$.  Note that $[B_{ij},B_{ji}]$ is a complete bipartite graph in colors $[4] \setminus \{i,j\}$.

\begin{figure}[ht]
\centering
\includegraphics{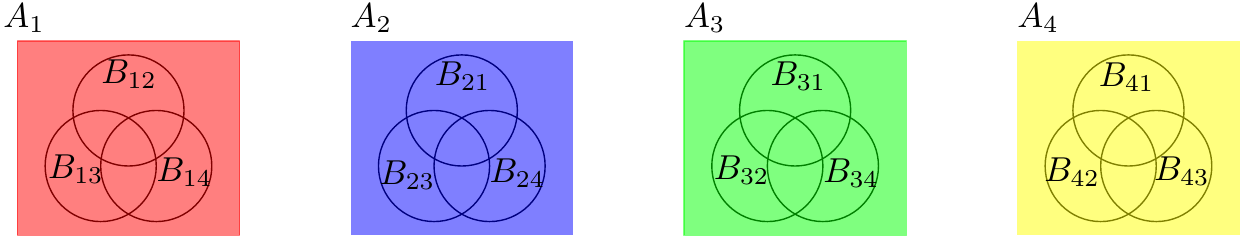}
\caption{Set-up for the proof of Theorem \ref{thm:r4complete}}\label{fig:r4diam}
\end{figure}

We first establish two claims.

\begin{claim}
We are done unless:
\begin{enumerate}[label={(C\arabic*)}]
    \item\label{c1} For all distinct $i,j,k,l \in [4]$, $B_{ij}\cap B_{ik}\cap B_{il} \neq \emptyset$, and 
    \item\label{c2} For all distinct $i,j,k,l \in [4]$, $B_{ij} \setminus (B_{ik} \cup B_{il}) = \emptyset$, and
    \item\label{c3} For all distinct $i,j,k,l \in [4]$,  $(B_{ik}\cup B_{il})\setminus B_{ij}= \emptyset$ or $(B_{ki}\cup B_{kj})\setminus B_{kl}= \emptyset$.
\end{enumerate}
\end{claim}

\begin{proof}
\begin{enumerate}[label={(C\arabic*)}]
 
\item Suppose there exists $i \in [4]$ such that $B_{ij} \cap B_{ik} \cap B_{il} = \emptyset$.  Then $H_1 = G_j[\{x\} \cup A_j \cup (A_i \setminus B_{ij})]$, $H_2 = G_k[\{x\} \cup A_k \cup (A_i \setminus B_{ik})]$ and $H_3 = G_l[\{x\} \cup A_l \cup (A_i \setminus B_{il})]$ cover $V(K)$, where $\diam(H_1), \diam(H_2), \diam(H_3) \leq 4$. 

\item Suppose there exist distinct $i,j,k,l \in [4]$ such that $B_{ij} \setminus (B_{ik} \cup B_{il}) \neq \emptyset$.  Then for a fixed $u \in B_{ij} \setminus (B_{ik} \cup B_{il})$, since every edge in $[B_{ij}, B_{ji}]$ is color $k$ or $l$ we have that $B_{ji}\subseteq N_{k}(u)\cup N_{l}(u)$. Since $u$ sends a color $k$ edge to $A_k$ and a color $l$ edge to $A_l$, $H_1 = G_i[\{x\} \cup A_i \cup (A_j \setminus B_{ji})]$, $H_2 = G_k[\{x\} \cup A_k \cup \{u\} \cup (N_k(u)\cap B_{ji})]$ and $H_3 = G_l[\{x\} \cup A_l \cup \{u\} \cup (N_l(u)\cap B_{ji})]$ cover $V(K)$, where $\diam(H_1), \diam(H_2), \diam(H_3) \leq 4$.

\item Suppose there exist distinct $i,j,k,l \in [4]$ such that $(B_{ik}\cup B_{il})\setminus B_{ij}\neq \emptyset$ and $(B_{ki}\cup B_{kj})\setminus B_{kl}\neq \emptyset$. Then by \ref{c2}, we have that $B_{ij}\subseteq B_{ik}=B_{il}$ and $B_{kl}\subseteq B_{ki}=B_{kj}$.  Let $u_i\in B_{ik}\setminus B_{ij}$ and $u_k\in B_{ki}\setminus B_{kl}$.  Note that since every edge in  $[B_{ik}, B_{ki}]$ has color $j$ or $l$, we may suppose without loss of generality that $u_iu_k$ has color $j$.  Also every edge from $u_k$ to $B_{ik}$ has color $j$ or $l$.  Since $u_k\not\in B_{kl}$, $u_k$ sends an edge of color $l$ to $A_l$ and since $u_i\not\in B_{ij}$, $u_i$ sends an edge of color $j$ to $A_j$.  Thus letting $H_1=G_k[A_i\setminus B_{ik}\cup A_k\cup \{x\}]$, $H_2=G_j[\{u_i,u_k\}\cup A_j\cup \{x\}\cup (N_j(u_k)\cap B_{ik})]$, and $H_3=G_l[A_l\cup \{x\}\cup \{u_k\}\cup (N_l(u_k)\cap B_{ik})]$ gives a cover of $V(K)$ with $\diam(H_1)\leq 4$, $\diam(H_2)\leq 5$, $\diam(H_3)\leq 4$. \qedhere
\end{enumerate}
\end{proof}

\begin{claim}\label{lastcase}
If \ref{c1}-\ref{c3} hold, then there exists distinct $i,j,k,l\in [4]$ such that $B_{ij}=B_{ik}$, $B_{jk}\cup B_{jl}\subseteq B_{ji}$, and $B_{kj}\cup B_{kl}\subseteq B_{ki}$.
\end{claim}

\begin{proof}
Using \ref{c3}, we have that for every ordering $(i,j,k,l)$ of the elements $\{1,2,3,4\}$, we have $B_{ik}\cup B_{il}\subseteq B_{ij}$ or $B_{ki}\cup B_{kj}\subseteq B_{kl}$.  From this we can deduce that there exists $i,j\in [4]$ such that $B_{ik}\cup B_{il}\subseteq B_{ij}$ and $B_{jk}\cup B_{jl}\subseteq B_{ji}$.  Likewise, since either $B_{ij}\cup B_{il}\subseteq B_{ik}$ or $B_{ji}\cup B_{jk}\subseteq B_{jl}$, we may assume that $B_{ij}=B_{ik}$.  We will have the desired property unless $B_{ki}$ is a proper subset of $B_{kj}\cup B_{kl}$. But this implies $B_{ji}\cup B_{jk}\subseteq B_{jl}$ and $B_{li}\cup B_{lk}\subseteq B_{lj}$.  So $B_{ji}=B_{jl}$ and $B_{ik}\cup B_{il}\subseteq B_{ij}$ and $B_{li}\cup B_{lk}\subseteq B_{lj}$ giving us the desired property.  
\end{proof}

We have one final case remaining from Claim \ref{lastcase}.  We use Lemma \ref{lem:bip2} to show that we are done in this final case.   

Without loss of generality, suppose $B_{14}\subseteq B_{12}=B_{13}$ and $B_{23}\cup B_{24}\subseteq B_{21}$ and $B_{32}\cup B_{34}\subseteq B_{31}$.  First, we let $H_1=G_4[\{x\}\cup A_4\cup (A_1\setminus B_{12})\cup (A_2\setminus B_{21})\cup (A_3\setminus B_{31})]$ and note that $\diam(H_1)\leq 4$ and $H_1$ covers everything except $B_{12}=B_{13}$, $B_{21}$ and $B_{31}$.  

Let $F_1=[B_{12}, B_{21}]$ and $F_2=[B_{13}, B_{31}]$ and note that all edges in $F_1$ are color $3,4$ and all edges in $F_2$ are color $2,4$.  We apply Lemma \ref{lem:bip2} to each of $F_1$ and $F_2$.  If say $F_2$ satisfies \ref{prop1}, then letting $H_2$ be the monochromatic subgraph covering $F_2$, and $H_3=G_2[\{x\}\cup B_{21}]$, we have the desired cover of $V(K)$ with $\diam(H_1)\leq 4$, $\diam(H_2)\leq 6$, and $\diam(H_3)\leq 2$.  So suppose neither $F_1$ nor $F_2$ satisfy \ref{prop1}. If both $F_1$ and $F_2$ satisfy \ref{prop2}, then $F_1 \cup F_2$ can be covered with at most two monochromatic subgraphs of color 4. If $F_1 \cup F_2$ can be covered with exactly one monochromatic subgraph of color 4, let $H_2$ be this subgraph and we have the desired cover of $V(K)$ with $\diam(H_1)\leq 4$ and $\diam(H_2)\leq 6$. If $F_1 \cup F_2$ must be covered with two monochromatic subgraphs of color 4, let $H_2, H_3$ be these subgraphs and we have the desired cover of $V(K)$ with $\diam(H_1)\leq 4$, $\diam(H_2)\leq 4$, and $\diam(H_3)\leq 2$.

Now suppose say $F_2$ satisfies \ref{prop2}. If $F_1$ satisfies \ref{p1'} where $B_{12}$ is the double covered side, then letting $H_2$ be the nontrivial color 3 subgraph of $F_1$ and letting $H_3$ be the color 4 subgraph which covers the rest of $F_1 \cup F_2$, we have the desired cover of $V(K)$ with $\diam(H_1)\leq 4$, $\diam(H_2)\leq 3$, and $\diam(H_3)\leq 4$. If $F_1$ satisfies \ref{p1'} where $B_{21}$ is the double covered side, then letting $H_2$ be the color 4 subgraph that covers $B_{21}$ along with the the color 4 subgraph of $F_2$ which it intersects and letting $H_3$ be the other color 4 subgraph of $F_2$, we have the desired cover of $V(K)$ with $\diam(H_1)\leq 4$, $\diam(H_2)\leq 4$, and $\diam(H_3)\leq 3$.

We may now suppose both $F_1$ and $F_2$ satisfy \ref{p1'}. If say $B_{21}$ is the double covered side of $F_1$ and $B_{12} = B_{13}$ is the double covered side of $F_2$, then letting $H_2$ be the color 4 subgraph which covers $F_1$ and letting $H_3 = G_3[\{x\} \cup B_{31}]$, we have the desired cover of $V(K)$ with $\diam(H_1)\leq 4$, $\diam(H_2)\leq 3$, and $\diam(H_3)\leq 2$. Now suppose $B_{21}$ is the double covered side of $F_1$ and $B_{31}$ is the double covered side of $F_2$. If the two nontrivial color 4 components intersect, then letting $H_2$ be the resulting color 4 subgraph and letting $H_3 = G_1[\{x\} \cup B_{12}]$, we have the desired cover of $V(K)$ with $\diam(H_1)\leq 4$, $\diam(H_2)\leq 6$, and $\diam(H_3)\leq 2$. If the color 4 components do not intersect, then every vertex in the nontrivial color 4 component of $F_1$ sends only color 2 edges to $B_{31}$ and every vertex in the nontrivial color 4 component of $F_2$ sends only color 3 edges to $B_{21}$. Thus letting $H_2$ be the resulting color 2 subgraph and letting $H_3$ be the resulting color 3 subgraph, we have the desired cover of $V(K)$ with $\diam(H_1)\leq 4$, $\diam(H_2)\leq 3$, and $\diam(H_3)\leq 3$.

We now claim that we are done unless both $F_1$ and $F_2$ satisfy \ref{p1'} where $B_{12}$ is the double covered side in each case.  Suppose we are in this case.  For $i\in \{2,3\}$, let $B_i'$ be the vertices in $B_{i1}$ which are in the component of color 4 and let $B_i''$ be the vertices in $B_{i1}\setminus B_i'$ which send an edge of color 4 to $B_{3-i}'$.  Finally, let $\bar{B}_i=B_{i1}\setminus (B_i'\cup B_i'')$.  If there were an edge of color 3 from $\bar{B}_2$ to $B_{31}$ we would be done, and if there were an edge of color 2 from $\bar{B}_3$ to $B_{21}$ we would be done.  So all edges in $[\bar{B}_2, \bar{B}_3]$ are color $1,4$, all edges in $[\bar{B}_2, B_3']$ are color $1,2$, and all edges in $[\bar{B}_3, B_2']$ are color $1,3$.  If every vertex in $B_2'$ sent an edge of color 3 to $\bar{B}_3$, then we would be done, so suppose there exists a vertex $y\in B_2'$ such that $y$ only sends edges of color $1$ to $\bar{B}_3$.  Likewise, there exists a vertex $z\in B_3'$ such that $z$ only sends edges of color $1$ to $\bar{B}_2$.  If there is an edge of color 1 between $\bar{B}_2$ and $\bar{B}_3$, then we would be done using the component of color 4 and the component of color 1.  Otherwise all edges between $\bar{B}_2$ and $\bar{B}_3$ are color 4, and we are done using two components of color 4.
\end{proof}

\subsection{Complete graphs, \texorpdfstring{$r=5$}{r=5}}

We note that in order to generalize of proof of Theorem \ref{thm:r4complete} to prove Conjecture \ref{con:mil} for $r=5$, it would be helpful to solve the following problem which is analogous to the last (main) case in the proof of Theorem \ref{thm:r4complete}.

\begin{conjecture}\label{prob:r=5}
Suppose  $G$ is a complete 4-partite graph with vertex set partitioned as $\{V_1, V_2, V_3, V_4\}$.  If $[V_i, V_j]$ is colored with colors from $[5]\setminus \{i,j\}$ for all distinct $i,j\in [4]$, then $G$ can be covered with at most three monochromatic subgraphs of bounded diameter.
\end{conjecture}

\subsection{Complete bipartite graphs, \texorpdfstring{$r=3$}{r=3}}

\begin{proof}[Proof of Theorem \ref{thm:r3bipartite}]
Suppose there exists a component $C$ of color $\chi\in [3]$ such that $V_i\setminus V(C)\neq \emptyset$ for all $i\in [2]$. Then since each of $[V_1\setminus C, V_2\cap C]$ and $[V_2\setminus C, V_1\cap C]$ are colored with colors from $[3]\setminus \{\chi\}$, we are done by applying Lemma \ref{lem:bip2} to each of $[V_1\setminus C, V_2\cap C]$ and $[V_2\setminus C, V_1\cap C]$.  So for every $\chi\in [3]$ and every component $C$ of color $\chi$ there exists $i\in [2]$ such that $V_i\subseteq V(C)$.  If there exists a component $C$ such that $V_i\subseteq V(C)$ for some $i\in [2]$ and $\diam(C)\leq 5$, then we are done either because $V_{3-i}\setminus V(C)=\emptyset$ or by applying Lemma \ref{lem:bip2} to $[V_{3-i}\setminus V(C), V_i]$.  So finally, suppose that there exists a component $C$, say of color $3$, such that $V_i\subseteq V(C)$ for some $i\in [2]$, say $i=2$, and $\diam(C)\geq 6$.  Let $v\in V(G)$ which witnesses $\diam(C)$ and for all $0\leq i\leq d$, let $D_i$ be the vertices at distance $i$ from $v$ in $C$.  Now let $X_1=D_0\cup D_2$, $X_2=(V_1\setminus V(C))\cup D_4\cup \dots \cup D_{2\floor{d/2}}$, $Y_2=D_1$, $Y_1=D_5\cup \dots \cup D_{2\ceiling{d/2}-1}$, and $Y_0=D_3$.  Note that $[X_i, Y_i]$ is a complete $[2]$-colored bipartite graph for all $i\in [2]$ and $[Y_0, (V_1\setminus V(C))\cup D_0\cup D_6\cup \dots \cup D_{2\floor{d/2}}]$ is a complete $[2]$-colored bipartite graph.  Also note that $C^*=C[D_0\cup D_1\cup D_2\cup D_3]$ is a subgraph of color 3 which has diameter at most 6 and has $X_1\cup Y_0\cup Y_2=D_0\cup D_1\cup D_2\cup D_3\subseteq V(C^*)$.  

We must now analyze cases based on what happens with $[X_i, Y_i]$ for each $i\in [2]$.  First, assume that for some $i \in [2]$, $[X_i,Y_i]$ satisfies \ref{prop1} in some color $j\in[2]$.  Using that $C^*$ covers $Y_0$ and applying Lemma \ref{lem:bip2} to $[X_{3-i}, Y_{3-i}]$ completes this case.  

Now assume \ref{p1'} holds for $[X_i, Y_i]$ for some $i \in [2]$ where $X_i$ is the double covered side.  Since for any $x \in X_i \setminus (D_2\cup D_4)$, for every $y \in Y_0$, the edge $xy$ exists and is color 1 or 2, let $H_1 = G_1[X_i \cup Y_i \cup Y_0]$ and $H_2 = G_2[X_i \cup Y_i \cup Y_0]$. These two subgraphs cover all of $X_i\cup Y_i\cup Y_0$ each with diameter at most 4. Lemma \ref{lem:bip2} applies to $[X_{3-i}, Y_{3-i}]$ for $H_3$ and $H_4$ each with diameter at most 6. 

Now assume \ref{p1'} holds for $[X_1, Y_1]$ where $Y_1$ is the double covered side.  Let $H_1$ be the color 1 component from $[X_1, Y_1]$ which covers $Y_1$.  Using that $C^*$ covers $X_1\cup Y_0$ and applying Lemma \ref{lem:bip2} to $[X_2, Y_2]$ completes this case.  

Now assume \ref{prop2} holds for $[X_1, Y_1]$ and \ref{p1'} holds for $[X_2, Y_2]$ where $Y_2$ is the double covered side.  Let $Y_0'=N_1(v)\cap Y_0$ and $Y_0''=N_2(v)\cap Y_0$.  If there exists a vertex $x'$ in $X_2\setminus D_4$ which is in both the color 1 and color 2 component of $[X_2,Y_2]$, then we are done since every vertex in $Y_0$ is adjacent to $x'$ in color 1 or 2.  So let $x'\in X_2\setminus D_4$ and suppose without loss of generality that $x'$ is only in the color 1 component.  If $x'$ sends a color 1 edge to $Y_0'$, then we are done.  Otherwise, $x'$ only sends color 2 edges to $Y_0'$.  So if $x'$ sends a color 2 edge to $Y_0''$, then we are done by using the two color 2 subgraphs from $[X_1, Y_1]$, as one of these subgraphs has now been extended to cover all of $Y_0$; or else $x'$ only sends color 1 edges to $Y_0''$, in which case we are done by using the two color 1 subgraphs from $[X_1, Y_1]$, one of which contains $Y_0'$, together with the color 1 subgraph from $[X_2, Y_2]$, which contains $Y_0''$, and the color 2 subgraph from $[X_2, Y_2]$.  

Lastly, assume \ref{prop2} holds for both $[X_1,Y_1]$ and $[X_2,Y_2]$.  Let $Y_0'=N_1(v)\cap Y_0$ and $Y_0''=N_2(v)\cap Y_0$.  If any vertex in $Y_0'$ sends an edge of color 1 to $X_2$, then we are done.  Otherwise there is a vertex in $X_2$ which only sends color 2 edges to $Y_0'$ in which case we are done.  
\end{proof}

\section{Monochromatic covers of complete multipartite graphs}\label{sec:multi}

In this section, we prove Conjecture \ref{con:rpart2} for $r\leq 4$.  Let $\mathcal{K}_k$ be the family of complete $k$-partite graphs.  Lemma \ref{2colorbip_nodiam} implies the following (which was already known by \cite{CFGLT}, and was almost certainly a folklore result before that).

\begin{theorem}\label{2partite}
For all $k\geq 2$ and all $K\in \mathcal{K}_k$, $\tc_2(K)\leq 2$.
\end{theorem}

Now we consider complete 3-partite graphs.

\begin{theorem}\label{3partite}
For all $k\geq 3$ and all $K\in \mathcal{K}_k$, $\tc_3(K)\leq 3$.
\end{theorem}

\begin{proof}
Let $\{V_1, V_2, V_3\}$ be the tripartition of $K$ (we may assume $K$ is 3-partite).  First suppose there exists a monochromatic component $C$, say of color 3, which covers, say $V_3$.  Then either $C$ covers all of $V(K)$ and we are done, or $K[(V_1\cup V_2)\setminus V(C), V_3]$ is a complete 2-colored bipartite graph and thus can be covered by two monochromatic components and we are done.  So suppose for the remainder of the proof that for all monochromatic components $C$ and all $i\in [3]$, $V_i\setminus V(C)\neq \emptyset$ ($\star$).  

\begin{claim}\label{clm:C1}
There exists a monochromatic component $C$ so that $V_i\cap V(C)\neq\emptyset$ for all $i\in [3]$. 
\end{claim}

\begin{proof}
Let $C$ be a monochromatic component and without loss of generality, suppose $V(C)\cap V_3=\emptyset$.  Let $K'=K[V_3, (V_1\cup V_2)\cap V(C)]$ be the induced 2-colored complete bipartite graph, say the colors are red and blue.  We apply Lemma \ref{2colorbip_nodiam} to $K'$.  By ($\star$), $K'$ cannot satisfy \ref{p3'}.  If \ref{p1'} is the case, then by ($\star$), it cannot be that $V_3$ is the double covered side, so $(V_1\cup V_2)\cap V(C)$ is the double covered side and thus we have a monochromatic component which has nontrivial intersection with all three parts.  So finally, suppose \ref{p2'} is the case, and let $\{X_1, X_2\}$ be the corresponding bipartition of $V_3$ and let $\{Y_1, Y_2\}$ be the corresponding bipartition of $(V_1\cup V_2)\cap V(C)$.  If for some $i\in [2]$, $Y_i\cap V_1\neq \emptyset$ and $Y_i\cap V_2\neq \emptyset$, then we have a monochromatic component which has nontrivial intersection with all three parts.  Otherwise we have, without loss of generality, $Y_1=V_1\cap V(C)$ and $Y_2=V_2\cap V(C)$.  Every edge in $[X_1, Y_2]$ is say blue and every edge in $[X_2, Y_2]$ is then red.  Since every edge from $Y_2$ to $V_1\setminus Y_1$ is either red or blue, this gives us a monochromatic component which has nontrivial intersection with all three parts.
\end{proof}

Now by Claim \ref{clm:C1}, there exists a monochromatic component $C$ so that $V_i\cap V(C)\neq\emptyset$ for all $i\in [3]$.  Let $X_1=V_1\cap V(C)$, $X_2=V_1\setminus V(C)$, $Y_1=V_2\cap V(C)$, $Y_2=V_2\setminus V(C)$, $Z_1=V_3\cap V(C)$, $Z_2=V_3\setminus V(C)$ and note that by $(\star)$, all of these sets are non-empty.  Note that the sets $X_1, Y_2, Z_1, X_2, Y_1, Z_2$ form a 2-colored (say red and blue) blow-up of a $C_6$.  In the remainder of the proof, we implicity prove the general result that a 2-colored blow-up of a $C_6$ can be covered by at most 3 monochromatic components.

\begin{claim}\label{clm:C2}
If there exists a monochromatic component covering any of $X_1\cup Y_2\cup Z_1$, $X_2\cup Y_2\cup Z_1$, $X_2\cup Y_1\cup Z_1$, $X_2\cup Y_1\cup Z_2$, $X_1\cup Y_1\cup Z_2$, or $X_1\cup Y_2\cup Z_2$, then we have a monochromatic 3-cover.  
\end{claim}

\begin{proof}
Suppose without loss of generality that there is a monochromatic component covering $X_1\cup Y_2\cup Z_1$.  Then since $[Y_1, X_2\cup Z_2]$ is a 2-colored complete bipartite graph, we are done by Theorem \ref{2partite}.
\end{proof}

We begin by focusing on the 2-colored (say red and blue) complete bipartite graphs $K_1=[Z_1, X_2\cup Y_2]$ and $K_2=[Z_2, X_1\cup Y_1]$, but note that $[X_1, Y_2]$ and $[X_2, Y_1]$ are also 2-colored complete bipartite graphs colored with red and blue.  We apply Corollary \ref{2colorbip_nodiam} to each of $K_1$ and $K_2$.  

\noindent
\textbf{Case 1} ($K_1$ or $K_2$ satisfies \ref{p3'})
Say $K_1$ satisfies \ref{p3'}.  Since $K_2$ can be covered by at most two monochromatic components, we are done; thus we may assume that \ref{p3'} is never the case.

\noindent
\textbf{Case 2} ($K_1$ or $K_2$ satisfies \ref{p2'})

Without loss of generality, say $K_1$ satisfies \ref{p2'}.  This means there are two red components covering $K_1$ and there are two blue components covering $K_1$.  

\textbf{Case 2.1} ($K_2$ satisfies \ref{p2'})

There are two red components covering $K_2$ and there are two blue components covering $K_2$.  Using the fact that $[X_1, Y_2]$ is a 2-colored complete bipartite graph, there is a, say, red edge from $X_1$ to $Y_2$.  This red edge joins one of the red components covering $K_1$ to one of the red components covering $K_2$ and thus there are at most three red components covering $K$.

\textbf{Case 2.2} ($K_2$ satisfies \ref{p1'} and $X_1\cup Y_1$ is the double covered side)

So there is a red component $R$ and a blue component $B$ which together cover $K_2$. Using the fact that $[X_1, Y_2]$ is a 2-colored complete bipartite graph, there is a, say, red edge from $X_1$ to $Y_2$.  This red edge joins one of the red components covering $K_1$ to the red component $R$ and thus there are two red components and one blue component ($B$) which together cover $K$.

\textbf{Case 2.3} ($K_2$ satisfies \ref{p1'} and $Z_2$ is the double covered side)

So there is a red component $R$ and a blue component $B$ which together cover $K_2$.  Using the fact that both $[X_1, Y_2]$ and $[X_2, Y_1]$ are 2-colored complete bipartite graphs, we either have that there is a blue edge from $B$ to $X_2\cup Y_2$ in which case there are two blue components and one red component ($R$) which together cover $K$, or every edge from $B$ to $X_2\cup Y_2$ is red and thus there are three red components which cover $K$.  

\noindent
\textbf{Case 3} ($K_1$ and $K_2$ both satisfy \ref{p1'})

In $K_1$ we have that $Z_1$ is the double covered side or $X_2\cup Y_2$ is the double covered side, and in $K_2$ we have that $Z_2$ is the double covered side or $X_1\cup Y_1$ is the double covered side.  For all $i\in [2]$, let $R_i$ and $B_i$ be, respectively, the red and blue components covering $K_i$.

We will split into two subcases.

\textbf{Case 3.1} ($X_1\cup Y_1$ is the double covered side or $X_2\cup Y_2$ is the double covered side)

Without loss of generality, say $X_2\cup Y_2$ is the double covered side in $K_1$.  If there is a blue edge from $B_2$ to $X_2\cup Y_2$, then $B_1$ and $B_2$ are contained together in a single blue component $B$ and thus $B, R_1, R_2$ forms a monochromatic cover of $K$.  So suppose every edge from $B_2$ to $X_2\cup Y_2$ is red.  So there is a red component $R$ which covers $R_1$ and $B_2\cap (V_1\cup V_2)$.  Thus $R, B_1, R_2$ forms a monochromatic cover of $K$.

\textbf{Case 3.2} ($Z_1$ is the double covered side and $Z_2$ is the double covered side)

If there is a blue edge between $B_1$ and $B_2$ or a red edge between $R_1$ and $R_2$, we would have three monochromatic components which cover $K$, so suppose every edge between $B_1$ and $B_2$ is red and every edge between $R_1$ and $R_2$ is blue ($\star \star$).

For all $i\in [2]$, let $X_i(B)=X_i\cap V(B_{3-i})$, $X_i(R)=X_i\cap V(R_{3-i})$, $Y_i(B)=Y_i\cap V(B_{3-i})$, and $Y_i(R)=Y_i\cap V(R_{3-i})$.

\begin{figure}[ht]
\begin{subfigure}[t]{.5\textwidth}
  \centering
 \includegraphics{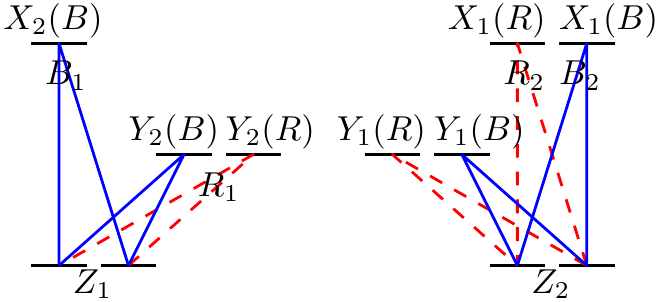}
\caption{Case 3.2.1: Some of the sets $Y_2(B)$, $Y_1(R)$, $Y_1(B)$, $X_1(R)$, $X_1(B)$ may be empty}\label{fig:321}
\end{subfigure}
\begin{subfigure}[t]{.5\textwidth}
   \centering
 \includegraphics{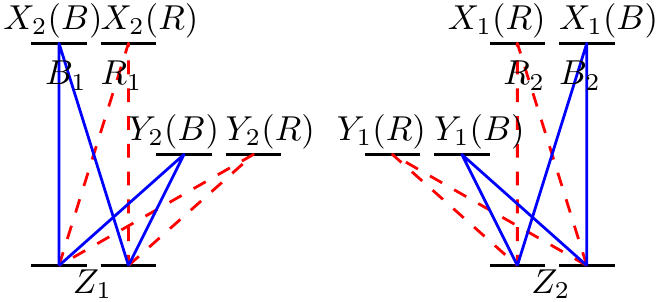}
\caption{Case 3.2.2}\label{fig:322}
\end{subfigure}
\caption{Case 3.2}\label{fig:32}
\end{figure}

\underline{Case 3.2.1} (For some $i\in [2]$, some $X_i(R), X_i(B), Y_i(R), Y_i(B)$ is empty)

Without loss of generality, suppose $X_2(R)=\emptyset$ (which implies $Y_2(R)\neq \emptyset$).  

If $Y_1(R)=\emptyset$, then we must have $X_1(R)\neq \emptyset$ and thus by ($\star\star$), we have that the complete bipartite graph $[Y_2(R), X_1(R)]$ colored in blue, together with $B_1$ and $B_2$ form a monochromatic cover.  So suppose $Y_1(R)\neq \emptyset$.

If $Y_1(B)=\emptyset$, then either every vertex in $X_2(B)$ sends a red edge to $Y_1(R)$ and thus there is red component covering $X_2\cup Y_1\cup Z_2$ and we are done by Claim \ref{clm:C2}, or there is a vertex in $X_2(B)$ which only sends blue edges to $Y_1(R)$ and thus there is a blue component covering $Z_1\cup X_2\cup Y_1$ and we are again done by Claim \ref{clm:C2}.  So suppose $Y_1(B)\neq \emptyset$.  

By ($\star\star$), we have that every edge in $[X_2, Y_1(B)]$ is red.  So if there is a red edge from $X_2$ to $Y_1(R)$, then there is a red component covering $X_2\cup Y_1\cup Z_2$ and we are done by Claim \ref{clm:C2}.  So suppose every edge from $X_2(B)$ to $Y_1(R)$ is blue, which implies there is a blue component $B$ covering $B_1$ and $Y_1(R)$.  Now either $X_1(R)\neq \emptyset$ in which case every edge in $[Y_2(R), X_1(R)]$ is blue and thus $B, [Y_2(R), X_1(R)], B_2$ forms the desired monochromatic cover, or $X_1(R)= \emptyset$ in which case $B, R_1, B_2$ forms the desired monochromatic cover.

\underline{Case 3.2.2} (For all $i\in [2]$, $X_i(R), X_i(B), Y_i(R), Y_i(B)$ are non-empty)

We have by ($\star\star$) that every edge from $Y_2(R)$ to $X_1(R)$ is blue and every edge from $Y_2(B)$ to $X_1(B)$ is red.  Suppose without loss of generality that there is a red edge from $X_1(R)$ to $Y_2(B)$ in which case there is a red component $R$ covering $V(R_1)\cup Y_2(B)\cup X_1(B)$ and thus $R, R_2$ and the red component covering $X_2(B)\cup Y_1(B)$ is the desired monochromatic cover.  
\end{proof}

The following example shows that in general, Conjecture \ref{con:rpart2} is best possible if true.

\begin{example}
For all $k, r\geq 2$, there exists $K\in \mathcal{K}_k$ such that $\tc_r(K)\geq r$.
\end{example}

\begin{proof}
Let $K\in \mathcal{K}_k$ in which one of the parts $X$ has order at least $r$, and let $x_1, \dots, x_r$ be $r$ distinct vertices in $X$.  For all $i\in [r]$, color all edges incident with $x_i$ with color $i$ and color all other edges arbitrarily.  
\end{proof}

The following is another example which shows that Theorem \ref{3partite} is best possible with the additional property that all of the parts have order 2.

\begin{example}
For all $k\geq 3$, there exists $K\in \mathcal{K}_k$ such that $\alpha(K)=2$ and $\tc_3(K)\geq 3$.
\end{example}

\begin{proof}
Let $K\in \mathcal{K}_k$ such that three parts $X,Y,Z$ have order 2.  Let $X=\{x_1, x_2\}, Y=\{y_1, y_2\}, Z=\{z_1,z_2\}$.  Let $\{x_1,y_1,z_1\}$ be a blue clique, let $\{x_2,y_1,z_2\}$ be a red clique, let $\{x_2,y_2,z_1\}$ be a green clique, let $\{x_1,y_2\}$ be red, let $\{y_2,z_2\}$ be blue, let $\{x_1,z_2\}$ be green, let every other edge incident with $\{y_1,z_2\}$ be red, every other edge incident with $\{x_1, z_1\}$ be blue, and every other edge incident with $\{x_2,y_2\}$ be green.
\end{proof}

Finally, we see that Conjecture \ref{con:SnotS} holds for $r=3$ by either Theorem \ref{thm:r3-5} with $r=3$ or Theorem \ref{2partite}, and Conjecture \ref{con:SnotS} holds for $r=4$ by combining Theorem \ref{thm:r3-5} with $r=4$ and Theorem \ref{3partite}.

Note that if Conjecture \ref{con:rpart} was true for $r=5$ (or, even stronger, if Conjecture \ref{con:rpart2} is true for $r-1=4$), then together with Theorem \ref{thm:r3-5} with $r=5$ this would imply Conjecture \ref{con:SnotS} is true for $r=5$.

\section{Partitioning into monochromatic connected subgraphs}\label{sec:partition}

Erd\H{o}s, Gy\'arf\'as, and Pyber proved the following theorem which strengthens the $\alpha=1$, $r=3$ case of Ryser's conjecture.

\begin{theorem}[Erd\H{o}s, Gy\'arf\'as, Pyber \cite{EGP}]\label{thm:EGP}
For all finite complete graphs $K$, $\tp_3(K)=2$.
\end{theorem}

Interestingly, no proof is known for infinite graphs (although the existence of a proof via personal communication is referenced in \cite{EGP}).   On the other hand, Hajnal proved \cite[Theorem 1]{HKSS} the weaker result that $\tp_r(K)\leq r$ for all infinite (countable or uncountable) complete graphs $K$ (in fact it can be specified that the trees have distinct colors and radius at most 2).

\begin{problem}\label{prob:infinite}
Let $K$ be a (countably) infinite complete graph.  Prove $\tp_3(K)=2$.
\end{problem}

The following is a beautiful strengthening of K\"onig's theorem (Theorem \ref{konig}, Theorem \ref{dualkonig}).  We reproduce the proof here in a less general form than what is found in \cite{FFGT1}.

\begin{theorem}[Fujita, Furuya, Gy\'arf\'as, T\'oth \cite{FFGT1}]
For all finite graphs $G$, $\tp_2(G)\leq \alpha(G)$.
\end{theorem}

\begin{proof}
We are done if $\alpha(G)=1$, so suppose $\alpha(G)\geq 2$ and the statement holds for all $G'$ with $\alpha(G')<\alpha(G)$.

We know that $\tc_2(G)\leq \alpha(G)$.  Let $R_1, \dots, R_p$ be the red components and let $B_1, \dots, B_q$ be the blue components in such a monochromatic cover.  Note that $$p+q\leq \alpha(G).$$
Let $R=\bigcup_{i=1}^p V(R_i)$ and $B=\bigcup_{i=1}^q V(B_i)$.  Let $R'=R\setminus B$ and $B'=B\setminus R$.  Note that we are done unless $R'\neq \emptyset$ and $B'\neq \emptyset$.  Also note that there are no edges between $R'$ and $B'$ so $$\alpha(G[R'])+\alpha(G[B'])\leq \alpha(G).$$
Thus $\alpha(G[R'])<\alpha(G)$ and $\alpha(G[B'])<\alpha(G)$.  By induction we have $p':=\tp_2(G[R'])\leq \alpha(G[R'])$ and $q':=\tp_2(G[B'])\leq \alpha(G[B'])$. Let $C_1, \dots, C_{p'}$ be the monochromatic partition of $G[R']$ and let $D_1, \dots, D_{q'}$ be the monochromatic partition of $G[B']$.  Note that $$p'+q'\leq \alpha(G[R'])+\alpha(G[B'])\leq \alpha(G).$$
Since $p'+q+p+q'\leq 2\alpha(G)$, we have say $p'+q\leq \alpha (G)$.  So $C_1, \dots, C_{p'}, B_1, \dots, B_q$ is the desired monochromatic partition.  
\end{proof}

Haxell and Kohayakawa proved a weaker version of Conjecture \ref{con:EGP} (but stronger in the sense that the subgraphs have bounded radius).

\begin{theorem}[Haxell, Kohayakawa \cite{HK}]\label{thm:HK}
Let $r\geq 2$.  If $n\geq \frac{3r^4r!\ln r}{(1-1/r)^{3(r-1)}}$, then $\tp_r(K_n)\leq r$. 
Furthermore, it can be specified that the trees have radius at most $2$ and have distinct colors.
\end{theorem}

Given this result, it would be interesting to prove a bounded diameter strengthening of Theorem \ref{thm:EGP}.  

\begin{problem}\label{prob:diameterpartition}
Does there exist a constant $d$ such that in every $3$-coloring of $K_n$ there exists a monochromatic $2$-partition consisting of subgraphs of diameter at most $d$?
\end{problem}

The lower bound on $n$ in Theorem \ref{thm:HK} was slightly improved by Bal and DeBiasio \cite{BD} to $n\geq 3r^2r!\ln r$.  The proofs in \cite{HK} and \cite{BD} go as follows:  Construct a set $X=\{x_1, \dots, x_r\}$ and disjoint set $Y$ so that for all $i\in [r]$, $x_i$ only sends edges of color $i$ to $Y$.  Then letting $Z=V(K_n)\setminus (X\cup Y)$, we have an $r$-colored complete bipartite graph $[Y,Z]$.  We say that $Y$ has a \emph{good partition} if there exists an integer $1\leq k\leq r$ and a partition $\{Y_1, \dots, Y_k\}$ of $Y$ (allowing for parts of the partition to be empty) such that for all $z\in Z$, there exists $i\in [k]$ and $y\in Y_i$ such that $zy$ has color $i$.  Then it is shown that if $Y$ is large enough (equivalently, $Z$ is small enough) then $Y$ has a good partition.  If $Y$ has a good partition, then there exists a partition $\{Z_1, \dots, Z_k\}$ of $Z$ (allowing for parts of the partition to be empty) such that for all $i\in [k]$ and $z\in Z_i$, $z$ sends an edge of color $i$ to $Y_i$.  Thus for all $i\in [k]$, the graph of color $i$ induced on $\{x_i\}\cup Y_i\cup Z_i$ can be covered by a tree of radius at most 2.

The following lemma is a slight modification of the relevant lemma in \cite{BD}. 

\begin{lemma}\label{lem:Ypartition}
Let $r\geq 2$ and let $G\in \mathcal{K}_2$ with parts $Y$ and $Z$, where $Y$ is finite if $r\geq 3$.  
If $|Z|<(\frac{r}{r-1})^{|Y|}$, then for every $r$-coloring of the edges of $G$ there exists a good partition of $Y$.
\end{lemma}

\begin{proof}
We say that a partition $\{Y_1, \dots, Y_k\}$ of $Y$ is \emph{good for} $z\in Z$ if there exists $i\in [k]$ and $y\in Y_i$ such that $zy$ has color $i$; otherwise we say that the partition of $Y$ is \emph{bad for} $z$.  

For all $z\in Z$, there are $(r-1)^{|Y|}$ partitions of $Y$ which are bad for $z$.  Since $$|Z|\cdot (r-1)^{|Y|}<\left(\frac{r}{r-1}\right)^{|Y|}\cdot (r-1)^{|Y|}=r^{|Y|},$$ there exists a partition of $Y$ which is good for every vertex in $Z$.  
\end{proof}

Our original intention was to come up with a new proof of Theorem \ref{thm:EGP} which would allow us to solve Problem \ref{prob:infinite} or Problem \ref{prob:diameterpartition}, but in the process we found an example to show that Lemma \ref{lem:Ypartition} is tight when $r=2$ or $Y$ is infinite.

\begin{example}\label{ex:r2}
Let $G\in \mathcal{K}_2$ with parts $Y$ and $Z$ where $Y$ and $Z$ are possibly infinite.  If $|Z|\geq 2^{|Y|}$, then there exists a $2$-coloring of the edges of $G$ such that $Y$ does not have a good partition.
\end{example}

\begin{proof}
Let $\{Z_{\mathbf{b}}: \mathbf{b}\in \{0,1\}^{|Y|}\}$ be a partition of $Z$ indexed by the binary strings of length $|Y|$ (i.e.\ functions from $Y$ to $\{0,1\}$) where every set is non-empty (which is possible since $|Z|\geq 2^{|Y|}$). For each $z\in Z_{\mathbf{b}}$ and $y\in Y$, color $zy$ with $\mathbf{b}(y)$ (so the colors are 0 and 1).  

\begin{figure}[ht]
\begin{center}
\includegraphics{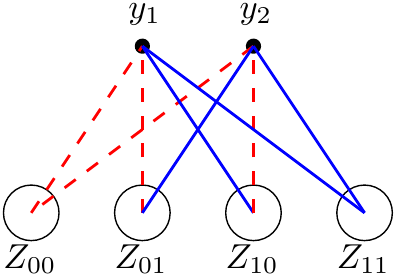}
\caption{Example \ref{ex:r2} in the case $|Y|=2$.}
\end{center}
\end{figure}

Consider a partition $\{Y_0,Y_1\}$ of $Y$ (with $Y_i$ possibly empty) and consider the binary string $\mathbf{a}$ where $\mathbf{a}(y)=j$ if and only if $y\in Y_j$.  Let $\mathbf{b}$ be the binary string where $\mathbf{b}(y)=1-\mathbf{a}(y)$ for all $y\in Y$.  So $z\in Z_{\mathbf{b}}$ does not send any edges of color $i$ to $Y_i$ for all $i\in \{0,1\}$.  
\end{proof}

We also have an example which shows that Lemma \ref{lem:Ypartition} is close to tight when $Y$ is finite and $r\geq 3$.

\begin{example}\label{ex:rdrandom}
Let $r\geq 3$ and let $G\in \mathcal{K}_2$ with parts $Y$ and $Z$ (where $Y$ and $Z$ are finite).  If $|Z|> 4|Y|\ln r(\frac{r}{r-1})^{|Y|}$, then there exists an $r$-coloring of the edges of $G$ such that $Y$ does not have a good partition.  
\end{example}

This is obtained by showing that with positive probability, a random $r$-coloring of $G$ does not have a good partition.  However, we don't give the details here since this result will be superseded by an upcoming result with a better constant term.

With regards to Problem \ref{prob:infinite}, Lemma \ref{lem:Ypartition} has the following consequence.

\begin{corollary}
Let $K$ be a 3-colored complete graph on a set $V$.  If there exists a maximal monochromatic component $C$ (that is a monochromatic component which is not properly contained in a monochromatic component of another color) such that $|V(C)|<2^{|V\setminus V(C)|}$,
then there exists a 2-partition of $K$.  In particular, if $V$ is countably infinite, then there is a 2-partition of $K$ unless every maximal monochromatic component is 
cofinite.
\end{corollary}

\begin{proof}
By the assumption, let $C$ be a maximal monochromatic component with $|V(C)|<2^{|V\setminus V(C)|}$ and without loss of generality, suppose $C$ is green.  Set $Z=V(C)$ and $Y=V\setminus V(C)$.  By maximality of $C$, we may suppose that all edges between $Y$ and $Z$ are either red or blue.  We apply Lemma \ref{2colorbip_nodiam} and note that we are done unless \ref{p1'} holds where $Y$ is the double covered side (again by maximality of $C$).  Now since $|Z|<2^{|Y|}$, we can apply Lemma \ref{lem:Ypartition} to get a good partition $\{Y_1, Y_2\}$ of $Y$ and a corresponding partition $\{Z_1, Z_2\}$ of $Z$ such that $Y_1\cup Z_1$ and $Y_2\cup Z_2$ induce monochromatic components.   
\end{proof}

%
%
%

The following corollary provides a proof of Theorem \ref{thm:partmulti}.

\begin{corollary}\label{cor:badmulti}Let $k\geq 2$ and $G\in \mathcal{K}_k$ with the vertex partition $\{V_1, \dots, V_k\}$.  If there exists $i\in [k]$, such that $|V_i|\geq 2^{|V(G)\setminus V_i|}$, then $\tp_2(G)\geq \floor{|V_i|/2^{|V(G)\setminus V_i|}}$.  In particular, for all integers $t\geq 1$ and $k\geq 2$, there exists $G\in \mathcal{K}_k$ such that $\tp_2(G)>t$.
\end{corollary}

\begin{proof}
We let $Z=V_i$ and $Y=V(G)\setminus V_i$ and color the edges between $Y$ and $Z$ as in Example \ref{ex:r2} (where we partition $Z$ into as equal sized sets as possible so that each part of the partition has at least $\floor{|Z|/2^{|Y|}}$ elements).  Regardless of the edges inside the set $Y$, no matter how the set $Y$ gets partitioned into red and blue subgraphs, there will be a part of the partition of $Z$ which sends blue edges to the red subgraphs and red edges to the blue subgraphs.
\end{proof}

The following question essentially asks whether the situation described in Corollary \ref{cor:badmulti} is the only way to avoid having a 2-partition of a 2-colored multipartite graph.

\begin{problem}
Is the following true?\\
Let $k\geq 2$ be an integer and let $G\in \mathcal{K}_k$ with vertex partition $\{V_1, \dots, V_k\}$.  If for all $i\in [k]$, $|V_i|<2^{|V(G)\setminus V_i|}$, then in every 2-coloring of the edges of $G$ there exists a 2-partition of $G$.
\end{problem}

Given an $r$-colored graph $G$ and a color $i\in [r]$, let $G_{cross}(i)$ be the multipartite graph consisting of the edges going between the components of color $i$.  So if there are $k$ components of color $i$, then $G_{cross}(i)$ is a $k$-partite graph colored with $[r]\setminus \{i\}$.  

\begin{problem}\label{ex:3cross}
Is the following true?\\
There exists a 3-coloring of a complete graph such that for all $i\in [3]$, there are at least three components of color $i$ and there is no partition of $G_{cross}(i)$ into two monochromatic connected subgraphs.  
\end{problem}

Encouraged by the exact answer for $r=2$ (from Lemma \ref{lem:Ypartition} and Example \ref{ex:r2}), we attempted to obtain a precise answer for $r\geq 3$ (even though it wouldn't help improve the lower bound in Theorem \ref{thm:HK} by any significant amount).  Towards this end, for all integers $r\geq 2$ and $d\geq 1$, let $Z(r,d)$ be the smallest positive integer $z$ such that if $G$ is a complete bipartite graph with parts $Y$ and $Z$ with $|Y|=d$ and $|Z|=z$, then there exists an $r$-coloring of $G$ in which there is no good partition of $Y$.  In this language, we know from Lemma \ref{lem:Ypartition}, Example \ref{ex:r2}, and Example \ref{ex:rdrandom} that $Z(2,d)=2^d$ for all $d\geq 1$ and $\left(\frac{r}{r-1}\right)^d\leq Z(r,d)\leq 4d\ln r\left(\frac{r}{r-1}\right)^d$ for all $r\geq 3, d\geq 1$.

\begin{problem}\label{prob:Zrd}
For all $r\geq 3$ and $d\geq 1$, determine $Z(r,d)$.
\end{problem}

We begin with a few simple observations.

\begin{observation}\label{obs:Zrd}
For all $r\geq 2$ and $d\geq 1$,
\begin{enumerate}
    \item If $r'\geq r$, then $Z(r',d)\leq Z(r,d)$.
    \item $Z(r,d)\geq d+1$.
    \item If $r\geq d+1$, then $Z(r,d)\leq d+1$.
    \item $Z(r,r)\leq r+\ceiling{\frac{r}{2}}+1$
\end{enumerate}
\end{observation}

\begin{proof}
Let $G$ be a complete bipartite graph with parts $Y$ and $Z$ with $|Y|=d$ and $|Z|=z$. 
\begin{enumerate}
    \item If there exists an $r$-coloring of $G$ such that every partition of $Y$ is bad, then since $r'\geq r$ the $r$-coloring of $G$ is an $r'$-coloring of $G$ such that every partition of $Y$ is bad.
    \item If $z\leq d$, then let $Y=\{u_1, \dots, u_d\}$ and $Z=\{v_1, \dots, v_z\}$ and suppose we are given an $r$-coloring of the edges of $G$.  Let $\{Y_1, \dots, Y_r\}$ be a partition of $Y$ (with some sets possibly empty) such that for all $i\in [z]$, $u_i\in Y_c$ where $c$ is the color of the edge $u_iv_i$, and for all $z+1\leq i\leq d$, $u_i\in Y_1$.  Clearly $\{Y_1, \dots, Y_r\}$ is a good partition of $Y$.  Thus $Z(r,d)\geq d+1$.
    \item Label the vertices of $Z$ as $v_1, \dots, v_z$ and consider the coloring of $G$ where for all $i\in [d+1]$, all edges incident with $v_i$ get color $i$.  Then every partition of $Y$ is bad for some vertex in $\{v_1, \dots, v_{d+1}\}$.
    \item Suppose $z=r+\ceiling{\frac{r}{2}}+1$.  Label the vertices of $Z$ as $v_1, \dots, v_r, u_1, \dots, u_{\ceiling{\frac{r}{2}}+1}$ and label the vertices of $Y$ as $y_1, \dots, y_r$.  Consider the following coloring of $G$. For all $i\in [r]$, all edges incident with $v_i$ get color $i$. For all $i\in [\ceiling{\frac{r}{2}}+1]$, the edges from $u_i$ to $\{y_1, \dots, y_{\ceiling{\frac{r}{2}}}\}$ are colored with $i$ and the edges from $u_i$ to $Y\setminus \{y_1, \dots, y_{\ceiling{\frac{r}{2}}}\}$ are colored with $i+1$, except that the edges from $u_{\ceiling{r/2}+1}$ to $Y\setminus \{y_1, \dots, y_{\ceiling{\frac{r}{2}}}\}$ are colored with 1.  If there is a good partition $\{Y_1, \dots, Y_{r}\}$ of $Y$, it must be the case that all sets in the partition are singletons because otherwise one of $v_1, \dots, v_r$ would witness a bad partition.  Also there is exactly one vertex $u_i\in \{u_1, \dots, u_{\ceiling{\frac{r}{2}}+1}\}$ which is not satisfied by a vertex from $\{y_1, \dots, y_{\ceiling{\frac{r}{2}}}\}$; however, the only color that $u_i$ sends to $Y\setminus \{y_1, \dots, y_{\ceiling{\frac{r}{2}}}\}$ has already been used on $\{y_1, \dots, y_{\ceiling{\frac{r}{2}}}\}$.\qedhere
\end{enumerate}
\end{proof}

We were able to compute some small values of $Z(r,d)$ using an integer linear program.  Surprisingly, we didn't even have enough computing power to determine $Z(4,4)$ or $Z(3,5)$.

\begin{table}[ht]
\centering
\begin{tabular}{|l|l|l|l|l|l|l|}
\hline
\backslashbox{$d$}{$r$} & \textbf{2} & \textbf{3} & \textbf{4} & \textbf{5} & \textbf{6} & \textbf{7} \\ \hline
\textbf{1} &  \cellcolor{yellow}2 & \cellcolor{yellow}2 & \cellcolor{yellow}2 & \cellcolor{yellow}2 & \cellcolor{yellow}2& \cellcolor{yellow}$\rightarrow$ \\ \hline
\textbf{2} &  \cellcolor{yellow}4 & \cellcolor{yellow}3 & \cellcolor{yellow}3 & \cellcolor{yellow}3 &\cellcolor{yellow}3& \cellcolor{yellow}$\rightarrow$ \\ \hline
\textbf{3} &  \cellcolor{yellow}8 & \cellcolor{yellow}5 & \cellcolor{yellow}4 & \cellcolor{yellow}4  &\cellcolor{yellow}4&  \cellcolor{yellow}$\rightarrow$ \\ \hline
\textbf{4} &  \cellcolor{yellow}16 & \cellcolor{yellow}8 & 6/7 & \cellcolor{yellow}5 &\cellcolor{yellow}5& \cellcolor{yellow}$\rightarrow$ \\ \hline
\textbf{5} &  \cellcolor{yellow}32 & 11/12 &  &  & \cellcolor{yellow}6&\cellcolor{yellow}$\rightarrow $\\ \hline
\textbf{6} & \cellcolor{yellow}$\downarrow$ &  &  &  & & \cellcolor{yellow}$\searrow$\\ \hline
\end{tabular}
\caption{Values of $Z(r,d)$.  Exact values are highlighted in yellow.}\label{tab:Zrd} 
\end{table}

Note that $Z(r,d)$ is equivalent to the following.  Let $d$ and $r$ be positive integers and let $\mathcal{W}_{r,d}$ be the set of functions from $[d]$ to $[r]$ (which we think of as words of length $d$ over the alphabet $[r]$).  Say that two functions $f,g\in\mathcal{W}_{r,d}$  are \emph{everywhere different} if  $f(i)\neq g(i)$ for all $i\in [d]$.  Let $Z(r,d)$ be the smallest integer $z$ such that there exists $Z\subseteq \mathcal{W}_{r,d}$ with $|Z|= z$ such that for all $f\in \mathcal{W}_{r,d}$, there exists $g\in Z$ such that $f$ and $g$ are everywhere different.  For instance when $d=2$ and $r=3$, it is easy to see that $Z=\{(1,1), (2,2), (3,3)\}$ is a smallest such set with this property.  In the case $d=3=r$, one can check that $Z=\{(1,1,2), (1,2,1), (2,1,1), (2,2,2), (3,3,3)\}$ is a smallest such set.
To see that this is equivalent to the bipartite graph version, we think of the set $Z$ as the colorings of the edges from each vertex $v\in Z$ to the set $Y$. Then each partition of $Y$ corresponds to a function $f\in \mathcal{W}_{r,d}$ which is everywhere different from the edge coloring incident with some vertex in $Z$.

We now show that $Z(r,d)$ is equivalent to two other well-studied parameters whose bounds seem to be difficult to improve in general.

Let $G^{\times d}$ be the $d$-fold directed product of $G$; that is, $V(G^{\times d})$ is the set of $d$-dimensional vectors with entries in $V(G)$ and $(a_1, \dots, a_d)$ is adjacent to $(b_1, \dots, b_d)$ if and only if $a_ib_i\in E(G)$ for all $i\in [d]$.  We will be interested in $K_r^{\times d}$, where we have that $(a_1, \dots, a_d)$ is adjacent to $(b_1, \dots, b_d)$ if and only if $a_i\neq b_i$ for all $i\in [d]$.  We say that $S\subseteq V(G)$ is a total dominating set if every vertex in $V(G)$ has a neighbor in $S$.  Let $\gamma_t(G)$ be the number of vertices in a minimum total dominating set of $G$.  Let $H(r,d)$ be the $(r-1)^d$-uniform hypergraph on $V:=V(K_r^{\times d})$ where $S\subseteq V$ is an edge of $H(r,d)$ if and only if there exists $v\in V(K_r^{\times d})$ such that $N(v)=S$; equivalently, for each $(a_1, \dots, a_d)\in V(H(r,d))$, the set $S_{(a_1, \dots, a_d)}=\{(b_1, \dots, b_d): b_i\neq a_i \text{ for all } i\in [k]\}$ is an edge of $H(r,d)$.

\begin{theorem}\label{thm:equiv}
For all $r\geq 2$ and $d\geq 1$, $Z(r,d)=\gamma_t(K_r^{\times d})=\tau(H(r,d))$
\end{theorem}

\begin{proof}
Let $r\geq 2$ and $d\geq 1$ be given.  First note that $\gamma_t(K_r^{\times d})=\tau(H(r,d))$ since the vertex sets of $H(r,d)$ and $K_r^{\times d}$ correspond to each other, and the edges of $H(r,d)$ correspond to the neighborhoods of vertices in $K_r^{\times d}$.  Clearly a transversal in $H(r,d)$ corresponds to a total dominating set in $K_r^{\times d}$.

To see that $Z(r,d)=\gamma_t(K_r^{\times d})$, suppose that we have a total dominating set $T$ of order $z$ in $K_r^{\times d}$, each vertex of which is a vector of length $d$ over the alphabet $\{0, \dots, r-1\}$.  Now let $Z$ be a set of $z$ vertices and for each vertex in $Z$, color the edges according to the corresponding vertex (vector) from $T$.  Every partition of $Y$ now corresponds to a vertex $(x_1, \dots, x_d)$ in $V(K_r^{\times d})$ and since $T$ is a total dominating set (and the definition of $K_r^{\times d}$), there exists a vertex $(x_1', \dots, x_d')\in T$ such that $x_i\neq x_i'$ for all $i\in [d]$ which means $(x_1, \dots, x_d)$ is a bad partition of $Y$.  On the other hand if $Z$ is a set of $z-1$ vertices, then since every set $T'$ of $z-1$ vertices in $K_r^{\times d}$ is not a total dominating set, there exists a vertex $(x_1, \dots, x_d)$ in $V(K_r^{\times d})$ which is not adjacent to anything in $T'$ and this vertex corresponds to a good partition of $Y$.  
\end{proof}

The following is a result independently obtained by Johnson \cite{Joh}, Lov\'asz \cite{Lov}, and Stein \cite{Ste}.

\begin{proposition}\label{prop:1}
For all hypergraphs $H$ with maximum degree $\Delta$, $\tau^*(H)\leq \tau(H)\leq (1+\ln \Delta)\tau^*(H)$.
\end{proposition}

Because $H(r,d)$ is $(r-1)^d$-uniform and $(r-1)^d$-regular, one can see that $\tau^*(H(r,d))=(\frac{r}{r-1})^d$ and thus we have the following corollary. 

\begin{corollary}\label{cor:1}
For all $r\geq 2$ and $d\geq 1$,
$(\frac{r}{r-1})^d\leq \tau(H(r,d))\leq (1+d\ln(r-1))(\frac{r}{r-1})^d$.
\end{corollary}

The following is a known fact about the total domination number of a graph (see \cite{HY}).

\begin{proposition}\label{prop:2}
Let $G$ be a graph on $n$ vertices with minimum degree $\delta$ and maximum degree $\Delta$.  Then
$\frac{n}{\Delta}\leq \gamma_t(G)\leq \frac{1+\ln \delta}{\delta}n$
\end{proposition}

We have $|V(K_r^{\times d})|=r^d$ and $\delta(K_r^{\times d})=\Delta(K_r^{\times d})=(r-1)^d$ and thus we have the following corollary.

\begin{corollary}\label{cor:2}
For all $r\geq 2$ and $d\geq 1$,
$(\frac{r}{r-1})^d\leq \gamma_t(K_r^{\times d})\leq (1+d\ln(r-1))(\frac{r}{r-1})^d$.
\end{corollary}

Note that by Theorem \ref{thm:equiv}, Corollary \ref{cor:1} and Corollary \ref{cor:2} can be derived from each other; however, it is interesting to note that they can be derived independently using the known bounds from Proposition \ref{prop:1} and Proposition \ref{prop:2} respectively.

\section{Monochromatic covers of hypergraphs}\label{sec:hypergraph}

The $\alpha=1$ case of Ryser's conjecture says $\tc_r(K_n^2)\leq r-1$.  Kir\'aly \cite{K} surprisingly gave a very simple proof that for all $k\geq 3$, $\tc_r(K_n^k)= \ceiling{r/k}$.  Earlier, Aharoni and Ziv \cite{AZ} proved that for $k\geq 3$, $\tc_r(K_n^k)\leq \ceiling{\frac{r-1}{k-1}}$ (they proved this in the dual language of $r$-partite hypergraphs in which every $k$ edges intersect).  Part of the reason determining $\tc_r(K_n^k)$ is so much easier for $k\geq 3$ than $k=2$ seems to come down to the very weak notion of connectivity typically used for hypergraphs.  Inspired by some recent results (\cite{CKK2}, \cite{CKK3}, \cite{CKP}, \cite{GHM}), we propose a more general problem which allows for stronger notions of connectivity in hypergraphs.

Let $c,\ell,k$ be positive integers with $k\geq 2$ and $c,\ell\leq k-1$ and let $H$ be a $k$-uniform hypergraph.  Say that a pair of $c$-sets $S, S'\in \binom{V(H)}{c}$ is $\ell$-connected if there exists edges $e_1, \dots, e_p$ such that $S\subseteq e_1$, $S'\subseteq e_p$, and $|e_i\cap e_{i+1}|\geq \ell$ for all $i\in [p-1]$.  A $(c,\ell)$-component $C$ of $H$ is a maximal set of pairwise $\ell$-connected $c$-sets.  Note that we can define a relation $\sim$ on $\binom{V(H)}{c}$ where $S\sim S'$ if and only if $S$ and $S'$ are $\ell$-connected.  When $c\geq \ell$, this is an equivalence relation and the $(c,\ell)$-components of $H$ are just the equivalence classes.

Let $\tc_r^{c,\ell}(H)$ be the smallest integer $t$ such that in every $r$-coloring of the edges of $H$, there exists a set of at most $t$ monochromatic $(c,\ell)$-components $\mathcal{C}$ (that is, each $C\in \mathcal{C}$ is a component in $H_i$ for some $i\in [r]$) such that $\bigcup_{C\in \mathcal{C}}C=\binom{V(H)}{c}$.  When $c=\ell$, we write $\tc_r^{\ell}(H)$ to mean $\tc_r^{\ell,\ell}(H)$.  

In this language, we can state Kir\'aly's result as follows.  We will also give Kir\'aly's proof of the upper bound.

\begin{theorem}[Kir\'aly \cite{K}]\label{kiraly}
For $n\geq k\geq 3$ and $r\geq 1$, $\tc_r^1(K_n^k)=\ceiling{r/k}$.
\end{theorem}

\begin{proof}
If $r=1$, the result is trivial, so let $r\geq 2$ and suppose that $\tc_{r-1}^1(K_n^k)\leq \ceiling{(r-1)/k}$.

If there exists a set $S$ of $k-1$ vertices such that $S$ is contained in edges of at most $\ceiling{r/k}$ colors, then we are done. So for every set $S\subseteq V(K)$ of order $k-1$, $S$ is contained in edges of at least $\ceiling{r/k}+1$ colors.  For every edge $e$ of color $r$, there are $k$ distinct $k-1$ sets contained in $e$ and thus there are distinct $S, S'\subseteq e$ with $|S|=k-1=|S'|$ and $i\in [r-1]$ such that $S$ and $S'$ are contained in a component of color $i$ which implies that $e$ is contained in a component of color $i$.  Since all of the edges of color $r$ are contained in a component of color $i\in [r-1]$, we actually have an $(r-1)$-coloring of $K$ and thus by induction there is a monochromatic $\ceiling{(r-1)/k}$-cover (which is of course a $\ceiling{r/k}$-cover).
\end{proof}

We propose the following general problem.

\begin{problem}
Let $r,c,\ell,k$ be positive integers such that $c,\ell\leq k-1$.  Determine the value of $\tc_r^{c,\ell}(K_n^k)$.
\end{problem}

We prove the following results.  

\begin{theorem}\label{thm:k/3}
Let $r,c,\ell,k$ be positive integers such that $1\leq \ell\leq c\leq k/3$.  Then 
\[
\tc_r^{c,\ell}(K_n^k)=\ceiling{\frac{r}{\floor{k/c}}}.
\]
\end{theorem}

Note that this gives Theorem \ref{kiraly} when $c=1=\ell$.  

In the case when $r=2$, we are essentially able to give a complete answer.

\begin{theorem}\label{thm:r2hyper}
Let $c,\ell,k$ be positive integers such that $\ell, c\leq k-1$.  Then 
\[
\tc_2^{c,\ell}(K_n^k)=
\begin{cases} 
1 & \text{ if } c\leq k/2\\
2 & \text{ if } k/2<c\leq k-\ell/2\\
\Omega(n) & \text{ if } \max\{k-\ell/2, k/2\}<c\leq k-1
\end{cases}
  \]
\end{theorem}

The case $c<\ell$ is harder to analyze, but we are able to determine one interesting case exactly.

\begin{theorem}\label{thm:kr3tight}
$\tc_3^{1,2}(K_n^3)=1$ (i.e.\ in every 3-coloring of $K_n^3$ there is a spanning monochromatic tightly connected component)
\end{theorem}


\subsection{Lower bounds}

The following example generalizes Kir\'aly's example (which corresponds to the case $c=\ell=1$) and provides the lower bound in Theorem \ref{thm:k/3}.

\begin{example}
For all $c,\ell\geq 1, r\geq 2, k\geq 3$ and $n\geq c\cdot \binom{r}{\ceiling{\frac{r}{\floor{k/c}}}-1}$,  $\tc_{r}^{c,\ell}(K_n^{k})\geq \ceiling{\frac{r}{\floor{k/c}}}$.
\end{example}

\begin{proof}
Set $t:=\floor{k/c}$ and $q:=\ceiling{r/t}-1$. Let $K=K_n^{k}$ and partition $V(K)$ into $m:=\binom{r}{q}$ sets $V_{x_1}, \dots, V_{x_m}$ of order at least $c$, where $x_1, \dots, x_m$ represent each of the subsets of $[r]$ of order $q$.  For each edge $e\in E(K)$, let $\phi(e)=\bigcup_{i:|e\cap V_{x_i}|\geq c}x_i$.  Since $|e|=k<(\floor{k/c}+1)c=(t+1)c$, $e$ intersects at most $t$ of the sets $V_{x_1}, \dots, V_{x_m}$ in at least $c$ vertices, so $|\phi(e)|\leq tq<r$ and thus $[r]\setminus \phi(e)\neq \emptyset$.  Color $e$ with the smallest $j\in [r]\setminus \phi(e)$.  Now let $A\subseteq [r]$ with $|A|=q$ and note that there exists $i$ such that $A=x_i$.  Note that no $(c,\ell)$-component having a color in $A$ contains any of the $c$-sets from $V_{x_i}$ and thus the number of $(c,\ell)$-components needed to cover $\binom{V(K)}{c}$ is more than $q$; i.e. $\tc_{r}^{c,\ell}(K_n^{k})\geq q+1=\ceiling{\frac{r}{\floor{k/c}}}$.
\end{proof}



The next example provides the lower bound in the last case of Theorem \ref{thm:r2hyper}.

\begin{example}\label{ex:n/c}
Let $c,\ell\geq 1, r\geq 2, n\geq k\geq 3$.  If $c>\max\{k-(1-1/r)\ell, k/2\}$, then $\tc_r^{c,\ell}(K_n^k)\geq \floor{n/c}+1$.
\end{example}

\begin{proof}
Set $t:=\floor{n/c}$, let $K=K_n^{k}$, and choose $t$ disjoint sets $x_1, \dots, x_t\subseteq V(K)$ each of order $c$.  Let $X=\{x_1, \dots, x_t\}$.  First note that since $c>k-(1-1/r)\ell$ we have 
\begin{equation}\label{eq:kell}
r(c+\ell-k)>r(c+\ell-(c+(1-1/r)\ell))=\ell.  
\end{equation}
For each $\ell$-set $y\in \binom{V(K)}{\ell}$ let $$I_y=\{i\in [t]: \ell+c-|y\cap x_i|=|y|+|x_i|-|y\cap x_i|=|y\cup x_i|\leq k\}$$ (equivalently $I_y=\{i\in [t]: y\cup x_i\subseteq e\in E(K)\}$) and note that $|I_y|\leq r-1$ as otherwise
$$\ell=|y|\geq \sum_{i\in I_y}|y\cap x_i|\geq r(\ell+c-k),$$ contradicting \eqref{eq:kell}.  

Let $\phi_y$ be an injective function from $I_y$ to $[r-1]$ and for all $i \in I_y$, color all edges containing $y\cup x_i$ with color $\phi_y(i)$.  Now color all other edges with color $r$.  Since $c>2k$, no edge in $E(K)$ contains more than one element of $x$ as a subset.  So by this fact and the way in the which the coloring was defined, no pair of $c$-sets from $X$ is in the same monochromatic $(c,\ell)$-component.
\end{proof}

\subsection{Upper bounds}

We begin with a few basic observations.

First, let $0\leq c\leq k$ be integers and let $H$ be a $k$-uniform hypergraph. For a set $S\in \binom{V(H)}{c}$, the \emph{link hypergraph of $S$}, denoted $H(S)$, is the hypergraph on vertex set $V(H)\setminus S$ and edge set $\{T\in \binom{V(H)}{k-c}: S\cup T\in E(H)\}$.  If $H$ is edge colored, then the edges of $H(S)$ inherit the color of the corresponding edge from $H$.

\begin{observation}\label{obs:basic}
Let $k\geq 2$ and $c,\ell,r\geq 1$ with $c,\ell\leq k-1$.
\begin{enumerate}
\item $\tc_r^{c,\ell}(K_n^{k+1})\leq \tc_r^{c,\ell}(K_n^{k})$.
\item $\tc_r^{c,\ell}(K_n^{k})\leq \tc_r^{c,\ell+1}(K_n^{k})$.
\item $\tc_r^{c,\ell}(K_n^k)\leq \tc_r^{c+1,\ell}(K_n^k)$.
\item $\tc_r^{c,\ell+1}(K_n^{k+1})\leq \tc_r^{c,\ell}(K_{n-1}^{k})$.
\end{enumerate}
\end{observation}

\begin{proof}~
\begin{enumerate}
\item Suppose we are given an $r$-coloring of $K_n^{k+1}$.  This induces an $r$-coloring of all of the $k$-sets.  So if we have a monochromatic $(c,\ell)$-cover of the $r$-coloring $K_n^k$, then this gives a $(c,\ell)$-cover of $K_n^{k+1}$.  
\item This follows since a monochromatic $(c,\ell+1)$-component is a monochromatic $(c,\ell)$-component.  
\item Since every $c$-set is contained in a $(c+1)$-set, a monochromatic $(c+1, \ell)$-component covers all of the $c$-sets contained in the $(c+1)$-sets.  
\item Suppose we are given an $r$-coloring of $K_n^{k+1}$.  Consider the link hypergraph of a vertex $v$ which induces an $r$-coloring of $K_{n-1}^{k}$.  Note that a monochromatic $(c,\ell)$-component in the link hypergraph of $v$ is a $(c, \ell+1)$-component in $K_n^{k+1}$.  So any monochromatic cover with $(c,\ell)$-components of the link hypergraph of $v$ gives a monochromatic cover with $(c, \ell+1)$-components of $K_n^{k+1}$.   \qedhere
\end{enumerate}
\end{proof}

\subsubsection{\texorpdfstring{$c\geq \ell$}{c>=l}}

\begin{observation}
Let $k-1\geq c\geq \ell\geq 1$, $r\geq 2$, and let $H$ be an $r$-colored $k$-uniform hypergraph.  Let $2\leq k'\leq \binom{k}{c}$ and let $G$ be a $k'$-uniform hypergraph on vertex set $\binom{V(H)}{c}$ where $\{S_1, \dots, S_{k'}\}\in E(G)$ if and only if there exists $e\in E(H)$ such that $S_1\cup  \dots \cup S_{k'}\subseteq e$.  Furthermore, color $\{S_1, \dots, S_{k'}\}\in E(G)$ with the color of the edge $e\in E(H)$ such that $S_1\cup  \dots \cup S_{k'}\subseteq e$, and note that $\{S_1, \dots, S_{k'}\}$ may receive more than one color.

If there are $t$ monochromatic $(1,1)$-components in $G$ which cover $V(G)$, then there exists $t'\leq t$ monochromatic $(c,\ell)$-components in $H$ which cover $\binom{V(H)}{c}$.
\end{observation}

\begin{proof}
Let $S$ and $S'$ be vertices in $G$ and suppose there is a monochromatic $(1,1)$-component which contains both $S$ and $S'$.  So there is a collection of edges, all the same color, $e_1, \dots, e_m\in E(G)$ such that $S\in e_1$, $S'\in e_m$ and $e_i\cap e_{i+1}\neq \emptyset$ for all $i\in [m-1]$.  Thus there are edges $f_1, \dots, f_m\in E(H)$, all the same color, such that $S\subseteq f_1$, $S'\subseteq f_m$ and $|f_i\cap f_{i+1}|\geq c\geq \ell$ for all $i\in [m-1]$.  Thus there is a monochromatic $\ell$-walk from $S$ to $S'$ in $H$.  The result follows.
\end{proof}

\begin{theorem}\label{thm:r^2}
Let $r\geq 2$, $k\geq 3$, and $c\geq \ell\geq 1$.  If $k/2<c\leq k-(1-1/r)\ell$, then $\tc_r^{c,\ell}(K_n^k)\leq r^2$, unless $r=2$ in which case $\tc_2^{c,\ell}(K_n^k)\leq 2$, or $r=3$ in which case $\tc_3^{c,\ell}(K_n^k)\leq 6$.
\end{theorem}

\begin{proof}
Let $K:=K_n^{k}$ with a given $r$-coloring of the edges.

Suppose $k/2<c\leq k-(1-1/r)\ell$.  First note that given a $c$-set $A$ and an $\ell$-set $B$, there exists $e\in E(K)$ such that $A\cup B\subseteq e$ if and only if $k\geq |A\cup B|=|A|+|B|-|A\cap B|=c+\ell-|A\cap B|$; i.e. $|A\cap B|\geq c+\ell-k$.  

Given any family of $r+1$ $c$-sets $X=\{X_1, \dots, X_{r+1}\}$, since $$r(c+\ell-k)\leq r(c+\ell-(c+(1-1/r)\ell))=\ell,$$ for every set of $r$ elements of $X$, there exists an $\ell$-set which is contained in an edge with each of the $r$ elements.  Furthermore, since $c\geq 2\ell/r$ (using $c\geq \ell$ and $r\geq 2$) we have
$$(r+1)(c+\ell-k)\leq k,$$ 
and we can choose a family of $r+1$ $\ell$-sets  $Y = \{Y_1, \dots, Y_{r+1}\}$ such that $Y_i$ is contained in an edge with every element in $X\setminus \{X_i\}$ and $|Y_1\cup \dots \cup Y_{r+1}|\leq k$ which implies that every pair $Y_i,Y_j$ is contained in the same edge of $K$ of some color, say $r$.  So if any two sets $X_i,X_j$ are both contained in an edge of color $r$ with an element in $Y$, there would be a $\ell$-walk of color $r$ between $X_i$ and $X_j$ in $K$.  So suppose that at most one element from $X$, say $X_{r+1}$, is contained in an edge of color $r$ with some element of $Y$.  However, now $Y_{r+1}$ is contained in an edge with every element in $X\setminus \{X_{r+1}\}$ and since there are only $r-1$ colors used on such edges, there is a monochromatic $\ell$-walk between some $X_i$ and $X_j$ in $K$.  Altogether, this implies that there is a monochromatic $\ell$-walk between some pair of distinct elements from $X$, so the closure of $G$ (the auxiliary graph $\hat{G}$ with an edge of color $i$ between any two vertices ($c$-sets) which have an $\ell$-walk of color $i$ between them) has independence number at most $r$ and thus by Observation \ref{obs:closure} and Fact \ref{fact:trivial}, we have $\tc_r(G)\leq \tc_r(\hat{G})\leq r\alpha(\hat{G})\leq r^2$.  If $r=2$, then Theorem \ref{dualkonig} applies and we have $\tc_2(G)\leq \tc_2(\hat{G})\leq \alpha(\hat{G})\leq 2$.  If $r=3$, then Aharoni's theorem \cite{A} applies (in the dual language) and we have $\tc_3(G)\leq \tc_3(\hat{G})\leq 2\alpha(\hat{G})\leq 6$.
\end{proof}

We now prove an upper bound on $\tc_r^{c,\ell}(K_n^k)$ when $c\leq k/2$ and $c\geq \ell$.  In particular, when $c\leq k/3$, this provides the upper bound for Theorem \ref{thm:k/3}.

\begin{theorem}\label{thm:ellcupper}
For $c\geq \ell\geq 1$, $r\geq 2$, and $c\leq k/2$, $\tc_r^{c,\ell}(K_n^{k})\leq \tc_r^{1}(K_n^{\floor{k/c}})$.  In particular, when $c\leq k/3$, we have $\tc_r^{c,\ell}(K_n^{k})\leq \ceiling{\frac{r}{\floor{k/c}}}$.
\end{theorem}

\begin{proof}
We have an $r$-coloring of $K:=K_n^{k}$ with vertex set $V$.  Let $H$ be an $r$-colored $\floor{k/c}$-uniform hypergraph on vertex set $\binom{V}{c}$ where $\{X_1,\dots, X_{\floor{k/c}}\}$ is an edge of color $i$ in $H$ if and only if there exists an edge $e$ of color $i$ in $K$ such that $X_1\cup  \dots \cup X_{\floor{k/c}}\subseteq e$.  Note that since $K$ is $k$-uniform, $H$ is a complete $\floor{k/c}$-uniform graph.  So $\tc_r^{c,\ell}(K_n^{k})\leq \tc_r^{1}(K_n^{\floor{k/c}})$.

When $\floor{k/c}\geq 3$, it follows from Theorem \ref{kiraly} that $\tc_r^{c,\ell}(K)\leq \tc_r^1(H)\leq \ceiling{\frac{r}{\floor{k/c}}}$.
\end{proof}

We now obtain the following complete answer when $r=2$ and $c\geq \ell$.

\begin{proof}[Proof of Theorem \ref{thm:r2hyper}]
If $c\leq k/2$, then by Theorem \ref{thm:ellcupper} we have $\tc_2^{c,\ell}(K_n^k)\leq \tc_2^{1}(K_n^{\floor{k/c}})=1$ where the last inequality holds by Theorem \ref{kiraly}. If $k/2<c\leq k-\ell/2$, then Theorem \ref{thm:r^2} applies and we have $\tc_2^{c,\ell}(K_n^k)\leq 2$.  Finally, if $\max\{k-\ell/2, k/2\}<c\leq k-1$, then Example \ref{ex:n/c} applies and we have $\tc_2^{c,\ell}(K_n^k)\geq \floor{n/c}=\Omega(n)$.
\end{proof}

In the case when $c\geq \ell$, we are left with the following.

\begin{problem}\label{prob:ell}
Determine $\tc_r^{c,\ell}(K_n^{k})$ when $c\geq \ell$, $c\geq 2$, $r\geq 3$ and $k/3<c\leq k-(1-1/r)\ell$.
\end{problem}

An interesting test case for $c>\ell$ would be $\tc_3^{2,1}(K_n^3)$.  We have $\tc_3^{2,1}(K_n^3)\leq \tc_3^{2}(K_n^3)\leq 6$ (from Observation \ref{obs:basic}(ii) and Theorem \ref{thm:r^2}), but perhaps $\tc_3^{2,1}(K_n^3)\leq 2$?

An interesting test case for $c=\ell$ and $k/2<c\leq k-2\ell/3$ (when $c=\ell$, this is $k/2<c\leq 3k/5$) would be $\tc_3^{3}(K_n^5)$.  

An interesting test case for $c=\ell$ and $k/3<c\leq k/2$ would be $\tc_4^{2}(K_n^4)$.

\subsection{\texorpdfstring{$c<\ell$}{c<l}}

The case $c<\ell$ seems to be harder to analyze because of the fact mentioned earlier that we can define a relation $\sim$ on $\binom{V(H)}{c}$ where $S\sim S'$ if and only if $S$ and $S'$ are $\ell$-connected.  When $c\geq \ell$, this is an equivalence relation and the $(c,\ell)$-components of $H$ are just the equivalence classes; however, when $c<\ell$ it is not necessarily the case that $\sim$ is transitive -- it may happen that $A, B, C$ are $c$-sets and there is an $\ell$-walk from $A$ to $B$ and an $\ell$-walk from $B$ to $C$, but this does not guarantee an $\ell$-walk from $A$ to $C$.

The only general non-trivial upper bound is the following observation which is an extension of Observation \ref{obs:basic}(iv). 

\begin{observation}\label{obs:c<l}
If $c<\ell$, then for all $0\leq s\leq \ell-1$, $\tc_r^{c,\ell}(K_n^k)\leq \tc_r^{c,\ell-s}(K_{n-s}^{k-s})$.  In particular, $\tc_r^{c,\ell}(K_n^k)\leq \tc_r^c(K_{n-(\ell-c)}^{k-(\ell-c)})$. 
\end{observation}

\begin{proof}
Let $[n]=V(K_n^k)$ and let $S\subseteq \binom{[n]}{s}$ and let $H(S)$ be the link hypergraph of $S$ (which is a complete $(k-s)$-uniform hypergraph on $n-s$ vertices).  Note that if $T$ and $T'$ are two $c$-sets in $\binom{V(H(S))}{c}$ and there is a monochromatic $(\ell-s)$-walk between $T$ and $T'$ in $H(S)$, then there is a monochromatic $\ell$-walk between $T$ and $T'$ in $K_n^k$.  So if there are $t:=\tc_r^{c,\ell-s}(K_{n-s}^{k-s})$ monochromatic $(c,\ell-s)$-components covering $\binom{[n]\setminus S}{c}$, then there are $t$ monochromatic $(c,\ell)$-components covering $\binom{[n]}{c}$.
\end{proof}

Note that previous observation in particular implies $\tc_r^{1,k-1}(K_n^k)\leq \tc_r^1(K_n^2)=\tc_r(K_n)$.

The first interesting test case for $c<\ell$ is $\tc_3^{1,2}(K_n^3)$. We have $\tc_3^{1,2}(K_n^3)\leq \tc_3^1(K_n^2)=2$ from above, but perhaps, $\tc_3^{1,2}(K_n^3)=1$?  We now show that this is indeed the case by more carefully considering the possible structures in the 3-colored link graph of a vertex.

The following Lemma appears in \cite{DM16}, but we reproduce it here for completeness. 

\begin{lemma}
  \label{lem:3col}
Let $K$ be a complete graph. For every 3-coloring of $K$, either
\begin{enumerate}
\item there exists a monochromatic connected subgraph on $n$ vertices, or

\item there exists a partition $\{W, X, Y, Z\}$ of $[n]$ (all parts non-empty),
such that $B^1:=[W,X]$ and $B^2:=[Y,Z]$ are complete in blue, $R^1:=[W,Y]$ and
$R^2:=[X,Z]$ are complete in red, and $G^1:=[W,Z]$ and $G^2:=[X,Y]$ are complete
in green, or

\item there exists a partition $\{W, X, Y, Z\}$ of $[n]$ with $X,Y,Z$ non-empty
such that $B:=W\cup X\cup Y$ is connected in blue, $R:=W\cup X\cup Z$ is
connected in red, and $G:=W\cup Y\cup Z$ is connected in green.  Furthermore,
$[X,Y]$ is complete in blue, $[X,Z]$ is complete in red, and $[Y,Z]$ is complete
in green, whereas no edge in $[W,X]$ is green, no edge in $[W,Y]$ is red, and no
edge in $[W,Z]$ is blue.
\end{enumerate}
\end{lemma}

\begin{proof}
Suppose $B$ is a maximal monochromatic, say blue, connected subgraph and set $U=V(K)\setminus B$.
If $U=\emptyset$ then we are in case (i); so suppose not.  Note that all edges from $B$ to $U$ are either red or green.  Let $R$ be a maximal, say red, component which intersects both $B$ and $U$. By the maximality of $B$, we have $B\setminus R\neq \emptyset$. 

First suppose $U\setminus R\neq \emptyset$.  In this case, both $[B\cap R,
U\setminus R]$ and $[B\setminus R, U\cap R]$ are complete in green.  This
implies $[B\cap R, U\cap R]$ and $[B\setminus R, U\setminus R]$ are complete in
red and $[B\cap R, B\setminus R]$ and $[U\cap R, U\setminus R]$ are complete in
blue.  So we are in case (ii), setting $W:=B\cap R$, $X:=B\setminus R$, $Y:=U\cap R$,
and $Z:=U\setminus R$.

Finally, suppose $U\setminus R=\emptyset$.  In this case $[B\setminus R, U]$ is
complete in green, so there is a maximal green component $G$ containing $U\cup
(B\setminus R)$.  Then we are in case (iii), setting $W:=B\cap R\cap G$, $X:=B\setminus
G$, $Y:=B\setminus R$, and $Z:=U$. 
\end{proof}

\begin{figure}[ht]
\centering
\includegraphics[scale=0.85]{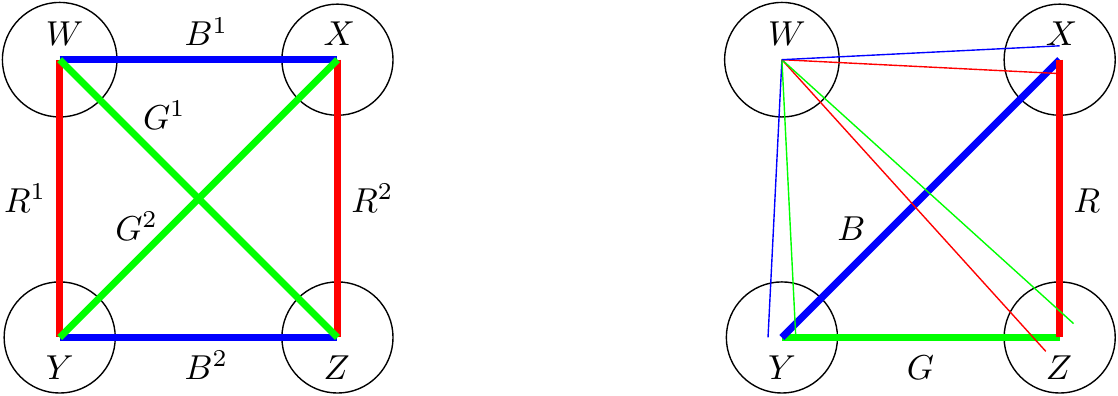}
\caption{Colorings of type (ii) and (iii) respectively.}\label{fig:types}
\end{figure}

\begin{proof}[Proof of Theorem \ref{thm:kr3tight}]
Let $K=K_n^3$ and let $u\in V(K)$.  If the link graph $K(u)$ is connected in any color, then we are done (as in the proof of Observation \ref{obs:basic}(iv)).  So by Lemma \ref{lem:3col}, there are two cases (type (ii) and type (iii)).  We will consider how the edges which do not contain $u$ (which have order 3) interact with the link graph $K(u)$ (which is a 2-uniform hypergraph).

\begin{claim}\label{clm:tight}
\begin{enumerate}
\item Let $H$ be a connected color $i$ subgraph in the link graph $K(u)$ for some $i\in [3]$.  If for all $c\in V(K(u))\setminus V(H)$, there exists $ab\in E(H)$ such that $abc$ is a color $i$ edge of $K$, then there is a monochromatic spanning tight component of $K$.

\item Suppose $H_1$ and $H_2$ are connected color $i$ subgraphs in the link graph $K(u)$ such that $\{V(H_1), V(H_2)\}$ forms a partition of $V(K(u))$.  If there exists $ab\in E(H_1)$ and $cd\in E(H_2)$ such that $abc, bcd$ are color $i$ edges of $K$, then we have a monochromatic spanning tight component. 
\end{enumerate}
\end{claim}

\begin{proof} Without loss of generality, suppose $i=1$ and say this color is red.  In either case, there is clearly a red tight walk between $u$ and any other vertex in $V(K(u))$.  So let $v,w \in V(K(u))$.  
\begin{enumerate}
\item If $v\in V(K(u))\setminus V(H)$, then let $e\in E(H)$ be the red edge guaranteed by the hypothesis and if $v\in V(H)$, then let $e\in E(H)$ such that $v\in e$.  Likewise, let $f$ be the red edge corresponding to $w$. Now any path between $e$ and $f$ in $H$, together with $u$ gives a red tight walk from $v$ to $w$ in $K$.

\item Let $e\in E(H)$ such that $v\in e$ and let $f\in E(H)$ such that $w\in f$.  If $e\in E(H_i)$ and $f\in E(H_i)$ for some $i\in [2]$, then there is a red tight walk from $v$ to $w$, so suppose $e\in E(H_1)$ and $f\in E(H_2)$.  Let $P$ be a path from $e$ to $ab$ in $H_1$ and let $Q$ be a path from $f$ to $cd$ in $H_2$.  The path $P$ together with $u$ gives us a red tight walk from $v$ to $ab$ (the last edge being $uab$), the edges $abc$ and $bcd$ give us a tight red walk from $v$ to $cd$ (the last edge being $bcd$), and finally $Q$ together with $u$ again, gives us a red tight walk from $bc$ to $w$.  So all together we have a red tight walk from $v$ to $w$.\qedhere
\end{enumerate}
\end{proof}

\tbf{Case 1} ($K(u)$ has a type (ii) coloring).

If for every vertex $v\in W\cup Y$, there exists a red $v,X,Z$-edge, then we are done by \ref{clm:tight}(i).  Likewise if for every vertex $v\in X\cup Z$, there exists a red $v,W,Y$-edge.  So suppose without loss of generality that there exists, say $w'\in W$ and $x'\in X$ such that no $w',X,Z$-edge is red and no $x',W,Y$-edge is red.  So every $w',x',Y\cup Z$-edge is either blue or green.  If all such edges are blue, then we are done by Claim \ref{clm:tight}(i), so suppose without loss of generality that there exists $y\in Y$ such that $w'x'y$ is green.  If any $w',x',Z$-edge is green, then we are done by \ref{clm:tight}(ii); so suppose every $w',x',Z$-edge is blue.  If for all $y\in Y$, there exists a blue $y,W,X$-edge, then we are done by Claim \ref{clm:tight}(i); so suppose there exists $y'\in Y$ such that no $y',W,X$-edge is blue.  We already know that no $x',y',W$-edge is red and we know that no $x',y', W$-edge is blue, so every $x',y', W$-edge is green.  If there exists a green $w',y',Z$-edge, then we are done by Claim \ref{clm:tight}(ii) and if there exists a blue $w',y',Z$-edge, then we are done by Claim \ref{clm:tight}(ii), so every $w',y',Z$-edge is red.  Since every $x',y',W$-edge is green, there exists $z'\in Z$ such that no $z',X,Y$-edge is green (by Claim \ref{clm:tight}(i)).  Finally, consider the edge $x'y'z'$.  We just showed that $x'y'z'$ cannot be green,  if $x'y'z'$ is red, then we are done by Claim \ref{clm:tight}(ii), and if $x'y'z'$ is blue, then we are done by Claim \ref{clm:tight}(ii).

\tbf{Case 2} ($K(u)$ has a type (iii) coloring).  Note that by Claim \ref{clm:tight}(i) we would be done if for all $x\in X$, there exists a green $x,Y,Z$-edge, or for all $y\in Y$, there exists a red $y,X,Z$-edge, or for all $z\in Z$, there exists a blue $z,X,Y$-edge; so suppose there exists $x'\in X$ such that no $x',Y,Z$-edge is green, $y'\in Y$ such that no $y', X,Z$-edge is red, and $z'\in Z$ such that no $z', X,Y$-edge is blue.  However, this is a contradiction as the edge $x'y'z'$ is either red, blue, or green.
\end{proof}

\subsection{Partitioning}

We now draw attention to the partition version of the problem raised by Fujita, Furuya, Gy\'arf\'as, and T\'oth \cite{FFGT1}.  Let $\tp_r^{c,\ell}(H)$ be the smallest integer $t$ such that in every $r$-coloring of the edges of $H$, there exists a set $\cT$ of at most $t$ monochromatic $(c,\ell)$-components $\mathcal{C}$ (that is, each $C\in \mathcal{C}$ is a component in $H_i$ for some $i\in [r]$) such that $\bigcup_{C\in \mathcal{C}}C=\binom{V(H)}{c}$ and $C\cap C'=\emptyset$ for all distinct $C,C'\in \mathcal{C}$.

\begin{problem}[{Fujita, Furuya, Gy\'arf\'as, T\'oth \cite[Problem 14]{FFGT1}}]
Is $\tp_6^1(K_n^3)=2$?
\end{problem}

\subsection{Large monochromatic subgraphs}

Finally, we raise the problem of determining the largest monochromatic $(c,\ell)$-component (recall that a $(c,\ell)$ component is a subset of $\binom{V(H)}{c}$) in an arbitrary $r$-coloring of $H$.   Let $\mc_r^{c,\ell}(H)$ be largest integer $m$ such that in every $r$-coloring of $H$, there is a monochromatic $(c,\ell)$ component of order at least $m$.

In this language Theorem \ref{thm:fur} says $\mc_r^{1}(G)\geq \frac{n}{(r-1)\alpha(G)}$. 
From Theorem \ref{thm:kr3tight}, we know $\mc_3^{1,2}(K_n^3)=n$. And from Theorem \ref{thm:k/3} we have have that if $1\leq \ell\leq c\leq k/3$, then $$\mc_r^{c,\ell}(K_n^k)\geq \frac{\binom{n}{c}}{\ceiling{\frac{r}{\floor{k/c}}}}.$$
So we ask the following question.

\begin{problem}
Determine $\mc_r^{c,\ell}(K_n^k)$.  In particular, determine $\mc_r^{1,2}(K_n^3)$ for $r\geq 4$.  
\end{problem}

\section{Further generalizations and strengthenings}\label{sec:further}

\subsection{Aharoni's proof for \texorpdfstring{$r=3$}{r=3}}\label{sec:aharoni}


Given a hypergraph $H$ and a matching $M$ in $H$, let $\rho(M)$ be the minimum size of a set of edges $F$ having the property that every edge in $M$ intersects some edge in $F$.  Let the \emph{matching width} of $H$, denoted $\mathrm{mw}(H)$, be the maximum value of $\rho(M)$ over all matchings $M$ in $H$.  Note that $\mathrm{mw}(H)$ is witnessed by a maximal matching.

\begin{observation}\label{mwobs}
Given a hypergraph of rank at most $r$ (that is, each edge has order at most $r$), $\mathrm{mw}(H)\leq \nu(H)\leq r\cdot \mathrm{mw}(H)$.
\end{observation}

\begin{proof}
Clearly $\rho(M)\leq |M|$ for all matchings $M$, which implies $\mathrm{mw}(H)\leq \nu(H)$.  Let $F$ be a set of edges $F$ which witnesses $\mathrm{mw}(H)$. Since $H$ has rank $r$, there are at most $r\cdot \mathrm{mw}(H)$ vertices spanned by $F$, and thus at most $r\cdot \mathrm{mw}(H)$ disjoint edges can intersect $F$.
\end{proof}

Aharoni's \cite{A} proof of Ryser's conjecture for the case $r=3$ implicitly shows that if $\tau(H')\leq (r-1)\mathrm{mw}(H')$ for all $(r-1)$-partite hypergraphs $H'$, then $\tau(H)\leq (r-1)\nu(H)$ for all $r$-partite hypergraphs $H$.  So we ask the following question.

\begin{problem}Let $r\geq 4$ and let $H'$ be an $(r-1)$-partite hypergraph. Is it true that $\tau(H')\leq (r-1)\mathrm{mw}(H')$?
\end{problem}

In the case $r=2$, this is trivial. In the case $r=3$ we have $\tau(H')=\nu(H')$ by K\"onig's theorem, and thus $\nu(H')\leq 2\mathrm{mw}(H')$ by Observation \ref{mwobs}.  We stress that we have no evidence one way or the other, but since a positive answer would imply Ryser's conjecture, we wouldn't be surprised if the answer is negative.\footnote{After submitting the paper, we asked this question on MathOverflow \url{https://mathoverflow.net/questions/372992/relationship-between-minimum-vertex-cover-and-matching-width} and indeed a construction providing a negative answer was given by Alex Ravsky.}

\subsection{Lov\'asz' conjecture}
Lov\'asz \cite{Lov75} made the following conjecture which would give an inductive proof of Ryser's conjecture.

\begin{conjecture}[Lov\'asz \cite{Lov75}]
In every $r$-partite hypergraph, there exists a set of at most $r-1$ vertices whose deletion decreases the matching number.

In the dual language (R2), in every $r$-colored graph, there exists a set of at most $r-1$ monochromatic components whose deletion decreases the independence number.
\end{conjecture}

Haxell, Narins, and Szab\'o \cite{HNS} proved this for all $3$-partite hypergraphs in which $\tau(H)=2\nu(H)$.  Aharoni, Bar\'at, Wanless \cite{ABW} proved a fractional version of this; that is, in every $r$-partite hypergraph $H$, there exists a set $S$ of at most $r-1$ vertices (which is contained in an edge) such that $\nu^*(H-S)\leq \nu^*(H)-1$.

\subsection{Monochromatic covers with extra restrictions}

Aharoni conjectured (see \cite{ABW}) the following strengthening of Ryser's conjecture (stated here in the dual language): in any $r$-coloring of $K_n$, there is a monochromatic $(r-1)$-cover in which either all the components have the same color, or there is a vertex which is contained in all the components.  Francetic, Herke, McKay, and Wanless disproved \cite[Theorem 3.1]{FHMW} this conjecture  by constructing  a 13-coloring of $K_n$ such that every color class has 13 components and every set of 12 monochromatic components which cover $V(K_n)$ has empty intersection.

In \cite{ABW} it was already noted that the following stronger conjecture is not true for all $r\geq 3$: in any $r$-coloring of $K_n$, there exists $i\in [r]$ such that there is a monochromatic $(r-1)$-cover in which all of the components have color $i$, or there is a monochromatic $(r-1)$-cover in which all of the components have colors in $[r]\setminus \{i\}$ and some vertex is contained in all of the components.

Note that our Conjecture \ref{con:rpart} is a weakening of this stronger conjecture.

\subsection{\texorpdfstring{$t$}{t}-intersecting hypergraphs/$t$-multicolored edges}

Bustamente and Stein \cite{BS}, and Kir\'aly and T\'othm\'er\'esz \cite{KT} independently studied the following problem.

\begin{problem}[\cite{BS}, \cite{KT}]
Let $r$ and $t$ be integers with $r\geq 2$ and $t\geq 1$ and let $H$ be an $r$-partite hypergraph in which every pair of edges intersects in at least $t$ vertices.  Determine an upper bound on $\tau(H)$.  

In the dual language (R2), we have a $r$-colored complete graph in which every pair of vertices is contained in components of at least $t$ different colors (equivalently, every edge gets $t$ different colors) and we are looking for the monochromatic cover with the smallest number of components.  
\end{problem}

The best known results are due to Bishnoi, Das, Morris, Szab\'o \cite{BDMS}; in fact, when $t>\frac{r}{3}$, their results are tight.

\begin{theorem}[Bishnoi, Das, Morris, Szab\'o \cite{BDMS}]\label{2wiset}
Let $t$ and $r$ be integers with $\frac{r}{3}< t\leq r$. If $H$ is an $r$-partite hypergraph in which every pair of edges has intersection size at least $t$, then $\tau(H)\leq \ceiling{\frac{r-t+1}{2}}$, and this is best possible.
\end{theorem}

Bishnoi, Das, Morris, Szab\'o \cite{BDMS} also studied a generalization where every set of $k$ edges intersect in at least $t$ vertices.  It turns out that their result is a generalization of Theorem \ref{kiraly} (which is the case $k\geq 3$, $t=1$ below). 

\begin{theorem}[Bishnoi, Das, Morris, Szab\'o \cite{BDMS}]\label{kwiset}
Let $k$, $r$, and $t$ be integers with $k\geq 3$, $r\geq 2$, and $t\geq 1$.  If $H$ is an $r$-partite hypergraph in which every set of $k$ edges has intersection size at least $t$, then $\tau(H)\leq \ceiling{\frac{r-t+1}{k}}=\floor{\frac{r-t}{k}}+1$, and this is best possible.  

In the dual language (R2), given an $r$-coloring of $K_n^k$ in which every set of $k$ vertices is contained in components of $t$ different colors (equivalently, every edge gets $t$ different colors), then there is a monochromatic $\ceiling{\frac{r-t+1}{k}}$-cover.  
\end{theorem}

What follows is an alternate proof of the upper bound in Theorem \ref{kwiset}.  In fact, this alternate proof shows that the monochromatic cover can be chosen to avoid any set of $t-1$ colors.

\begin{proof}
Let $S\subseteq [r]$ with $|S|=t-1$.  Every edge has a color from the set $[r]\setminus S$, which has size $r-(t-1)$ and thus by Theorem \ref{kiraly}, there is a monochromatic $\ceiling{\frac{r-(t-1)}{k}}$-cover where all of the components have a color from the set $[r]\setminus S$.
\end{proof}

Now we use Theorem \ref{kwiset} to give an alternate proof of the upper bound of Theorem \ref{2wiset} (with the additional property that the monochromatic cover can be chosen to avoid any set of $\ceiling{\frac{3t-r}{2}}-1$ colors).

\begin{proof}
For all distinct $x,y\in V(K)$, let $A_{xy}$ be the set of colors appearing on the edge $xy$.  We claim that in every set $\{x,y,z\}$ of three vertices there are at least $\ceiling{\frac{3t-r}{2}}$ colors which appear more than once.  To see this, we may assume $|A_{xy}\cap A_{xz}\cap A_{yz}|\leq \floor{\frac{3t-r}{2}}$, as otherwise we are done.  So, the number of colors appearing at least twice is 
\begin{align*}
&~~~~|A_{xy}\cap A_{xz}|+|A_{xy}\cap A_{yz}|+|A_{xz}\cap A_{yz}|-2|A_{xy}\cap A_{xz}\cap A_{yz}|\\
&=|A_{xy}|+|A_{xz}|+|A_{yz}|-|A_{xy}\cup A_{xz}\cup A_{yz}|-|A_{xy}\cap A_{xz}\cap A_{yz}|\\
&\geq 3t-r-\floor{\frac{3t-r}{2}}=\ceiling{\frac{3t-r}{2}}.
\end{align*}

Now form a complete 3-uniform hypergraph $H$ on $V(K)$ where for each edge $xyz$ and all $i\in [r]$, we color $xyz$ with any color $i$ if color $i$ appears more than once on $\{x,y,z\}$ in $K$.  So $H$ is an $r$-colored complete 3-uniform hypergraph where each edge gets at least $\ceiling{\frac{3t-r}{2}}$ colors.  Now applying Theorem \ref{kwiset} to $H$, we see that there is a monochromatic $\floor{\frac{r-\ceiling{\frac{3t-r}{2}}}{3}}+1$-cover of $H$ where $$\floor{\frac{r-\ceiling{\frac{3t-r}{2}}}{3}}+1\leq \floor{\frac{r-\frac{3t-r}{2}}{3}}+1=\ceiling{\frac{r-t+1}{2}}.$$  This monochromatic $\ceiling{\frac{r-t+1}{2}}$-cover of $H$ corresponds to a monochromatic $\ceiling{\frac{r-t+1}{2}}$-cover of $K$.
\end{proof}

\subsection{Linear hypergraphs/linear colorings}

We say that $H$ is a \emph{linear} hypergraph if every pair of vertices is contained in at most one edge.  Francetic, Herke, McKay, and Wanless \cite{FHMW} proved Ryser's conjecture for intersecting linear hypergraphs in the case $r\leq 9$.

\begin{theorem}
Let $r\geq 2$ and let $H$ be an $r$-partite intersecting linear hypergraph.  If $r\leq 9$, then $\tau(H)\leq r-1$.  

In the dual language (R2), this says that for all $2\leq r\leq 9$ and every $r$-coloring of a complete graph $K$ in which every monochromatic component is a clique and every edge gets exactly one color, there exists a monochromatic $(r-1)$-cover of $K$.
\end{theorem}

\subsection{Local colorings}

A local $r$-coloring of $G$ is an edge coloring in which each vertex is adjacent to edges of at most $r$ different colors.  We define $\tc_{r-loc}$ and $\tp_{r-loc}$ analogously to $\tc_r$ and $\tp_r$.

S\'ark\"ozy extended the methods of \cite{BD} to prove a strengthening of Theorem \ref{thm:HK} for local colorings.

\begin{theorem}[S\'ark\"ozy \cite{Sar80}]
For all $r\geq 1$ and $n\geq r^{2(r+2)}$, $\tp_{r-loc}(K_n)\leq r$ (in fact, the subgraphs can be chosen to be trees of radius at most 2 of distinct colors).  Furthermore, $\tc_{r-loc}(K_n)\geq r$ whenever a projective plane of order $r-1$ exists.
\end{theorem}

This raises the following question.

\begin{problem}
Is $\tc_{r-loc}(K_n)\geq r$ for all $r$? In particular, is $\tc_{7-loc}(K_n)\geq 7$?
\end{problem}

\subsection{Monochromatic covers with subgraphs of special types}

It was conjectured by Gy\'arf\'as \cite{Gy89} that in every $r$-coloring of a complete graph $K$, there is a partition into at most $r$ monochromatic paths, and there is an example to show that this is best possible.  For finite complete graphs, this is known for $r\leq 3$ \cite{Pok}.  Interestingly, this is known for all $r$ for countably infinite complete graphs \cite{R78} and uncountably infinite complete graphs \cite{Souk}.  

It was conjectured by Erd\H{o}s, Gy\'arf\'as, and Pyber \cite{EGP} that in every $r$-coloring of a complete graph $K$, there is a partition into at most $r$ monochromatic cycles.  For finite complete graphs, this is known for $r=2$ \cite{BT}, but it is not true for $r\geq 3$ \cite{Pok}.  It may still be true that there is a cover with at most $r$ monochromatic cycles or a partition into at most $r+1$ monochromatic cycles.

Since Ryser's conjecture is known to be true for $r=3$, it would be interesting to ask whether it remains true if the monochromatic cover must consist of certain types of graphs.  A \emph{branch vertex} in a tree is a vertex of degree at least 3.  A \emph{spider} is a tree with at most one branch vertex.  A \emph{broom} is a tree obtained by joining the center of a star to an endpoint of a path (equivalently, obtained by repeatedly subdividing one edge of a star).
It is a fun and simple exercise to prove that in every 2-coloring of $K_n$ there exists a monochromatic spanning spider (not a particular spider, but some member of the family of spiders).  A more challenging result due to Burr (the result is unpublished, but a very nice proof can be found in \cite{GySurv1}) is that in every 2-coloring of $K_n$ there is a monochromatic spanning broom (not a particular broom, but some member of the family of brooms). 

So we ask the following specific questions.  

\begin{problem}
In every 3-coloring of $K_n$ is there a monochromatic 2-cover consisting of spiders? consisting of brooms?
\end{problem}

\subsection{Monochromatic covers of Steiner triple systems}

A \emph{Steiner triple system (STS) of order $n$} is a 3-uniform hypergraph on $n$ vertices such that every pair of vertices is contained in exactly one edge.  It is well known that an STS of order $n$ exists if and only if $n\equiv 1,3\bmod 6$.  For a given $n$, let $\mathcal{STS}_n$ be the family of all STS of order $n$.  Gy\'arf\'as \cite{GySTS} proved that for all $H\in \mathcal{STS}_n$, $\mc_3(H)\geq \frac{2n}{3}+1$ and that this is best possible when $n\equiv 3\bmod 18$.  Gy\'arf\'as also proved that for all $r\geq 3$ and $H\in \mathcal{STS}_n$, $\mc_r(H)\geq \frac{n}{r-1}$ and this is best possible for infinitely many $n$ when $r-1\equiv 1,3 \bmod 6$ and an affine plane of order $r-1$ exists. DeBiasio and Tait \cite{DT} extended these results showing that, in particular, for almost all $H\in \mathcal{STS}_n$, $\mc_3(H)\geq (1-o(1))n$.  We propose the following problem.

\begin{problem}
Let $r\geq 2$ and let $n\equiv 1,3 \bmod 6$.  
\begin{enumerate}
\item Determine bounds on $\tc_r(H)$ which hold for all $H\in \mathcal{STS}_n$. 
\item Is $\tc_r(H)=\tc_r(H')$ for all $H, H'\in \mathcal{STS}_n$?
\end{enumerate}
\end{problem}

\subsection{Monochromatic covers of complete $r$-partite $r$-uniform hypergraphs}

Just as Kir\'aly was able to prove that a natural generalization of Ryser's conjecture holds for complete $k$-uniform hypergraphs when $k\geq 3$ (Theorem \ref{kiraly}), Gy\'arf\'as and Kir\'aly \cite{GK} proved that a natural generalization of Conjecture \ref{con:GL} holds for complete $k$-partite $k$-uniform hypergraphs when $k\geq 3$; that is, they showed that if $k\geq 3$ and $G$ is a complete $k$-partite $k$-uniform hypergraph, then $\tc_r(G)=r$.

\subsection{Monochromatic covers of bounded diameter graphs}

Tonoyan \cite{Ton} proved that for all $r\geq 1$, $d\geq D\geq 1$, $n\geq 2$, there exists $N$ such that in every $r$-coloring of the edges of every graph on at least $N$ vertices with diameter at most $D$ contains a monochromatic subgraph on at least $n$ vertices of diameter at most $d$.

It would be interesting to consider host graphs of bounded diameter and ask whether it is possible to cover them with monochromatic subgraphs of bounded diameter.

\begin{problem}
Let $r,D,\delta$ be positive integers. 
\begin{enumerate}
\item Determine an upper bound on $\tc_r(G)$ which holds for all graphs $G$ of diameter at most $D$. (Even the case $r=2=D$ is open).
\item Determine an upper bound on $\dc^\delta_r(G)$ which holds for all graphs $G$ of diameter at most $D$.
\end{enumerate}
\end{problem}

\subsection{Monochromatic covers of large minimum degree graphs}

Bal and DeBiasio made the following conjecture about monochromatic covers of graphs with large minimum degree and proved a weaker result.

\begin{conjecture}[Bal, DeBiasio \cite{BD}]\label{BD:con1}
For all integers $r\geq 1$, if $G$ is a graph on $n$ vertices with $\delta(G)> \frac{r(n-r-1)}{r+1}$, then $\tc_r(G)\leq r$ (possibly $\tp_r(G)\leq r$).
\end{conjecture}

\begin{theorem}[Bal, DeBiasio \cite{BD}]
For all integers $r\geq 1$ there exists $n_0$ such that if $G$ is a graph on $n\geq n_0$ vertices with $\delta(G)\geq (1-\frac{1}{er!})n$, then $\tp_r(G)\leq r$.
\end{theorem}

Gir\~ao, Letzter, and Sahasrabudhe \cite{GLS} proved the partition version of Conjecture \ref{BD:con1} for $r=2$.  

Buci\'c, Kor\'andi, and Sudakov \cite{BKS} proved the following theorem which solved a different conjecture from \cite{BD}.

\begin{theorem}[Buci\'c, Kor\'andi, Sudakov \cite{BKS}]
For all integers $r\geq 1$, if $G$ is a graph on $n$ vertices with $\delta(G)\geq (1-\frac{1}{2^r})n$, then $G$ has a monochromatic $r$-cover consisting of components of different colors.
\end{theorem}

\subsection{Monochromatic covers of random graphs}

Bal and DeBiasio proved the following results about monochromatic covers of random graphs.  

\begin{theorem}[Bal, DeBiasio \cite{BD}]~ For all integers $r\geq 2$,
\begin{enumerate}
\item if $p\ll \left(\frac{r\log n}{n}\right)^{1/r}$, then a.a.s.\ $\tc_r(G_{n,p})\to \infty$.

\item if $p\gg \left(\frac{r\log n}{n}\right)^{1/(r+1)}$, then a.a.s.\ $\tc_r(G_{n,p})\leq r^2$.

\item if $p\gg \left(\frac{r\log n}{n}\right)^{1/3}$, then a.a.s.\ $\tp_r(G_{n,p})\leq 2$.
\end{enumerate}
\end{theorem}

Kohayakawa, Mota, and Schacht \cite{KMS} proved that if $p\gg \left(\frac{r\log n}{n}\right)^{1/2}$, then a.a.s.\ $\tp_r(G_{n,p})\leq 2$ and proved that for $r\geq 3$ if $p\ll  \left(\frac{r\log n}{n}\right)^{1/(r+1)}$, then a.a.s.\ $\tc_r(G_{n,p})>r$ which disproved a conjecture from \cite{BD}.  Recently Buci\'c, Kor\'andi, and Sudakov extended these results, disproving a conjecture from \cite{KMS} in such a way which drastically reshaped the known picture.  Many problems remain regarding the sharpening of these results.

\begin{theorem}[Buci\'c, Kor\'andi, Sudakov \cite{BKS}]~ For all integers $r\geq 2$, there exist constants $C,c>0$ such that 
\begin{enumerate}
\item if $p< \left(\frac{c\log n}{n}\right)^{\sqrt{r}/2^{r-2}}$, then a.a.s.\ $\tc_r(G_{n,p})>r$.

\item if $p> \left(\frac{C\log n}{n}\right)^{1/r}$, then a.a.s.\ $\tc_r(G_{n,p})\leq r^2$.

\item if $p>\left(\frac{C\log n}{n}\right)^{1/2^r}$, then a.a.s.\ $\tc_r(G_{n,p})\leq r$. 
\end{enumerate}
\end{theorem}

\section{Acknowledgements}

This research began as part of the NSA funded REU program SUMSRI held at Miami University during summer 2018.  We thank Delaney Aydel, Siya Chauhan, Teresa Flaitz, and Cody Newman for their contributions during the program.  

We also thank Andr\'as Gy\'arf\'as and G\'abor S\'ark\"ozy for their help in obtaining \cite{T1}, Austin Moore for his help with the program resulting in Table \ref{tab:r6}, Bob Krueger and Deepak Bal for their discussions regarding Problem \ref{prob:Zrd}, and Zolt\'an Kir\'aly for pointing out the reference \cite{GK}.

Finally, we sincerely thank the referees for their careful reading of the paper.

\bibliographystyle{amsplain}
\bibliography{bibliography}

\newpage

\section*{Appendix A: General properties of a minimal counterexample}\label{sec:appendix}

In this section we collect a a few more observations about a hypothetical minimal counterexample to Ryser's conjecture.

Recall that to prove $\tc_r(G)\leq (r-1)\alpha(G)$ it suffices to consider $r$-colorings of multigraphs in which every monochromatic component is a clique.  In such an $r$-colored multigraph, we call an edge $e$ of color $i$ and multiplicity 1 an \emph{essential edge of color $i$}.

\begin{theorem}
Suppose there exists positive integers $r$ and $n$, a multigraph $G$ on $n$ vertices with $\alpha:=\alpha(G)$, and an $r$-coloring $c:E(G)\to [r]$ (in which every monochromatic component is a clique) such that $G$ cannot be covered by at most $(r-1)\alpha$ monochromatic components.  Choose such a graph and a coloring which (i) minimizes $r$, (ii) minimizes $\alpha$, (iii) minimizes $n$, (iv) minimizes $e(G)$. Then $G$ has the following properties:
\begin{enumerate}
\item\label{+1} $\tc_r(G)=(r-1)\alpha+1$
\item Every color class contains at least $(r-1)\alpha+1$ components.
\item For all $i\in [r]$, every component of color $i$ contains an essential edge of color $i$.
\item Every vertex is incident with an edge of every color. In particular, every monochromatic component has at least 2 vertices. 
\item Every set of $r$ components intersect in at most one vertex.
\item If $\alpha=1$ and $0\leq t\leq r-1$, then every set of $r-t$ components intersect in at most $t!$ vertices.
\item For all $S\subseteq [r]$ with $s:=|S|\geq 2$, $\alpha(G_S)\geq \frac{(r-1)\alpha+1}{s-1}$.  In particular, for all distinct $i,j\in [r]$, we have $\alpha(G_{i,j})\geq (r-1)\alpha(G)+1$.
\end{enumerate}
\end{theorem}

\begin{proof}~
\begin{enumerate}
\item 
Let $v$ be any vertex.  Consider the graph $G'$ obtained by removing all of the vertices in monochromatic components containing $v$.  Since $\alpha(G')\leq \alpha(G)-1$, and $G$ is a minimal counterexample, we have $\tc_r(G')\leq (r-1)(\alpha-1)$ and thus $(r-1)\alpha+1\leq \tc_r(G)\leq (r-1)(\alpha-1)+r=(r-1)\alpha+1$.
\item 
This follows since every color class is a monochromatic cover and we are assuming that we are in a minimal counterexample.
\item 
If not, we may remove the edges of color $i$ corresponding to this component and call the resulting graph $G'$.  Note that $e(G')<e(G)$, but $\alpha(G')=\alpha(G)$, so by minimality, we have $\tc_r(G)\leq \tc_r(G')\leq (r-1)\alpha$.
\item 
If there exists $v$ such that $v$ is incident with edges of at most $r-1$ colors, then consider the graph $G'$ obtained by removing all of the vertices in monochromatic components containing $v$.  Since $\alpha(G')\leq \alpha(G)-1$, and $G$ is a minimal counterexample, we have $\tc_r(G')\leq (r-1)(\alpha-1)$ and thus $\tc_r(G)\leq (r-1)(\alpha-1)+(r-1)=(r-1)\alpha$, a contradiction.

\item 
If the components are not of distinct colors, then their intersection is empty; so suppose the colors are distinct.  If there are at least two vertices $u,v$ in the intersection, replace them with a vertex $w$ such that for all $x\notin\{u,v\}$, $wx$ is an edge colored with every color appearing on $ux$ or on $vx$; call this new graph $G'$.  A covering of $G'$ gives a covering of $G$ since any component in the covering of $G'$ which contains $w$, contains $u$ and $v$ in $G$.
\item 
The previous item provides the base case, so suppose $t\geq 1$.  Let $H_1, \dots, H_{r-t}$ be a set of $r-t$ components of distinct colors and suppose for contradiction that there are at least $t!+1$ vertices in the intersection $A$.  Let $v$ be a vertex which is not in any of $H_1, \dots, H_{r-t}$ and note that every edge from $v$ to $A$ has one of $t$ different colors and thus by pigeonhole, there is a fixed color $i$ such that $v$ sends at least $\ceiling{\frac{t!+1}{t}}=(t-1)!+1$ edges of color $i$ to $A$.  However, now we have $r-(t-1)$ components which intersect in at least $(t-1)!+1$ vertices, a contradiction.

\item 
By (vi), we have $(r-1)\alpha(G)+1=\tc_r(G)\leq \tc_s(G_S)\leq (s-1)\alpha(G_S)$ and thus $\alpha(G_S)\geq \frac{(r-1)\alpha(G)+1}{s-1}$.
\end{enumerate}
\end{proof}

\end{document}